\documentclass[12pt]{amsart}
\usepackage[colorlinks=true, pdfstartview=FitV, linkcolor=blue, citecolor=blue, urlcolor=blue]{hyperref}
\usepackage{amssymb}

\newcommand{\be}{\begin{enumerate}}
\newcommand{\ee}{\end{enumerate}}
\newcommand{\quot}[2]{#1/\!\!/#2}
\DeclareMathOperator{\gr}{gr}
\newcommand{\C}{\mathbb {C}}
\newcommand{\NN}{\mathcal N}
\newcommand{\OO}{\mathcal O}
\newcommand{\OOO}{\mathcal O}
\renewcommand{\SS}{\mathcal S}
\newcommand{\TT}{\mathcal T}
\newcommand{\N}{\mathbb N}
\newcommand{\Z}{\mathbb Z}
\newcommand{\Q}{\mathbb Q}
\DeclareMathOperator{\SL}{SL}
\DeclareMathOperator{\SLtwo}{SL_2}
\DeclareMathOperator{\SLthree}{SL_3}
\DeclareMathOperator{\GL}{GL}
\DeclareMathOperator{\PSL}{{PSL}}

\DeclareMathOperator{\PSp}{{PSp}}
\DeclareMathOperator{\Sp}{{Sp}}
\DeclareMathOperator{\SO}{{SO}}
\DeclareMathOperator{\Orth}{{O}}
\DeclareMathOperator{\PSO}{{PSO}}
\DeclareMathOperator{\Jac}{{Jac}}
\DeclareMathOperator{\Hom}{{Hom}}

\DeclareMathOperator{\Cent}{{Cent}}
\DeclareMathOperator{\Lie}{{Lie}}
\DeclareMathOperator{\gen}{gen}
\newcommand{\Cst}{{\C^*}}
\DeclareMathOperator{\Ad}{{Ad}}
\newcommand{\inv}{^{-1}}
\newcommand{\lie}[1]{{\mathfrak{#1}}}
\newcommand{\lieg}{{\lie g}}
\newcommand{\lieh}{{\lie h}}
\DeclareMathOperator{\sll}{{\mathfrak{sl}}}
\DeclareMathOperator{\sltwo}{{\mathfrak{sl}_{2}}}
\DeclareMathOperator{\hh}{{\mathfrak{h}}}
\DeclareMathOperator{\mm}{{\mathfrak{m}}}
\DeclareMathOperator{\nn}{{\mathfrak{n}}}
\DeclareMathOperator{\bb}{{\mathfrak{b}}}
\renewcommand{\phi}{{\varphi}}
\newcommand{\twomatrix}[4]{{\left(\smallmatrix #1 & #2 \\ #3 & #4\endsmallmatrix\right)}}
\newcommand{\I}{\mathcal I}
\DeclareMathOperator{\tr}{{tr}}

\renewcommand{\AA}{\mathsf{A}}
\newcommand{\CC}{\mathsf{C}}
\newcommand{\BB}{\mathsf{B}}
\newcommand{\DD}{\mathsf{D}}
\newcommand{\EE}{\mathsf{E}}
\newcommand{\FF}{\mathsf{F}}
\newcommand{\GG}{\mathsf{G}}
\newcommand{\Ffour}{\FF_{4}}
\newcommand{\Gtwo}{\GG_{2}}

\newcommand{\Cone}{\mathcal C}

\newcommand{\Bfour}{\BB_{4}}

\DeclareMathOperator{\codim}{{codim}}
\DeclareMathOperator{\Spin}{{Spin}}

\DeclareMathOperator{\rk}{{rank}}
\DeclareMathOperator{\Cov}{{Cov}}
\DeclareMathOperator{\id}{Id}
\DeclareMathOperator{\Sym}{Sym}
\DeclareMathOperator{\htt}{ht}

\newcommand{\margin}[1]{}
\newcommand{\lab}[1]{\label{#1}}

\renewcommand{\epsilon}{\varepsilon}
\newcommand{\eps}{\varepsilon}
\newcommand{\simto}{\overset{\sim}{\longrightarrow}}

\newcommand{\phm}{\phantom{-}}

\newcommand{\name}[1]{\textsc{#1\/}}
\numberwithin{equation}{subsection}

\newtheorem{theorem}[subsection]{Theorem}
\newtheorem{lemma}[subsection]{Lemma}
\newtheorem{proposition}[subsection]{Proposition}
\newtheorem{corollary}[subsection]{Corollary}
\newtheorem{conjecture}[subsection]{Conjecture}
\newtheorem{critA}[subsection]{Criterion}

\theoremstyle{definition}

\newtheorem{definition}[subsection]{Definition}
\newtheorem{example}[subsection]{Example}
\newtheorem{examples}[subsection]{Examples}
\theoremstyle{remark}

\newtheorem{remark}[subsection]{Remark}

\begin{document}
\title[Representations with a Reduced Null Cone]{Representations with a Reduced Null Cone} 
\author{Hanspeter Kraft}
 \address {Mathematisches Institut der
Universit\"at Basel,\newline\indent
Rheinsprung 21, CH-4051 Basel, Switzerland}
\email{Hanspeter.Kraft@unibas.ch}
\author{Gerald W. Schwarz}
\address{Department of Mathematics,
Brandeis University,\newline\indent
Waltham, MA 02454-9110, USA}
\email{schwarz@brandeis.edu} 
\date{November 6, 2011}

\begin{abstract}
Let $G$ be a complex reductive group and $V$ a $G$-module. Let $\pi\colon V\to\quot VG$ be the quotient morphism and set $\NN(V)=\pi\inv(\pi(0))$.  We consider the following question.   Is the null cone $\NN(V)$   reduced, i.e., is the ideal of $\NN(V)$ generated by $G$-invariant polynomials?   We have complete results when $G$ is $\SL_2$,  $\SL_3$ or a simple group of adjoint type, and also when $G$ is semisimple of adjoint type and the $G$-module $V$ is irreducible.
 \end{abstract}
\maketitle
{\footnotesize
\tableofcontents}

\section{Introduction}
Let $G$ be a reductive complex algebraic group and    let $V$ be a finite-dimensional $G$-module. Let $\pi \colon V\to\quot VG$ be the categorical quotient morphism given by the $G$-invariant functions on $V$, and let 
$$
\NN:=\pi\inv(\pi(0))=\{v\in V\mid \overline{Gv}\ni 0\}\subset V
$$ 
be the {\it null cone}. 
We say that $V$ is \emph{coreduced} if $\NN$ is reduced. This means that the ideal $I(\NN)\subset \OOO(V)$ of the set $\NN$ is generated by the invariant functions $\mm_{0}:=\OOO(V)^G\cap I(\NN)$, the homogeneous maximal ideal of $\OOO(V)^{G}$, and so $I(\NN)=I_{G}(\NN):=\mm_{0}\OOO(V)$, where we use $\OOO(X)$ to denote the regular functions on a variety $X$.
If it is important to specify the group or representation involved, we will use notation such as $(V,G)$, $\NN(V)$, $\NN_{G}(V)$, etc. 

We say that $V$ is {\it strongly coreduced\/} if every fiber of $\pi$ is reduced. We can reformulate this in terms of slice representations. Let $Gv$ be a closed orbit. Then $G_v$ is reductive and we have a splitting of $G_v$-modules $V=T_v(Gv)\oplus N_v$. Then $(N_v,G_v)$ is the \emph{slice representation of $G_v$ at $v$\/}. We show that the fiber $\pi\inv(\pi(v))$ is reduced if and only if $(N_v,G_v)$ is coreduced (Remark \ref{rem:reducedfiber}). Hence $V$ is strongly coreduced if and only if every slice representation of $V$ is coreduced.
 
 Recall that $V$ is \emph{cofree\/} if $\OOO(V)$ is a free module over $\OOO(V)^G$. Equivalently, $\pi\colon V\to\quot VG$ is flat. A main difficulty in our work is that, a priori, $V$ may be coreduced but $\pi$ may have a nonreduced fiber $F\neq\NN$.  This cannot happen in the cofree case (Proposition \ref{prop:slice.cofree}). We conjecture  that this is true in general:
 
 \begin{conjecture}
 A coreduced $G$-module is strongly coreduced.
 \end{conjecture}

In the cofree case  the associated cone to any fiber $F$ (see \cite{BoKr1979Uber-Bahnen-und-de} or \cite[II.4.2]{Kr1984Geometrische-Metho}) is the null cone. From this one can immediately see that $\NN$ reduced implies that $F$ is reduced. There is another case in which the associated cone of $F$ is $\NN$: the case in which the isotropy group $H$ of the closed orbit $Gv\subset F$ has the same rank as $G$, i.e.,  contains a maximal torus $T$ of $G$. Thus if the slice representation of $H$ at $v$ is not coreduced, then neither is $(V,G)$ (Proposition \ref{prop:slice.zero.weight}). For an irreducible representation $V$ of $G$, having $V^T\neq 0$ means that the weights of $V$ are in the root lattice; equivalently, the center of $G$ acts trivially on $V$. Hence we  have a representation of the adjoint group $G/Z(G)$. This explains why many of our results require that the group be adjoint, or that at least one of the irreducible components of our representation contains a zero weight vector.
 
 Here is a summary of the contents of this paper. In \S \ref{sec:preliminary} we present elementary results and determine the coreduced representations of tori (Proposition \ref{prop:tori}). In section \ref{sec:covariants} we show how to use covariants to prove that a null cone is not reduced and as an application we determine the coreduced representations of $\SL_2$ (Theorem \ref{thm:SLtwo}). In \S \ref{sec:cofreerep} we show that every cofree irreducible representation of a simple algebraic group is coreduced (Theorem \ref{thm:cofree}) and that, sort of conversely, every irreducible representation of a simple group which is strongly coreduced is cofree (Theorem \ref{thm:stronglycoreduced}). In \S \ref{sec:slices} we consider modules $V$ with $V^T\neq 0$, $T$ a maximal torus of $G$. We develop methods (based on weight multiplicities) to show that a slice representation at a zero weight vector is not coreduced (we say that $V$ has a \emph{bad toral slice\/}). We then show that $V$ has a bad toral slice if all the roots of $G$ appear in $V$ with multiplicity two or more (Proposition \ref{prop:rootsmult2}). 
 
In \S \ref{sec:cored-exceptional} we apply  our techniques to find the maximal coreduced representations of the adjoint exceptional  groups (Theorem \ref{thm:coredexceptional}). The case of $\Ffour$ is rather complicated and needs some heavy computations (see \ref{AppA}).
In \S \ref{sec:cored-classical} we do the same thing for the classical adjoint groups (Theorem \ref{thm:cored.simple}), and  in \S \ref{sec:irreducible-semisimple} we determine the irreducible coreduced representations of semisimple adjoint groups (Theorem \ref{thm:adjoint-semisimple}). This is not straightforward, e.g., the representation $(\C^7\otimes\C^7,\Gtwo\times\Gtwo)$ is not coreduced,  but showing this is  difficult (see \ref{AppB}). 

In \S \ref{sec:CIT} we   show that, essentially, the classical representations of the classical groups are coreduced (with a restriction for  $\SO_n$). This is a bit surprising, since these representations are often far from cofree. In \S \ref{sec:sl3} we  classify the coreduced representations of $\SL_3$ (not just $\PSL_3$). To do this, we need to develop some techniques for finding irreducible components of null cones (see \S \ref{sec:comp-nullcone}). These techniques should be useful in other contexts.
\par\smallskip\noindent
{\small
{\bf Acknowledgement: } We 
thank Michel Brion and John Stembridge  for helpful discussions and remarks, and Jan Draisma for his computations. 
}

\bigskip
\section{Preliminaries and Elementary Results}\lab{sec:preliminary}
We begin with some positive results. Let $G$ be a reductive group and $V$ a finite-dimensional $G$-module.

\begin{proposition}\lab{prop:coreduced}
Suppose that $G$ is connected semisimple and that
$V$ satisfies one of the following conditions. 
\begin{enumerate}
\item $\dim\quot VG\leq 1$;
\item $V=\Ad G$.
\end{enumerate}
Then $V$ is  coreduced.
\end{proposition}

\begin{proof} If $\dim\quot VG=0$, then $\NN=V$ is reduced.
If $\dim\quot VG=1$, then $\OO(V)^G$ is generated by a homogeneous irreducible polynomial $f$ and its zero set $\NN$ is reduced. If $V$ is the adjoint representation of $G$, then $\NN$ is irreducible of codimension equal to the rank $\ell$ of $G$. Since $\NN$ is defined by $\ell$ homogeneous invariants  and the rank of $d\pi$ is $\ell$ on an open dense subset of $\NN$, it follows that $\NN$ is reduced and even normal (\cite{Ko1963Lie-group-represen}, cf. Proposition~\ref{prop:Serre} below).
\end{proof}

\begin{example}
Suppose that $G$ is finite and acts nontrivially on $V$. Then $\NN=\{0\}$ (as set) is not reduced since not all the coordinate functions can be $G$-invariant.
\end{example}

Let $V$ be a $G$-module. Then $\quot VG$ parameterizes the closed $G$-orbits in $V$. Let $Gv$ be a closed orbit and let $N_v$ be the slice representation of $G_v$. We say that $Gv$ is a \emph{principal orbit\/} and that $G_v$ is a \emph{principal isotropy group\/} if $\OOO(N_v)^{G_v}=\OOO(N_v^{G_v})$. In other words, $N_v$ is the sum of a trivial representation and a representation $N_v'$ with $\OOO(N_v')^{G_v}=\C$. We say that $V$ is \emph{stable} if  $N_v'=(0)$; equivalently, there is an open dense subset of $V$ consisting of closed orbits. If $G$ is semisimple and there is a nonempty open subset of points with reductive isotropy group, then $V$ is stable. In particular, if the general point in $V$ has finite isotropy group, then the representation is stable with finite principal isotropy groups

Our example above generalizes to the following
\begin{remark}\lab{rem:non-connected}
Let $V$ be a $G$-module where $G/G^0\neq\{e\}$. If $G/G^{0}$ acts nontrivially on the quotient $\quot VG^{0}$, then $V$ is not coreduced. Note that, for example, $G/G^{0}$ acts nontrivially if the   principal isotropy group of $(V,G)$ is trivial.
 \end{remark}

\begin{proposition}\lab{prop:non-connected}
Assume that $(V,G^0)$ is not coreduced Then $(V,G)$ is not coreduced.
\end{proposition}

\begin{proof}
The null cones for $G^{0}$ and  $G$ are the same (as sets). There is an $f \in I(\NN)$ which is not in $ I_{G^0}(\NN)$. Hence $f\not\in I_G(\NN)$ and $(V,G)$ is not coreduced.
\end{proof}

For the next three more technical results we have to generalize the definition of coreducedness to pointed $G$-varieties.

\begin{definition}
A {\it pointed $G$-variety\/} is a pair $(Y,y_{0})$ where $Y$ is an affine $G$-variety and $y_{0}$ a fixed point. A pointed $G$-variety $(Y,y_{0})$ is {\it coreduced\/} if the fiber $\pi^{-1}(\pi(y_{0}))$ is reduced where $\pi\colon Y \to \quot YG$ is the quotient morphism.
\end{definition}
 
\begin{lemma}\lab{lem:invariant.ideal}
Let $(X,x_{0})$ be a pointed $G$-variety and $Y\subset X$ a closed $G$-stable subvariety containing  $x_{0}$. Assume that the ideal $I(Y)$ of $Y$ is generated by $G$-invariant functions. Then $(X,x_{0})$ is coreduced if and only if $(Y,x_{0})$ is coreduced.
\end{lemma}

\begin{proof} Let $\mm \subset \OOO(X)^{G}$ be the maximal ideal of $\pi(x_{0})$ where $\pi\colon X \to \quot XG$ is the quotient morphism, and let $\nn\subset \OOO(Y)^{G}$ denote the image of $\mm$. By assumption, the ideal $\mm\OOO(X)$ contains $I(Y)$ and so $\OOO(Y)/(\nn\OOO(Y)) \simeq \OOO(X)/(I(Y) + \mm\OOO(X))=\OOO(X)/(\mm\OOO(X))$.
\end{proof}

\begin{example}\lab{ex:invariant.hypersurface}
Let $V$ be a $G$-module  Denote by $\theta_{n}$ the $n$-dimensional trivial representation 
and let $F \subset V\oplus\theta_{n}$ be a $G$-stable hypersurface containing $0$. If $G$ is semisimple, then $F$ is defined by a $G$-invariant polynomial. Hence $(F,0)$ is coreduced if and only if $V$ is coreduced.
\end{example}

\begin{lemma}\lab{lem:HnormalizingG}
Let $(X,x_{0})$ be a pointed $G$-variety. Let $H$ be a reductive group acting on $X$ such that $G$ sends $H$-orbits to $H$-orbits. Assume that  every $G$-invariant function on $X$ is $H$-invariant. If $(X,x_{0})$ is coreduced with respect to $G$, then so is $(\quot XH, \pi(x_{0}))$.
\end{lemma}
 
\begin{proof}
Put $Y:=\quot XH$. Then $G$ acts on $Y$, because $G$ preserves the $H$-invariant functions $\OOO(X)^{H}\subset\OOO(X)$. Suppose that $f$ is an element of $\OO(Y)$ which vanishes on the null fiber $\NN_{G}(Y,\pi_{H}(x_{0}))$. By assumption, $\NN_{G}(X,x_{0}) = \pi_{H}^{-1}(\NN_{G}(Y,\pi_{H}(x_{0})))$ and so  $f\circ\pi_{H}$ vanishes on $\NN_{G}(X,x_{0})$. Hence $f \circ \pi_{H} = \sum_{i}a_{i}b_{i}$ where the $a_{i}$ are $G$-invariant and vanish at $x_{0}$. Since $f \circ \pi_{H}$ is $H$-invariant, we may average the $b_{i}$  over $H$ and still have $f \circ \pi_{H} = \sum_{i}a_{i}b_{i}$. But then $a_{i}=\bar a_{i}\circ\pi_{H}$ and $b_{i}= \bar b_{i}\circ\pi_{H}$ for unique $\bar a_{i} \in \OO(Y)^{G}$ and $\bar b_{i} \in \OO(Y)$. Thus $f  = \sum_{i}\bar a_{i}\bar b_{i}$ and so  $(Y, \pi_{H}(x_{0}))$ is coreduced.
\end{proof}

\begin{example}\lab{ex:HnormalG}
If $(V,G)$ is a coreduced representation and $H \subset G$ a closed normal subgroup, 
then $(\quot VH, \pi_{H}(0))$ is coreduced (with respect to $G/H$).
\end{example}

\begin{example}\lab{ex:quotientSOn}
 Let $(X,x_0)$ be a pointed $G$-variety, let $W$ be a $G$-module of dimension $n$ and let $H=\SO_n$ acting as usual on $\C^n$. Assume that $G\to\GL(W)$ has image in $\SL(W)$. Consider the pointed $(G\times H)$-variety $(Y,y_0):=(X\times (W\otimes\C^n),(x_0,0))$. We claim that if $(Y,y_0)$ is a coreduced $(G\times H)$-variety, 
then $(X\times S^2(W^{*}),(x_0,0))$ is a coreduced  $G$-variety. 
 
By classical invariant theory, the generators of $(n\C^n,\SO_n)$ are the inner product invariants $f_{ij}$ of the copies of $\C^n$ together with the determinant $d$. The relations are generated by the equality $d^2=\det(f_{ij})$. Identifying $n\C^n$ with $W\otimes\C^n$ we see that the quadratic invariants transform by the representation $S^2(W^*)$ of $G$, the determinant $d$ transforms by $\bigwedge^n(W^*)=\theta_1$ and the relation is $G$-invariant. Now applying   Lemmas  \ref{lem:HnormalizingG} and \ref{lem:invariant.ideal}  gives the claim. 
\end{example}

\begin{lemma}\lab{lem:retraction}
Let  $(Y,y_0)$ be a pointed $G$-variety and $Z\subset Y$ a closed $G$-stable subvariety containing $y_{0}$. Suppose that there is a $G$-equivariant retraction  $p\colon (Y,y_0)\to (Z,y_0)$. If $(Y,y_0)$ is coreduced, then so is $(Z,y_0)$.
\end{lemma}

\begin{proof}
Clearly, if $y$ is in the null cone of $Y$, then $p(y)$ is in the null cone of $Z$.  Thus if $f\in\OO(Z)$ vanishes on the null cone of $Z$ then $\tilde f:=p^*f\in\OOO(Y)$ vanishes on the null cone of $Y$. By hypothesis we have that $\tilde f=\sum_i a_i b_i$ where the $a_i$ are invariants vanishing at $y_{0}$. Restricting to $Z$ we get a similar sum for $f$. Hence $Z$ is coreduced. 
\end{proof}

\begin{examples}\lab{ex:retraction}
\be
\item If $(G,V_{1}\oplus V_{2})$ is a coreduced representation, then so is $(G,V_{i})$, $i=1,2$.
\item If $(G,V)$ is a coreduced representation and $H\subset G$ a closed normal subgroup,  then $(G, V^{H})$ is also coreduced.
\item Let $V_i$ be a $G_i$ module, $i=1$, $2$. Then  $(G_{1}\times G_{2},V_{1}\oplus V_{2})$ is coreduced if and only if both $(G_{1},V_{1})$ and $(G_{2},V_{2})$ are coreduced.  Here we use that $\NN(V_1\oplus V_2)=\NN(V_1)\times\NN(V_2)$. 
\ee
\end{examples}

We finish this section with the case of tori which is quite easy. We will then see in section~\ref{sec:slices} that this case can be applied to representations containing zero weights.

\begin{proposition}\lab{prop:tori}
Let $V$ be a $T$-module   where $T$ is a torus. Let $\alpha_1,\dots,\alpha_n$ be the nonzero weights of $V$. 
Then $V$ is coreduced if and only if the solutions of 
$$
\sum_i m_i\alpha_i=0, \quad m_i\in\N,
$$ 
are generated by solutions where the $m_i$ are zero or one. 
\end{proposition}
It is well-known that the monoid of relations $\sum_i m_i\alpha_i=0$, $m_i\in\N$ is generated by the  {\it indecomposable\/} relations, i.e., by relations which cannot be written as a sum of two nontrivial relations. {\it So a necessary and sufficient condition for coreducedness is that the indecomposable relations $\sum_i n_i\alpha_i=0$, $n_i\in\N$, satisfy $n_{i}=0$ or $1$}.

\begin{proof} We may assume that $V^T=0$. Let $x_1,\dots,x_n$ be a weight basis of $V^{*}$ corresponding to the $\alpha_i$.
Suppose that there is an indecomposable relation where, say, $m_1\geq 2$. Then the monomial $x_1x_2^{m_2}\dots x_n^{m_n}$ vanishes on the null cone, but it is not in the ideal of the invariants. Hence our condition is necessary. 

On the other hand, suppose that the indecomposable relations are of the desired form. Now any polynomial vanishing on $\NN(V)$ is a sum of monomials with this property, and a monomial $p$ vanishing  on $\NN(V)$ has a power which is divisible by an invariant monomial $q$ without multiple factors. But then $p$ is divisible by $q$ and so $\NN(V)$ is reduced.
\end{proof}

\begin{corollary}\lab{cor:cstar.reps}
Let $T=\C^*$. Then $\NN(V)$ is reduced if and only if $\OO(V)^T=\C$ or the nonzero weights of $V$ are $\pm k$ for a fixed $k\in\N$.
\end{corollary}

\bigskip
\section{The Method of Covariants}\lab{sec:covariants}
In this section we explain how covariants can be used to show that a representation is not coreduced. As a first application we classify the coreduced representations of $\SLtwo$.

Let $G$ be a reductive group and $V$ a representation of $G$. A $G$-equivariant morphism $\phi\colon V \to W$ where $W$ is an irreducible representation of $G$ is called a {\it covariant of $V$ of type $W$\/}. Clearly, covariants of type $W$ can be added and multiplied with invariants and thus form an $\OO(V)^{G}$-module $\Cov(V,W)$ which is known to be finitely generated (see \cite[II.3.2~Zusatz]{Kr1984Geometrische-Metho}).

A nontrivial covariant $\phi\colon V \to W$ defines a $G$-submodule $\phi^{*}(W^{*})\subset \OO(V)$ isomorphic to the dual $W^{*}$, and every irreducible $G$-submodule of $\OO(V)$ isomorphic to $W^{*}$ is of the form $\phi^{*}(W^{*})$ for a suitable covariant $\phi$. Moreover, $\phi$ vanishes on the null cone $\NN$ if and only if $\phi^{*}(W^{*})\subset I(\NN)$. We say that $\phi$ is a {\it generating covariant\/} if $\phi$ is not contained in $\mm_{0}\Cov(V,W)$, or equivalently, if $\phi^{*}(W^{*})$ is not contained in $I_{G}(\NN)$. Thus we obtain the following useful criterion for non-coreducedness.

\begin{proposition}\lab{prop:gencovariants}
If $\phi$ is a generating covariant which vanishes on $\NN$ then $V$ is not coreduced.
\end{proposition}

\begin{remark}\lab{rem:gendifferential}
Let $f\in\OOO(V)^{G}$ be a generating homogeneous invariant of positive degree, i.e., $f\in\mm_{0}\setminus\mm_{0}^{2}$. Then the differential $df\colon V\to V^{*}$ is a generating covariant. In fact,  using the contraction $(df,\id)= \deg f \cdot f$ we see that if $df = \sum_{i}f_{i}\phi_{i}$ where the $f_{i}$ are  homogeneous non-constant invariants, then $f = \frac{1}{\deg f}\sum_{i}f_{i}(\phi_{i},\id)\in\mm_{0}^{2}$.
\end{remark}

\begin{example}\lab{ex:twosltwo}
Let $G$ be $\SLtwo$ and $V = \sll_{2}\oplus \sll_{2}$ where $\sll_{2}=\Lie \SLtwo$ is the Lie algebra of $\SLtwo$. Then the null cone $\NN(V)$ consists of commuting pairs of nilpotent matrices and so the covariant 
$$
\phi\colon \sll_{2}\oplus \sll_{2}\to \sll_{2}, (A,B)\mapsto AB-\frac{1}{2}\tr(AB)   
\left[\begin{matrix}1&0\\0&1\end{matrix}\right]
$$  
vanishes on $\NN(V)$, i.e.,  $\phi^{*}(\sll_{2}) \subset I(\NN)$. But $\phi^{*}(\sll_{2})$ is bihomogeneous of degree $(1,1)$ and therefore is not contained in $I_{G}(\NN)$  because there are no invariants of degree 1.
\end{example}

\begin{example}\lab{ex:3D2}
Let $G$ be $\SO_{4}$ and $V=\C^4\oplus \C^4\oplus \C^4$. The weights of $\C^{4}$ relative to %
the maximal tours $T=\SO_2\times\SO_2$ are $\pm\eps_{1}$, $\pm\eps_{2}$, and the degree 2 invariants (dot products) $q_{ij}\colon (v_{1},v_{2},v_{3})\mapsto v_i\cdot v_j$, $1\leq i\leq j\leq 3$, generate the invariant ring. Let $V_{++}$ denote the span of the positive weight vectors and let $V_{+-}$ denote the span of the weight vectors corresponding to $\epsilon_1$ and $-\epsilon_2$. Then $\NN= G V_{++}\cup GV_{+-}$, and an easy calculation shows that every homogeneous covariant $V \to \C^{4}$ of degree 3 vanishes on the null cone $\NN$. Now, using LiE (see \cite{LeCoLi1992LiE:-A-package-for}), one finds that the covariants of type $\C^{4}$ have multiplicity 19 in degree 3, whereas there are 6  linearly independent invariants in degree 2 and obviously 3  linear covariants of type $\C^4$.  Therefore, there is a generating covariant of type $\C^{4}$ in degree 3 and so $V$ is not coreduced. (See  Theorem~\ref{CIT.thm}(4) for a more general statement.)
\end{example}

We now use our method to classify the cofree $\SLtwo$-modules. Starting with the natural representation on $\C^{2}$ we get a linear action on the coordinate ring $\OO(\C^{2})=\C[x,y]$ where $x$ has weight 1 with respect to the standard torus $T=\Cst \subset \SLtwo$. The homogeneous components $R_{m}:=\C[x,y]_{m}$ of degree $m$ give all  irreducible representations of $\SLtwo$ up to isomorphism. A binary form $f\in R_{m}$ will be written as 
$$
f=\sum_{i=0}^{m}a_{i}\binom{m}{i}x^{i}y^{m-i}
$$
so that the corresponding coordinate functions $x_{i}$ are weight vectors of weight $m-2i$. The null cone of $R_{m}$ consists of those forms $f$ which have a linear factor of multiplicity strictly greater than  $\frac{m}{2}$. More generally, for any representation $V$ of $\SLtwo$ we have
$$
\NN= \SLtwo V_{+}
$$
where $V_{+}$ is the sum of all weight spaces of strictly positive weight. In particular, $\NN$ is always irreducible.
\begin{example}\lab{ex:R4}
The binary forms of degree 4 have the following invariants 
$$
A:=x_{0}x_{4} - 4x_{1}x_{3}+ 3x_{2}^{2} \quad\text{and}\quad H:=\det\begin{bmatrix} x_{0}&x_{1}&x_{2}\\x_{1}&x_{2}&x_{3}\\x_{2}&x_{3}&x_{4}\end{bmatrix}
$$
classically called ``Apolare'' and ``Hankelsche Determinante'' which generate the invariant ring (see \cite{Sc1968Vorlesungen-uber-I}). It is easy to see that the null cone $\NN(R_{4}) =\SLtwo(\C x^{3}y\oplus\C x^{4})$ is the closure of the 3-dimensional orbit of $x^{3}y$ and thus has codimension 2. 
A simple calculation shows that the Jacobian $\Jac(H,A)$ has rank 2 at $x^{3}y$. It follows that $\NN(R_{4})$ is a reduced complete intersection. (One can deduce from $\rk\Jac(H,A) =2$  that $A,H$ generate the invariants.)
\end{example}

\begin{example}\lab{ex:kR1}
The representation $k R_{1}= R_{1}\oplus R_{1}\oplus \cdots\oplus R_{1}$ ($k$ copies) can be identified with the space $M_{2\times k}(\C)$ of $2\times k$-matrices where $\SLtwo$ acts by left multiplication. Then the null cone $\NN$ is the closed subset of matrices of rank $\leq 1$ which is the determinantal variety defined by the vanishing of the $2\times 2$-minors $M_{ij}= x_{1i}x_{2j}-x_{2i}x_{1j}$, $1\leq i<j\leq k$. It is known that the ideal of $\NN$ is in fact generated by the minors $M_{ij}$. This is an instance of the so-called Second Fundamental Theorem, see \cite[Chap.~11, section~5.1]{Pr2007Lie-groups}. Thus $\NN$ is reduced, and the minors $M_{ij}$ generate the invariants. 
\end{example}

\begin{theorem}\lab{thm:SLtwo}
Let $V$ be a nontrivial coreduced representation of $\SL_2$ where $V^{\SLtwo}=(0)$. Then $V$ is isomorphic to $R_2$, $R_3$, $R_4$ or $kR_1$, $k\geq 1$.
\end{theorem}
The proof is  based on the following results. 

\begin{lemma}\lab{lem:vanishingcov}
Let $V$ be a representation of $\SLtwo$ and $\phi\colon V \to R_{m}$ a homogeneous covariant of degree $d$.
\be
\item If $d>m$, then $\phi(\NN)=(0)$.
\item If $\pm\id$ acts trivially on $V$ and $2d>m$, then $\phi(\NN)=(0)$.
\ee
\end{lemma}

\begin{proof} Let $V_{+}\subset V$ be the sum of the positive weight spaces. Since $\NN=\SLtwo V_{+}$ it suffices to show that $\phi$ vanishes on $V_{+}$. Choose coordinates $x_{1},\ldots,x_{n}$ on $V$ consisting of weight vectors and let $z=x_{1}^{k_{1}}x_{2}^{k_{2}}\cdots x_{n}^{k_{n}}$ be a monomial occurring in a component of $\phi$. Then $\sum_{i}k_{i}=d$ and the weight of $z$ occurs in $R_{m}$. Since $m<d$ the monomial $z$ must contain a variable $x_{i}$ with a weight $\leq 0$ and so $z$ vanishes on $V_{+}$. This proves (1).

For (2) we remark that $V$  contains only even weights and so a variable $x_{i}$ with non-positive weight has to appear in $z$ as soon as $2d>m$.
\end{proof}

\begin{lemma}\lab{lem:trivStab}
Let $V$ be a nontrivial representation of $\SLtwo$ not isomorphic to $R_{1},R_{2},R_{3}$ or $R_{4}$. Then the principal isotropy group  is either trivial or equal to $\{\pm\id\}$.
\end{lemma}

\begin{proof}
This is well-known for the irreducible representations $R_{j}$, $j\geq 5$. Let $T$ and $U$ denote the usual maximal torus and maximal unipotent subgroup  of $\SL_2$. Denote by $H_{i}$ the generic stabilizer of $R_{i}$, $i=1,2,3$ and $4$. Then we have $H_{1}=U$, $H_{2}=T$, $H_{3}= \{\begin{bmatrix}\zeta&0\\0&\zeta^{2}\end{bmatrix}\mid\zeta^{3}=1\}\simeq \Z/3$ and $H_{4}=\{\begin{bmatrix}\zeta&0\\0&\zeta^{3}\end{bmatrix}\mid\zeta^{4}=1\}\cup\{\begin{bmatrix}0&\zeta\\-\zeta^{3}&0\end{bmatrix}\mid\zeta^{4}=1\}\simeq\tilde Q_{8}$, the group of quaternions of order 8. It is easy to see that the generic stabilizer of $H_{1}$ and $H_{3}$ on any nontrivial representations of $\SLtwo$ is trivial, and that the generic stabilizer of $H_{2}$ and $H_{4}$ on the $R_{2j}$, $j>0$, is $\{\pm\id\}$.
\end{proof}

\begin{proof}[Proof of Theorem~\ref{thm:SLtwo}] For $V=R_{2}$ or $R_{3}$ the quotient $\quot{V}{\SLtwo}$ is one-dimensional and so both are coreduced. In the examples~\ref{ex:R4} and \ref{ex:kR1} we have seen that $R_{4}$ and $kR_{1}$ are coreduced. So it remains to show that any other representation $V$ of $\SLtwo$ is not coreduced.

By Lemma~\ref{lem:trivStab} we can assume that the principal isotropy group is trivial or $\{\pm\id\}$. In the first case, $V$ contains a closed orbit isomorphic to $\SLtwo$. By Frobenius reciprocity, we know that the multiplicity of $R_{m}$ in $\OOO(\SLtwo)$ is equal to $\dim R_{m}=m+1$. This implies that the rank of the $\OOO(V)^{G}$-module $\Cov(V,R_{m})$ is at least $m+1$. Since we can assume that $V$ contains at most one summand isomorphic to $R_{1}$ this implies that there is a generating homogeneous covariant $\phi\colon V \to R_{1}$ of degree $>1$. By Lemma~\ref{lem:vanishingcov}(1) and Proposition~\ref{prop:gencovariants}  it follows that $V$ is not coreduced.

Now assume that the principal isotropy group is $\{\pm\id\}$. As above this implies that the rank of $\Cov(V,R_{2})$ is at least $3$. Since $R_{2}\oplus R_{2}$ is not coreduced (Example~\ref{ex:twosltwo}) we can assume that $V$ contains at most one summand isomorphic to $R_{2}$. It follow that there is a generating homogeneous covariant $\phi\colon V \to R_{2}$ of degree $>1$ and the claim follows by Lemma~\ref{lem:vanishingcov}(2) and Proposition~\ref{prop:gencovariants}.
\end{proof}

\bigskip
\section{Cofree Representations}\lab{sec:cofreerep}
Let $G$ be a (connected) reductive group, $V$ a $G$-module and $\pi\colon V \to \quot{V}{G}$ the quotient morphism. 

\begin{remark}\lab{rem:reducedfiber}
Let $Gv$ be a closed orbit with slice representation $(N_v,G_v)$. Then, by \name{Luna}'s slice theorem, the fiber $F:=\pi\inv(\pi(v))$ is isomorphic to $G\times^{G_v}\NN(N_v)$ which is   a bundle over $G/G_v\simeq Gv$ with fiber $\NN(N_v)$. Hence $F$ is reduced if and only if $N_v$ is coreduced.
\end{remark}

If the fiber $F=\pi^{-1}(z)$ is reduced, then $F$ is smooth in a dense open set $U\subset F$ which means that the rank of the differential $d\pi_{u}$ equals $\dim V - \dim_{u}F$ for $u\in U$. Thus we get the following criterion for non-coreducedness.

\begin{lemma}\lab{lem:differential}
If $X$ is an irreducible component of $\NN(V)$ and the rank of $d\pi$ on $X$ is less than the codimension of $X$ in $V$, then $V$ is not coreduced.
\end{lemma}
 
Recall that a $G$-module $V$ is said to be \emph{cofree\/} if $\OO(V)$ is a free $\OO(V)^G$-module. Equivalently, $\OO(V)^G$ is a polynomial ring ($V$ is \emph{coregular}) and the codimension of $\NN(V)$ is $\dim\quot VG$. See \cite{Sc1978Representations-free} for more details and a classification of cofree representations of simple groups.

\begin{proposition}\lab{prop:slice.cofree}
Let $V$ be a cofree $G$-module. If the null cone is reduced, then so is every fiber of $\pi\colon V\to\quot VG$, and every slice representation of $V$ is coreduced.
\end{proposition}

\begin{proof}  Since $V$ is cofree, the map $\pi$ is flat. By \cite[12.1.7]{Gr1967Elements-de-geomet}, the set $\{v\in V\mid \pi\inv(\pi(v)) \text{ is reduced at }v\}$ is open in $V$. But this set is a cone. Thus if the null cone is reduced, then so is any fiber of $\pi$, and every slice representation is coreduced. 
 \end{proof}

For a cofree representation $V$ the (schematic) null cone $\NN(V)$ is a complete intersection. 
Using \name{Serre}'s criterion \cite[Ch.~8]{Ma1989Commutative-ring-t} one can characterize quite precisely when $\NN(V)$ is reduced.  

\begin{proposition}\lab{prop:Serre}
Let $V$ be a cofree $G$-module.
Then $V$ is coreduced if and only if  rank  $d\pi=\codim\NN(V)$ on an open dense subset of $\NN(V)$. 
\end{proposition}

\begin{example}\lab{ex:quadraticform}
Let $G=\SL_{n}$ and $V := S^{2}(\C^{n})^{*}\oplus \C^{n}$. The quotient map $\pi\colon V \to \C^{2}$ is given by the two invariants $f:=\det(q)$ and $h:=q(v)$ of bidegrees $(n,0)$ and $(1,2)$ where $(q,v)\in V$. An easy calculation shows that for $n=2$ the differential $d\pi$ has rank $\leq 1$ on the null cone. Hence $(S^{2}(\C^{2})^{*}\oplus \C^{2}, \SLtwo)$ is not coreduced, which we already know from Theorem \ref{thm:SLtwo}.

We claim that for $n\geq 3$ the null cone is irreducible and reduced. Set $q_{k}:=x_k^{2}+ x_{k+1}^{2}+ \cdots + x_{n}^{2}$,
$$
X_{k}:=\{q_{k}\} \times \{v\in\C^{n}\mid q_{k}(v)=0\}\subset V\text{ and } X_{n+1}:=\{0\}\times \C^{n}.
$$
Then $\NN(V) = \bigcup_{k=2}^{n+1} \SL_{n}\cdot X_{k}$. Since $\dim \SL_{n}\cdot X_{k}= \dim \SL_{n}q_{k}+ n-1$ for $2\leq k \leq n$ we get $\codim  \SL_{n}\cdot X_{2} = 2 < \codim  \SL_{n}\cdot X_{k}$ for all $k>2$, and so $\NN(V) = \overline{ \SL_{n}\cdot X_{2}}$ is irreducible. Moreover, 
$$
d f_{(q_{2},v)}(q,w)= a_{11} \quad \text{and}\quad d h_{(q_{2},v)}(q,w)= \sum_{i=2}^{n} 2v_{i}y_{i}+q(v)
$$ 
where 
$v=(v_{1},\ldots,v_{n})$, $q = \sum a_{ij}x_{i}x_{j}$ and $w = (y_{1},\ldots, y_{n})$. It follows that the two linear maps
$d f_{(q_{2},v)}$ and $d h_{(q_{2},v)}$ are linearly independent on a dense open set of $X_{2}$, hence the claim.
\end{example}

In order to see that the null cone is reduced in a dense set we can use the 
following result due to \name{Knop}  \cite{Kn1986Uber-die-Glattheit} which goes back to 
\name{Panyushev} \cite{Pa1985Regular-elements-i}. Recall that the {\it regular sheet\/} $\SS_{V}$ of a representation 
$(V,G)$ of an algebraic group is the union of $G$-orbits of maximal dimension.

\begin{proposition}\lab{prop:kanModul} 
Let $(V,G)$ be a representation of a semisimple group and let $\pi\colon V \to\quot VG$ be the quotient map. Assume that  $x\in V$ belongs to the regular sheet and that $\pi(x)$ is a smooth point of the quotient. Then $\pi$ is smooth at $x$.
\end{proposition}

\begin{corollary}\lab{cor:cofree} 
Let  $(V,G)$ be a cofree representation of a semisimple group. Assume that the regular sheet $\SS_{V}$ of $V$ meets the null cone $\NN(V)$ in a dense set. Then $(V,G)$ is coreduced.
\end{corollary}

Let $\theta$ be a finite automorphism of a semisimple group $H$ and let $G$ denote the identity component of the fixed points $H^{\theta}$. Given any eigenspace $V$ of $\theta$ on the Lie algebra $\lieg$ of $G$, we have a natural representation of $G$ on $V$. These representations  $(V,G)$ are called $\theta$-representa\-tions. They have been introduced and studied by \name{Vinberg}, see \cite{Vi1976The-Weyl-group-of-}. Among other things he proved that $\theta$-representations are cofree and that every fiber of the quotient map contains only finitely many orbits. As a consequence of Corollary~\ref{cor:cofree} above we get the following result.

\begin{corollary}\lab{thm:theta} 
Every $\theta$-representation $(V,G)$ where $G$ is semisimple is coreduced.
\end{corollary}

Finally we can prove the main result of this section.
\begin{theorem}\lab{thm:cofree}
An irreducible cofree representation $V$ of a simple group $G$ is coreduced.
\end{theorem}

\begin{proof}
It follows from the classification (see \cite{Po1976Representations-wi}, \cite{KaPoVi1976Sur-les-groupes-li}, \cite{Sc1978Representations-free}) that the only irreducible cofree representations of the simple groups which are not $\theta$-representations (or have one-dimensional quotient) are the spin representation of  $\Spin_{13}$ and the half-spin representations of $\Spin_{14}$. For these representations, \name{Gatti-Viniberghi} \cite{GaVi1978Spinors-of-13-dime} show that every irreducible component of the null cone has a dense orbit.
\end{proof}

\begin{example}\lab{ex:twosltwoB} We give an example of a {\it cofree but not coreduced\/} representation.
Let $V=\sll_{2}\oplus\sll_{2}\simeq 2R_{2}$ as in Example~\ref{ex:twosltwo}. Each copy of $R_{2}$ has a weight basis $\{x^2,xy,y^2\}$ relative to the action of the maximal torus $T=\C^*$. The null cone is $U^-(\C x^2\oplus \C x^2)$ where $U^-$ is the maximal unipotent subgroup of $G$ opposite to the usual Borel subgroup. One can easily see that $d\pi_{(\alpha x^2,\beta x^2)}$ is nontrivial only on the vectors $(\gamma y^2,\delta y^2)$, giving a rank of two. But $V$ is cofree with $\codim \NN(V)=3$. Thus $V$ is not coreduced.
\end{example}

We can prove a sort of converse to the theorem above. 
Recall that $V$ is {\it strongly coreduced\/} if every fiber of $\pi$ is reduced; equivalently, every slice representation of $V$ is coreduced.

\begin{theorem}\lab{thm:stronglycoreduced}
A strongly coreduced irreducible representation of a simple group $G$ is cofree.
\end{theorem}
If $G$ is simple, then we use the ordering of Bourbaki \cite{Bo1968Elements-de-mathem} for the simple roots $\alpha_j$ of $G$ and we let $\phi_j$ denote the corresponding fundamental representations. We use the notation $\nu_j$ to denote the 1-dimensional representation of $\C^*$ with weight $j$.

\begin{proof} 
We use the techniques of  \cite{KaPoVi1976Sur-les-groupes-li} (but we follow the Appendix of \cite{Sc1978Representations-reg}). Let $V$ be non-coregular (which is the same as $V$ not being cofree). If $V$ is $\phi_1^3(\AA_3)$ or $\phi_2^3(\AA_3)$, then there is a closed orbit with finite stabilizer whose slice representation is not coregular. Thus the slice representation is certainly not trivial, hence $V$ is not strongly coreduced. Otherwise,  there is a copy $T=\C^*\subset\SL_2\subset G$ and a closed orbit $Gv$, $v\in V^T$, such that the identity component $G_v^0$ of the stabilizer $G_v$ of $v$ has rank $1$. Moreover, one of the following occurs:
 \begin{enumerate}
\item $G_v^0=T$ and the slice representation of $G_v$, restricted to $T$, has at least 3 pairs of nonzero weights $\pm a$, $\pm b$, $\pm c$ (where we could have $a=b=c$).
\item The module is $\phi_1\phi_2(\AA_3)$ or $\phi_2\phi_3(\AA_3)$,  $G_v$ centralizes $G_v^0=T$    and the slice representation is $\theta_2+\nu_1+\nu_{-1}+\nu_{2}+\nu_{-2}$ where $\theta_n$ denotes the $n$-dimensional trivial representation.
\item $G_v^0=\SL_2$ and the slice representation of $G_v$, restricted to $T$, contains at least 4 pairs of weights $\pm a$, $\pm b$, $\ldots$.
\end{enumerate}
If, in case (1) above, the weights are not of the form $\pm a$ for a fixed $a$, then the $G$-module $V$ is not strongly coreduced. The same remark holds in case (3). Of course, in case (2), the module is not strongly coreduced. 
We went through the computations again and saw in which cases the weights of the slice representations were of the form $\pm a$ for a fixed $a$. One gets a list of representations as follows. (The list is complete up to  automorphisms of the group.)
\begin{enumerate}
\addtocounter{enumi}{3}
\item $\phi_i(\AA_n)$, $5\leq i$, $2i\leq n+1$.
\item $\phi_n(\BB_n)$, $n\geq 7$.
\item $\phi_n(\DD_n)$, $n\geq 9$.
\item $\phi_i(\CC_n)$, $3\leq i\leq n$, $n\geq 5$.
\end{enumerate}
For the groups of type $\AA$ and $\CC$, consider $\SL_2\subset G$ such that  the fundamental representation restricted to $\SL_2$ is $2R_1+\theta_{1}$. For the groups of type $\BB$ and $\DD$ consider $\SL_2\subset G$ such that the fundamental representation restricted to $\SL_2$ is $4R_1+\theta$. Then using the techniques of \cite{KaPoVi1976Sur-les-groupes-li} one sees that there is a closed orbit in $V^{\SL_2}$ whose stabilizer is a finite extension of $\SL_2$ such that the slice representation restricted to $\SL_2$ contains at least two copies of $R_2$. Hence the slice representation is not coreduced.
\end{proof}

\bigskip
\section{The Method of Slices and Multiplicity of Weights}\lab{sec:slices}

Let $G$ be a reductive group and $T\subset G$ a maximal torus. 
It is well-known that  the orbit $Gv$ is closed for any zero weight vector $v\in V^T$.
We say that {\it $V$ is a $G$-module with a zero weight\/} if $V^{T}\neq (0)$. The basic result for  such modules  is the following.

\begin{proposition}\lab{prop:slice.zero.weight}
Let $V$ be a $G$-module with a zero weight. If the slice representation at a zero weight vector is not coreduced then neither is  $V$.
\end{proposition}

For the proof we use the following result. If $X \subset V$ is a closed subset of a vector space $V$, then the {\it associated cone\/} $\Cone X$ of $X$ is defined to be the zero set of $\{\gr f \mid f\in I(X)\}$ where $\gr f$ denotes  the (nonzero) homogeneous term of $f$ of highest degree. If $V$ is a $G$-module and $X$ a closed subset of a fiber $F \neq \NN(V)$ of the quotient map, then $\Cone X = \overline{\Cst X} \setminus \Cst X$ (cf.\ \cite[\S 3]{BoKr1979Uber-Bahnen-und-de}).

\begin{lemma}\lab{associatedcone.lem}
Suppose that  $G_v$ has the same rank as $G$. Then the associated cone of $F:=\pi\inv(\pi(v))$ is equal to $\NN(V)$.
\end{lemma}

\begin{proof} 
We know that the associated cone of every fiber of $\pi$ is contained in $\NN(V)$. For the reverse inclusion we can assume that $T\subset G_{v}$. Let $v_{0}\in\NN(V)$. Then $\overline{Tgv_{0}}\ni 0$ for a suitable $g\in G$. This implies that $\overline{T(gv_{0}+v})\ni v$ and so $\C gv_{0}+v\subset F$. The lemma follows since $gv_{0}\in \overline{{\Cst}(\C gv_{0}+v)}$.
\end{proof}

\begin{proof}[Proof of Proposition~\ref{prop:slice.zero.weight}]
Suppose that $\NN(V)$ is reduced, and let $0\neq f\in\I(F)$ where $F$ is as in the lemma above. Then the leading term $\gr f$ lies in the ideal of $\NN(V)$, so that there are homogeneous $f_i\in\mm_{0}$ and homogeneous $h_i\in \OO(V)$ such that $\gr f=\sum_i f_i h_i$ where $\deg f_i+\deg h_i=\deg \gr f$ for all $i$. Then $\tilde f:=\sum_i (f_i-f_i(v))h_i$ lies in $\I_G(F)$ and $\gr f = \gr\tilde f$. Replacing $f$ by $f-\tilde f$ we are able to reduce the degree of $f$. Hence by induction we can show that $f\in \I_G(F)$. Thus $F$ is reduced. 
\end{proof}

\begin{example}\lab{ex:binaryformsviazeroweight}
We use the proposition above to give another proof that the irreducible representations $R_{2m}$ of $\SLtwo$ are not coreduced for $m\geq 3$ (see Theorem~\ref{thm:SLtwo}).
We have $R_{2m}^{T} = \C x^{m}y^{m}$, and a zero weight  vector $v$ has stabilizer $T\simeq\C^*$ if $m$ is odd and  $N(T)$ if $m$ is even. The slice representation of $T$ at $v$ has the weights $\pm 4, \dots, \pm 2m$ (each with multiplicity one), and so, for $m\geq 3$, we have at least the weights $\pm 4$ and $\pm 6$. But then the slice representation restricted to $T$ is not coreduced, hence neither are the representations $R_{2m}$ of $\SLtwo$ for $m\geq 3$. 
\end{example}

\begin{example}\lab{ex:twoadjointreps}
Let $G$ be a semisimple groups and $\lieg$ its Lie algebra. Then the representation of $G$ on $\lieg\oplus\lieg$ is not coreduced.
\begin{proof}
Let $T\subset G$ be a maximal torus and $\alpha$ a root. Put  $T_{\alpha}:=(\ker\alpha)^{0}$. Then for a generic $x\in \Lie T_{\alpha}$ we have $G_x = \Cent_{G}T_{\alpha}= G_{\alpha}\cdot T_{\alpha}$ where $G_{\alpha}\simeq \SL_{2}$ or $\simeq \PSL_{2}$. The slice representation at $x$ is $\Lie(G_x)\simeq \sltwo+\theta_{\ell-1}$ where $\ell=\rk G$. Thus the slice representation at $(x,0)\in\lieg\oplus\lieg$ restricted to $G_{\alpha}$ contains two copies of $\sltwo$, and the result follows from Example~\ref{ex:twosltwo} or Example~\ref{ex:twosltwoB}.
\end{proof}
\end{example}

Let $G$ be semisimple with Lie algebra $\lieg$. If $\mu$ is a dominant weight of $\lieg$, let $V(\mu)$ denote the corresponding simple $G$-module.  Recall that the following are equivalent: 
\be
\item[(i)] $V(\mu)$ has a zero weight; 
\item[(ii)] All weights of $V(\mu)$ are in the root lattice; 
\item[(iii)] $\mu$ is in the root lattice;
\item[(vi)] The center of $G$ acts trivially on $V(\mu)$.
\ee

\begin{remark}\lab{rem:shortroots}
Let $V$ be a non-trivial simple $G$-module with a zero weight. Then the short roots are weights of $V(\mu)$ and the highest short root is the smallest nontrivial dominant weight. This follows from the following result due to \name{Steinberg}, see \cite{St1998The-partial-order-}. (We thank John Stembridge for informing us of this result.)
\end{remark}

\begin{lemma}\label{lem:domweights}  Let $\nu\prec\mu$ be dominant weights. Then there are positive roots $\beta_{i}$, $i=1,\ldots,n$, such that 
\be
\item $\mu - \nu = \beta_{1}+\beta_{2}+\cdots+\beta_{n}$, and
\item $\mu-\beta_{1}-\cdots-\beta_{j}$ is dominant for all $j=1,\ldots,n$.
\ee
\end{lemma}
If there is a $v\in V^{T}$ such that $G_{v}^{0}=T$, then we can use  Proposition~\ref{prop:tori} to show that the slice representation at $v$ is not coreduced by giving an indecomposable relation of the weights of the slice representation which involves coefficients $\geq 2$. We will see that this is a very efficient method to prove non-coreducedness in many cases.

\begin{lemma}\lab{lem:torus.slice}
Let $G$ be a semisimple group and let $V$ be a $G$-module. Then all the roots of $\lieg$ are weights of $V$ if and only if there is a zero weight vector $v\in V^{T}$ whose isotropy group is a finite extension of the maximal torus $T$ of $G$. 
\end{lemma}
   
\begin{proof} 
Clearly if $(G_v)^0=T$, then the roots of $\lieg$ are weights of $V$. Conversely, assume all the roots appear  and let $\alpha$ be a root of $\lieg$. The weight spaces with weight a multiple of  $\alpha$ form a  submodule of $V$ for the action of the corresponding copy of $\SL_2$. Since $\alpha$ occurs as a weight of $V$, this module is not the trivial module. Hence there is a   $v\in V^{T}$ such that $x_{\alpha}(v)\neq 0$ where $x_\alpha\in\lieg$ is a root vector of $\alpha$. Thus the kernel of $x_\alpha$ is a proper linear subspace of  $V^{T}$ and there is a $v\in V^{T}$ which is not annihilated by any $x_\alpha$. Then the isotropy subalgebra of $v$ is $\lie t$.
\end{proof}

\begin{definition} We say that a representation $V$ of $G$ has {\it a toral slice\/} if there is a $v\in V^{T}$ such that $G_{v}^{0}=T$. We say that $V$ has {\it a bad slice\/} if there is a $v \in V^{T}$ such that the slice representation at $v$ restricted to $G_{v}^{0}$ is not coreduced, and that $V$ has {\it a bad toral slice} if, in addition, $G_{v}^{0}= T$.
\end{definition}

Now Proposition~\ref{prop:slice.zero.weight} can be paraphrased by saying that {\it a representation with a bad slice is not coreduced}.

\begin{example}\lab{ex:s3kc3}
Consider the representation $(S^{3k}(\C^3),\SL_3)$, $k\geq 2$. Then the isotropy group of the zero weight vector is a finite extension of the maximal torus $T$ of $\SL_3$, and the slice representation $W$ of the torus contains the highest weight $2k\alpha+k\beta$ as well as the weights $-k\alpha$ and $-k\beta$. Thus there is the ``bad'' relation
\[
(2k\alpha+k\beta) + 2(-k\alpha) + (-k\beta) = 0
\]
and so $V$ has a bad toral slice.
\end{example}

\begin{example}\lab{ex:A1A1} The following representations are not coreduced.
\be
\item $G = \SL_{2}\times\SL_{2}$ on  $(\sltwo\otimes\sltwo)\oplus(R_{i}\otimes R_{j})$, $i+j\geq 1$;
\item $G=\SL_{2}\times\SL_{2}\times\SL_{2}$ on  $(\sltwo\otimes\sltwo\otimes\C)\oplus (\C\otimes\sltwo\otimes \sltwo)$.
\ee
\begin{proof}
(1) Let $t\in\sltwo$ be a nonzero diagonal matrix. The stabilizer of $v=t\otimes t\in \sltwo\otimes\sltwo$ is $\Cst\times\Cst$. If $i$ is odd then the slice representation contains the weights $(\pm2,\pm2)$ and $(\pm1,0)$ or $(\pm1,\pm1)$, and so we find the bad relations
$$
(2,2) + (2,-2) + 4(-1,0)=0 \quad\text{or}\quad (2,2) + 2(-1,-1)=0.
$$
The same argument applies if $j$ is odd.
If $i$ and $j$ are both even and $i>0$, then the slice representation contains the weights $(\pm2,\pm2)$ and $(\pm2,0)$, and so we find the bad relation
$$
(2,2) + (2,-2) + 2(-2,0)=0.
$$

(2) The stabilizer of $v=t\otimes t\otimes x + x\otimes t \otimes t$ where $0\neq x\in\C$   is $\Cst\times\Cst\times\Cst$. The slice representation contains the weights $(\pm 2,\pm 2, 0)$ and $(0,\pm 2,0)$, and we can proceed as in (1).
\end{proof}
\end{example}

\begin{example}\lab{ex:2c3}
Consider the second fundamental representation of $\Sp_{6}$: $\phi_{2}(\CC_{3})=\bigwedge_{0}^{2}\C^{6}:=\bigwedge^2(\C^6)/\C\beta$ where $\beta\in \bigwedge^2(\C^6)$ is the invariant form. It has the isotropy group $\Sp_{2}\times\Sp_{4}$ with slice representation $\bigwedge_{0}^{2}\C^{4}+\theta_{1}$. We claim that $(2\bigwedge_{0}^{2}\C^{6} ,\Sp_{6})$ is not coreduced, although it it cofree (\cite{Sc1978Representations-free}). 
\newline
In fact, the slice representation is $(2\bigwedge_{0}^{2}\C^{4}+\C^{2}\otimes \C^{4}+\theta_{2},\Sp_{2}\times\Sp_{4})$. Quotienting by $\Sp_{2}$ we get a hypersurface $F \subset 3\bigwedge_{0}^{2}\C^{4}+\theta_{3}$ defined by an $\Sp_{4}$-invariant function. Now the claim follows from   Example~\ref{ex:invariant.hypersurface}, because $(3\bigwedge_{0}^{2}\C^{4},\Sp_{4})=(3 \C^{5},\SO_{5})$ is not coreduced as we will see in Theorem~\ref{CIT.thm}(4).
\end{example}

Next we want to show that a representation $V$ is not coreduced if the weights contain all roots with multiplicity at least 2. This needs some preparation.

\begin{lemma} \lab{lem:slice.tensor}
Let $(V,G)$ and $(W,H)$ be two representations. Let $v\in V$ and $w\in W$ be nonzero zero weight vectors with slice representations
$(N_{V}\oplus\theta_{n},G_{v})$ and $(N_{W}\oplus\theta_{m},H_{w})$ where $N_{V}^{G_{v}}=0$ and $N_{W}^{H_{w}}=0$. Then the slice representation $N_{V\otimes W}$ of $G_{v}\times H_{w}$ at $v\otimes w$ contains
\begin{multline*}
(V^{\oplus(m-1)}\oplus N_{V},G_{v}) \oplus (W^{\oplus(n-1)}\oplus N_{W},H_{w}) \oplus \\((\lieg/\lieg_{v}\oplus N_{V})\otimes(\lieh/\lieh_{w}\oplus N_{W}),G_{v}\times H_{w}).
\end{multline*}
\end{lemma}

\begin{proof}
The lemma  follows from the decomposition $(V,G_{v})=(\lieg/\lieg_{v}\oplus N\oplus \theta_{n})$ and similarly for $(W,H_{w})$, and the fact that 
$$
T_{v\otimes w}((G\times H)v\otimes w)=\lieg(v)\otimes w+v\otimes\lieh(w) \subset \lieg/\lieg_v\otimes\theta_m+\theta_n\otimes\lieh/\lieh_w.
$$
\end{proof}

\begin{corollary}\lab{cor:slice.tensor}
\begin{enumerate}
\item The two slice representations $(N_V,G_v)$ and $(N_W,H_w)$ occur  as   subrepresentations of the slice representation   at $v\otimes w$. In particular, if $(V, G)$ has a bad slice, then so does $(V\otimes W,G\times H)$.
\item The slice representation at $v\otimes w$ contains $N_V\otimes N_W$,   $\lieg/\lieg_v\otimes\lieh/\lieh_w$, $\lieg/\lieg_v\otimes N_W$ and $N_V\otimes\lieh/\lieh_w$.
\item If $n>1$ (resp.\ $m>1$), then the slice representation contains a copy of $W$ (resp.\ $V$).
\end{enumerate}
\end{corollary}

\begin{remark}\lab{rem:tensor} Since $G_{v}$ and $H_{w}$ have maximal rank, the isotropy group of $v\otimes w$ can at most be a finite extension of $G_{v}\times H_{w}$. Note also that
the corollary generalizes in an obvious way to a representation of the form $(V_{1}\otimes V_{2}\otimes\cdots\otimes V_{k}, G_{1}\times G_{2} \times \cdots \times G_{k})$ where each $(V_{i},G_{i})$ is a representation with a zero weight.  
\end{remark}

\begin{proposition}\lab{prop:product}
Let $G = G_{1} \times \cdots \times G_{k}$ be a product of simple groups and $V = V_{1} \otimes \cdots \otimes V_{k}$ a simple $G$-module where $k>1$. Assume that all roots of $G$ occur in $V$.
Then $V$ is coreduced if and only if $G$ is of type $\AA_{1}\times \AA_{1}$ and $V = \sll_{2}\otimes\sll_{2}$.
\end{proposition}
 
\begin{proof} By Lemma~\ref{lem:torus.slice}
the product  $T = T_1 \times \cdots \times T_k$ of the maximal tori appears 
as the connected component of the isotropy group  of an element $v:=v_{1}\otimes\cdots\otimes v_{k}\in V^{T}$ where  $v_{i}$ is a generic element in $V_{i}^{T_{i}}$. Denote by $W_{i}:=N_{V_{i}}$ the slice representation  at $v_{i}$. Then the tensor products $W_{i_{1}}\otimes \cdots\otimes W_{i_{m}}$ where  $i_{1}<\cdots<i_{m}$ appear as subrepresentations of the slice representation at $v$ (see Remark~\ref{rem:tensor} above).

First assume that $k>2$. Choose simple roots $\alpha,\beta,\gamma$ of $G_1$, $G_2$, $G_3$, respectively. Then
$$
(\alpha+\beta) + (\beta+\gamma) + (\alpha+\gamma) + 2(-\alpha-\beta-\gamma) = 0
$$
is an indecomposable relation with a coefficient $>1$.

Now assume that $k=2$ and that $\rk G_{1}>1$ and choose two adjacent simple roots $\alpha,\beta$ of $G_{1}$ so that $\alpha+\beta$ is again a root. Let $\gamma$ be a simple root of $G_{2}$. Then the relation
$$
(\alpha + \gamma) + (\alpha - \gamma) + 2(\beta - \gamma) + 2(-(\alpha+\beta)+\gamma) = 0
$$
is indecomposable, but contains coefficients $>1$.

As a consequence, $G$ is of type $\AA_{1}\times\AA_{1}$ and $V = R_{2r}\otimes R_{2s}$. A simple calculation shows that this is coreduced only for $r=s=1$.
\end{proof}

\begin{proposition}\lab{prop:rootsmult2}
Let $G$ be a semisimple group and let $V$ be a $G$-module. Assume that all roots of $G$ are weights of $V$ with multiplicity at least $2$. Then $V$ admits a bad toral slice.
\end{proposition}

\begin{proof} 
Choose a generic element $v$ of the zero weight space $V^{T}$ of $V$. Then $(G_{v})^{0}=T$ by Lemma~\ref{lem:torus.slice}, and all roots occur in the slice representation $W$ of $T$ at $v$ as well as the highest weights of $V$. We will show that there is a bad relation.

If not all simple factors of $G$ are of type $\AA$, then  there is always a root $\alpha$ which expressed in terms of simple roots has some coefficient $\geq 2$: $\alpha = \sum_{i}n_{i}\alpha_{i}$ where $\{\alpha_{1},\ldots,\alpha_{r}\}$ is a set of simple roots of $\lieg$, $n_{i}\in\N$, and $n_{j}>1$ for some $j$. But then $\alpha + \sum_{i} n_{i}(-\alpha_{i}) = 0$ is a bad relation and thus $N$ is not coreduced.

We may thus assume that $G$ is of type $\AA_{n_{1}}\times\AA_{n_{2}}\times\cdots\times\AA_{n_{k}}$, $n_{1}\geq n_{2}\geq\cdots\geq n_{k}\geq 1$. Let $\{\alpha_{1},\ldots,\alpha_{n}\}$ be a set of simple roots, $n:=n_{1}+n_{2}+\cdots + n_{k}$. We can assume that the highest weights of the irreducible components of $V$ are all of the form $\lambda=\sum_{i}n_{i}\alpha_{i}$ where $n_{i}\in\{0,1\}$. It is easy to see that such a weight is dominant if and only if $\lambda$ is a sum of highest roots.  Thus each irreducible component $V_k$ of $V$ is a tensor product of certain $\sll_{n}$'s. Now it follows from the previous proposition that either  $V_k$ is isomorphic to $\sll_{j}$  or isomorphic to $\sltwo\otimes\sltwo$. If $n_{1}>1$, then $\sll_{n_{1}}\oplus\sll_{n_{1}}$ must occur and so $V$ is not coreduced (Example~\ref{ex:twoadjointreps}). The remaining cases where $G$ is of type $\AA_{1}\times\AA_{1}\times\cdots\times\AA_{1}$ follow immediately from Example~\ref{ex:A1A1}.
\end{proof}

We finish this section with a criterion   for the non-coreducedness of an irreducible representation of a simple group. 
We begin with a lemma about weights and multiplicities.

\begin{lemma}\lab{lem:multweights}
Let $\mu,\nu$ be nonzero dominant weights of $\lieg$. 
\begin{enumerate}
\item If there is a weight of $V(\mu)$ of multiplicity $m$, then there are nonzero weights in $V(\mu+\nu)$ with multiplicity $\geq m$.
\item Suppose that zero is a weight of $V(\mu)$. Then the  multiplicities of the nonzero weights of $V(\mu)$ are bounded above by the multiplicities of the (short) roots.
\end{enumerate}
\end{lemma}

\begin{proof}
Let $v_{\mu}\in V(\mu),v_{\nu}\in V(\nu)$ be highest weight vectors. Recall that the coordinate ring $\OO(G/U)$ is a domain and contains every irreducible representation of $G$ exactly once. Therefore, the multiplication with $v_{\nu}$ is injective and sends $V(\mu)$ into $V(\mu+\nu)$, i.e., $V(\mu)\otimes \C v_{\nu}\hookrightarrow V(\mu+\nu)$ as a $T$-submodule, and we have (1).

For (2), recall that the highest short root is the smallest dominant weight (Remark~\ref{rem:shortroots}). Looking at root strings we see that the multiplicity of the highest short root has to be at least that of an arbitrary nonzero weight.
\end{proof}

The following criterion will be constantly used for the classifications in the following sections. Let $G$ be a simple group. We use the notation $\phi,\psi,\ldots$ for irreducible representations of $G$ and denote by $\phi\psi$ the Cartan product of $\phi$ and $\psi$.

\begin{critA}\lab{critA}
Let $\phi,\psi$ be irreducible representations of $G$ with a zero weight. Then $\phi\psi$ has a bad toral slice in the following cases.
\be
\item[(i)] $\phi$ has a bad toral slice.
\item[(ii)] $\phi\psi$ contains a nonzero weight of multiplicity $>1$.
\item[(iii)] The zero weight of $\phi$ has multiplicity $>1$.
\ee
\end{critA}
\begin{proof} As in the proof above, every nonzero weight vector $w\in\psi$ defines an embedding $\phi \hookrightarrow \phi\psi$ which shows that $\phi\psi$ contains all sums of two (short) roots and therefore all roots. Thus $\phi\psi$ has a toral slice.  In case (i) we choose for $w\in\psi$ a weight vector of weight 0 and obtain a $T$-equivariant embedding $\phi \hookrightarrow \phi\psi$ which shows that $\phi\psi$ has a bad toral slice.

In case (ii)  the (short) roots occur in $\phi\psi$ with multiplicity at least 2. Now let $\alpha$ be a short root. Then $2\alpha$ and $\alpha$ occur in a toral slice representation, and we have the bad relation $(2\alpha) + 2(-\alpha) = 0$.

Finally, (iii) implies (ii) by Lemma~\ref{lem:multweights}(1).
\end{proof}

\begin{remark}\lab{rem:cartanproduct}
Let $G$ be of type $\AA$, $\DD$ or $\EE$. 
If $\phi = \phi_{i_{1}}\phi_{i_{2}}\cdots\phi_{i_{k}}$ is a coreduced representation with a zero weight, then either $k=1$ or all $\phi_{i_{j}}$ are multiplicity-free. In all other cases, $\phi$ has a bad toral slice.
\par\noindent
(If $k>1$ and if one of the $\phi_{i_{j}}$ has a weight space of multiplicity $\geq 2$, then the roots occur in $\phi$ with multiplicity $\geq 2$, by Lemma~\ref{lem:multweights}, and thus $\phi$ is not coreduced, by Proposition~\ref{prop:rootsmult2}.)
\end{remark}

\bigskip
\section{Coreduced Representations of the Exceptional Groups}\lab{sec:cored-exceptional}
Let $G$ be an exceptional simple group. In this section we classify the coreduced representations $V$ of $G$ which contain a zero weight. We know that each irreducible summand of $V$ is coreduced (Example~\ref{ex:retraction}(1)). Therefore it suffices to describe the {\it maximal\/} coreduced representations. The types $\EE_{n}$ and $\Gtwo$ are easy consequences from what we have done so far, but the type $\FF_{4}$ turns out to be quite involved.

\begin{proposition}\lab{prop:typeE}
Let $G$ be a simple group of type $\EE$ and let $V$ be a $G$-module with a zero weight. If $V$ is coreduced, then $V$ is the adjoint representation of $G$. Any other $V$ with a zero weight has a bad toral slice. 
\end{proposition}

\begin{proof}
Since the groups of type $\EE$ are simply laced, every irreducible representation $\phi$ with a zero weight contains all roots and thus has a toral slice. Now it follows from Lemma~\ref{lem:weightsE} below that every representation of the form $\phi\oplus V$ where $V$ is nontrivial  has a bad toral slice. Hence  a coreduced representation with a zero weight is irreducible.

(a) Let $G=\EE_8$. One can check   with LiE that the fundamental representations of $G$ except for the adjoint representation $\phi_{1}(\EE_{8})$ contain the roots with multiplicity $\geq 2$. Since the zero weight of $\phi_{1}(\EE_{8})$ has multiplicity $\geq 2$, it follows from Criterion \ref{critA} that every irreducible representation except the adjoint representation has a bad toral slice. 

(b) Let $G = \EE_{7}$.
Of the fundamental representations only  $\phi_{1}=\lieg=\Ad G$, $\phi_{3}$, $\phi_{4}$ and $\phi_{6}$ are representations of the adjoint group. Using LiE one shows that every fundamental representation except $\phi_{1}$ and the 56-dimensional representation $\phi_{7}$ has a nonzero weight of multiplicity at least 6. Hence, by Remark~\ref{rem:cartanproduct}, the only other candidates for a coreduced representation besides $\phi_{1}$ are $\phi_{7}^{2k}$, $k\geq 1$. But $\phi_{7}^{2}$ contains the roots with multiplicity 5. Thus every irreducible representation except the adjoint representation has a bad toral slice.

(c) Let $G=\EE_6$.  From the fundamental representations only  $\phi_{2}=\lieg$ and $\phi_{4}$ are representations of the adjoint group. By LiE, $\phi_{3},\phi_{4}$ and $\phi_{5}$ have nonzero weights of multiplicity at least 5, and $\phi_{1}^{2}, \phi_{1}\phi_{6}$ and $\phi_{6}^{2}$ have nonzero weights of multiplicities at least 4. Thus all irreducible representations with a zero weight except the adjoint representation $\phi_{2}$ have a bad toral slice. 
\end{proof}

\begin{lemma}\lab{lem:weightsE} 
Let $G$ be a simple group of type $\EE$, $V$ a $G$-module and $T\subset G$ a maximal torus. Then $V$\!, considered as a representation of $T$, is not coreduced. 
\end{lemma}
\begin{proof} We have to show that the weights $\Lambda=\{\lambda_{i}\}$ of $V$ admit a ``bad relation,'' i.e., an indecomposable relation $\sum_{i}n_{i}\lambda_{i}=0$ where $n_{i}\geq 0$ and at least one $n_{j}\geq 2$ (Proposition~\ref{prop:tori}). 
This is clear if $\Lambda$ contains the roots, in particular for all representations of $\EE_{8}$.

For $\EE_{7}$ we first remark that $\omega_{1},\omega_{3},\omega_{4},\omega_{6}$ are in the root lattice and $\omega_{7}\prec \omega_{2}, \omega_{5}$ in the usual partial order. This implies that for every dominant weight $\lambda$ we have either $\omega_{1}\prec\lambda$ or $\omega_{7}\prec\lambda$. Thus the weights of $V$ either contain the roots or the Weyl orbit of $\omega_{7}$. Using LiE one calculates the Weyl orbit of $\omega_{7}$ and shows that there is a ``bad relation'' among these weights. 

Similarly, for  $\EE_{6}$ one shows that for a dominant weight $\lambda$ not in the root lattice one has either $\omega_{1}\prec\lambda$ or $\omega_{6}\prec\lambda$. Then, using LiE, one calculates the Weyl orbit of $\omega_{1}$ and shows that there is a ``bad relation'' among these weights. Since $\omega_{6}$ is dual to $\omega_{1}$ its weights also have a ``bad relation.''
\end{proof}

We prepare to consider $\FF_4$. The following result will be used several times  in connection with slice representations at zero weight vectors.

\begin{lemma}\lab{lem:maxsubgroup}
Let $G$ be simple and $V$ a $G$-module where $V^{G}=0$. Let $H\subset G$ be a maximal connected reductive subgroup which fixes a nonzero point $v\in V$. Then $Gv \subset V$ is closed with stabilizer a finite extension of $H$.
\end{lemma}

\begin{proof}
Since $H$ is maximal, $N_G(H)/H$ is finite so that $Gv$ is closed \cite{Lu1975Adherences-dorbite}. Similarly, $G_v$ can only be a finite extension of $H$.
\end{proof}

For the maximal subgroups of the simple Lie groups see the works of Dynkin \cite{Dy1952Semisimple-subalge,Dy1952Maximal-subgroups-}.

\begin{example}\lab{ex:phi2cnslice}
Let $V=\phi_2(\CC_n)$, $n\geq 3$. Then $H:=\CC_1\times\CC_{n-1}$ is a maximal subgroup of $\CC_n$ where $(\phi_1(\CC_n),\CC_1\times\CC_{n-1})=\phi_1(\CC_1)\oplus\phi_1(\CC_{n-1})$. Now $H$  fixes a line in $V$. Thus a finite extension of $H$ (actually $H$ itself) is the stabilizer of a closed orbit, and one easily sees that the slice representation is $\theta_1+\phi_2(\CC_{n-1})$. By induction one sees that the principal isotropy group of $\phi_2(\CC_n)$ is a product of $n$ copies of $\SL_2$.
\end{example}

\begin{example}\lab{ex:F4}
Let $G = \FF_{4}$ which is an  adjoint group. Now $\phi_{1}=\lieg$  and $\phi_{4}$ is the irreducible 26-dimensional representation whose nonzero weights are the short roots. The representations $\phi_{2}$ and $\phi_{3}$ contain the roots with multiplicities at least two. Moreover, $\phi_{1}^{2}, \phi_{1}\phi_{4}$ and $\phi_{4}^{2}$ contain the roots with multiplicities at least 3. Hence every irreducible representation of $G$ except for $\phi_{1}$ and $\phi_{4}$ has a bad toral slice.
\end{example}

\begin{proposition}\lab{prop:F4}
The representations   $\phi_{1}(\Ffour)$ and $2\phi_{4}(\Ffour)$ are the maximal coreduced representations %
of $\FF_4$. 
Moreover, the representation $2\phi_{4}(\Ffour)$ contains a dense orbit in  the null cone.
\end{proposition}

\begin{proof}
The sum $\phi_1+\phi_4$ is not coreduced because the slice representation of the maximal torus is not coreduced (the nonzero weights are the short roots and these contain a bad relation). This leaves us to consider copies of $\phi_4$. We know that $2\phi_{4}$ is cofree (\cite{Sc1978Representations-free}). So it suffices to show that $3\phi_{4}$ is not coreduced and that $2\phi_{4}$ contains a dense orbit in the null cone. For both statements we use some heavy calculations which are given in \ref{AppA}, see Proposition~\ref{Aprop:F4}.
\end{proof}

\begin{example}\lab{ex:G2}
Let $G = \Gtwo$ which is an adjoint group. The fundamental representation
$\phi_{1}$ of dimension 7 and $\phi_{2}$ (adjoint representation) are the only coreduced irreducible representations. This follows from Criterion \ref{critA}, because $\phi_{1}^{2}$ contains a nonzero weight of multiplicity $\geq 2$.
\end{example}

\begin{proposition}\lab{prop:G2}
Let 
$G=\Gtwo$. Then  $2\phi_{1}$ and the adjoint representation $\phi_{2}$ are the maximal coreduced representations of $G$.
\end{proposition}

\begin{proof}
The invariants of $2\phi_1$ are just the $\SO_7$-invariants, so this representation is coreduced (see Theorem~\ref{CIT.thm}(4)). Now $\bigwedge^3(\phi_1)$ contains a copy of $\phi_1$, and it is easy to see that the corresponding covariant vanishes on the null cone of $3\phi_1$.  
In  fact, this holds for any covariant of type $\phi_{1}$ of degree $\geq 3$.
Since the covariant is alternating of degree three, it cannot be in the ideal of the quadratic invariants. More precisely, we have 
$S^{2}(\phi_{1}\otimes\C^{3})^{G }= \theta_{1}\otimes S^{2}\C^{3}$ and so
$$
S^{2}(\phi_{1}\otimes\C^{3})^{G }\cdot(\phi_{1}\otimes\C^{3}) =
\phi_{1}\otimes(S^{2}\C^{3}\cdot\C^{3})
$$
and this space does not contain $\phi_{1}\otimes \bigwedge^{3}\C^{3}$. Thus $3\phi_1$ is not coreduced. 

To see that $\phi_1+\phi_2$ is not coreduced we choose a nontrivial zero weight vector in $\phi_{2}=\lieg$ which is annihilated by a short root $\alpha$. Then the isotropy group has rank 2 and semisimple rank 1, and the slice representation contains two copies of $(R_{2},\AA_{1})$, hence is not coreduced (Theorem~\ref{thm:SLtwo}).
\end{proof}

Let us summarize our results.

\begin{theorem}\lab{thm:coredexceptional}
The following are the maximal coreduced representations of the exceptional groups containing a zero weight.
\be
\item For $\EE_{n}$: the adjoint representations $\phi_{2}(\EE_{6}), \phi_{1}(\EE_{7}), \phi_{1}(\EE_{8})$.
\item For $\FF_{4}$: $\Ad \FF_{4}=\phi_{1}(\FF_{4})$ and $2\phi_{4}(\FF_{4})$.
\item For $\Gtwo$: $\Ad \Gtwo=\phi_{2}(\Gtwo)$ and $2\phi_{1}(\Gtwo)$.
\ee
\end{theorem}

\begin{remark}\lab{rem:irred-cored-exceptional}
The proofs above 
and in \ref{AppA} show that if an irreducible representation $(V,G)$ of an adjoint exceptional group $G$ is not coreduced, then $V$ has a bad slice. 
\end{remark}

\bigskip
\section{Coreduced Representations of the Classical Groups}\lab{sec:cored-classical}

In this section we classify the coreduced representations $V$ of the {\it simple adjoint groups\/} of classical type. 
If $G$ is adjoint and simply laced, i.e., of type $\AA$ or $\DD$, then a reducible representation $V$ is not coreduced by Proposition~\ref{prop:rootsmult2}, and so the maximal coreduced representations are all irreducible. We will see that this is also true  for $G$ of type $\CC_{n}$, $n\geq 3$, but not for type $\BB_{n}$.

The case of $\SLtwo$ has been settled in Theorem~\ref{thm:SLtwo} even without assuming that the center acts trivially. So we may assume that the rank  of $G$ is at least $2$.

\begin{theorem}\lab{thm:cored.simple}
Let $G$ be a simple classical group of rank at least $2$. Then, up to automorphisms, the following representations are the maximal coreduced representations of $G/Z(G)$.
\begin{enumerate}
\item $G=\AA_n$, $n\geq 2$: $\Ad \AA_{n}=\phi_{1}\phi_{n}$, $\phi_2^2(\AA_3)$, $\phi_1^3(\AA_2)$;
\item $G=\BB_n$, $n\geq 2$: $\Ad \BB_{n}=\phi_{2}(\BB_{n})$ $(\phi_2^2$ if $n$=$2)$,  $\phi_1^2(\BB_{n})$,  $n\phi_1(\BB_{n})$;
\item $G=\CC_n$, $n\geq 3$: $\Ad\CC_{n}=\phi_{1}^{2}(\CC_{n})$, $\phi_2(\CC_{n})$, $\phi_4(\CC_4)$;
\item $G=\DD_n$, $n\geq 4$: $\Ad \DD_{n}=\phi_{2}(\DD_{n})$, $\phi_1^2(\DD_{n})$;
\end{enumerate}
\end{theorem}

In section~\ref{sec:cofreerep} we showed that every irreducible cofree representation of a simple group is coreduced. Looking at the list above and the one in Theorem~\ref{thm:coredexceptional} we see that we have the following partial converse.

\begin{corollary}\lab{cor:adjoint}
Let $G$ be a simple adjoint group and $V$ an irreducible representation of $G$. 
Then $V$ is coreduced if and only if  $V$ is cofree. 
\end{corollary}

We start with   type $\AA_{n}$, $n\geq 2$. Recall that $\phi_{p}:=\bigwedge^{p}\C^{n+1}$, $p=1,\ldots, n$. 

\begin{lemma}\lab{lem:multSL}
Consider the representations $\phi_p$ and $\phi_q$ of $\SL_{n+1}$ where $1\leq p\leq q\leq n$ and $n\geq 2$. Then there is a nonzero weight of $\phi_p\phi_q$ of multiplicity $\geq 2$ except in the cases
\begin{enumerate}
\item $\phi_1^2$ or $\phi_{n}^2$,
\item $\phi_1\phi_{n}$,
\item $\phi_2^2(\SL_4)$,
\end{enumerate}
where the zero weight has multiplicity greater than one in (2) and (3).
\end{lemma}

\begin{proof}
It is easy to calculate that the weight $2\epsilon_1+\dots+2\epsilon_{p-1}+\epsilon_p+\dots+\epsilon_{q+1}$ occurs in $\phi_p\otimes\phi_q$ with multiplicity $q-p+2$ and that it occurs in $\phi_{p-1}\otimes\phi_{q+1}$ once. Since $\phi_p\otimes \phi_q=\phi_p\phi_q+\phi_{p-1}\otimes\phi_{q+1}$ we see that our weight occurs with multiplicity $q-p+1$ in $\phi_p\phi_q$. This gives us a nonzero weight of multiplicity at least two, except in the following two cases:
\begin{enumerate}
\item $\phi_1\phi_{n}$ where the above weight is the zero weight, and
\item $\phi_p^2$ where  $1\leq p\leq n$.
\end{enumerate}
However, in the second case, we can suppose, by duality, that $2p\leq n+1$. If $2p\leq n$, then one sees as above that $\epsilon_1+\dots+\epsilon_{2p}$ occurs with multiplicity $\frac{1}{p}\binom {2p}{p-1}$ which is at least $2$ as long as $p>1$. If $2p=n+1$, then one computes that $\epsilon_{1}-\epsilon_{2}=2\epsilon_{1}+\epsilon_{3}+\cdots+\epsilon_{2p}$ occurs with multiplicity $\frac{1}{p-1}\binom{2p-2}{p-2}$ which is $\geq 2$ as long as $p> 2$. Thus the only 
possibilities are $\phi_1^2$ and $\phi_2^2(\SL_4)$. 
\end{proof}

The next lemma was proved by \name{Stembridge}. We give a slightly different version of his proof.

\begin{lemma}\lab{lem:stembridge}
Let $\phi$ be an irreducible representation of $\PSL_{n+1}$, $n\geq 2$. Then the roots of $G$ occur with multiplicity at least two in $\phi$, except in the following cases.
\begin{enumerate}
\item The adjoint representation $\phi_{1}\phi_{n}$;
\item $\phi_{1}^{k(n+1)}(\SL_{n+1})$ or its dual, $k=1,2,\ldots$;
\item $\phi_2^2(\SL_4)=\phi_1^2(\DD_3)$.
\end{enumerate}
\end{lemma}

\begin{proof}
The representation $\phi$ has highest weight $\lambda=\sum_i\lambda_i\omega_i$ where the $\omega_i$ are the fundamental dominant weights and $\sum_i i\lambda_i$ is a multiple of $n$. Now, Criterion \ref{critA} together with Lemma~\ref{lem:multSL} above implies that the only irreducible representations of $\PSL_{n+1}$ containing the roots with multiplicity one are those listed.
\end{proof}

\begin{proposition}\lab{prop:PSLn}
Let $n\geq2$. The nontrivial irreducible coreduced representations of $\PSL_{n+1}$ are the adjoint representation $\phi_{1}\phi_{n}$, $\phi_{2}^{2}(\SL_{4})$, $\phi_{1}^{3}(\SL_{3})$ and $\phi_{2}^{3}(\SL_{3})$. All other irreducible representations  admit a bad toral slice.
\end{proposition}
 
 \begin{proof}
 By Proposition~\ref{prop:rootsmult2} we know that the only candidates for coreduced irreducible representations of $\PSL_{n}$ are those listed in Lemma~\ref{lem:stembridge} above. So it remains to show that $S^{km}(\C^{m})$ is not coreduced for $m>3$ and for $m=3,k>2$.  For $m\geq 4$ the slice representation at a generic fixed point of the maximal torus $T$ contains the weights $\beta_{i}:=km\epsilon_i$ and the weight $\alpha:=-k(2\epsilon_{1}+2\epsilon_{2}+\epsilon_{3}+\cdots+\epsilon_{m-2})$ of the monomial $(x_{3}\cdots x_{m-2}x_{m-1}^{2}x_{m}^{2})^{k}$ which satisfy the indecomposable relation $m\alpha+2\beta_{1}+2\beta_{2}+\beta_{3}+\cdots+\beta_{m-2}=0$, and so the slice representation is not coreduced. 
 
 For $m=3$ and $k>1$ we have the weights $\beta_{i}:=3k\epsilon_{i}$ and the weight $\alpha:=-3(k-1)\epsilon_{1}-3(k-2)\epsilon_{2}$ of the monomial $x_{2}^{3}x_{3}^{3(k-1)}$ which satisfy the indecomposable relation $k\alpha+(k-1)\beta_{1}+(k-2)\beta_{2}= 0$. Again it follows that the slice representation is not coreduced.  
 \end{proof}

Now we look at type $\BB_n$. 

\begin{proposition}\lab{prop5}
Let $G=\SO_{2n+1}$ be the adjoint group of type $\BB_n$, $n\geq 2$. Then the only nontrivial irreducible coreduced representations are the adjoint representation $\phi_{2}$, the standard representation $\phi_1$ and $\phi_1^2$.
All other irreducible representations  admit a bad toral slice.

The representations $\phi_{2}$ and $\phi_{1}^{2}$ are maximal coreduced whereas $k\phi_{1}$ is coreduced if and only if $k\leq n$.
\end{proposition}

\begin{proof}
The highest weights of irreducible representations of $G$ are just sums of the highest weights $\omega_{1},\ldots,\omega_{n-1}, 2\omega_{n}$ of the representations $W_\ell:=\bigwedge^\ell(\C^{2n+1})$ for $1\leq \ell\leq n$. For $\ell=2m+1$ or $2m$, $m\geq 2$, one can compute that the weights of $W_\ell$ contain the roots of $G$ with multiplicity $\binom  {n-2}{m-1}$ which is at least $2$. Thus $W_\ell$ admits a non-coreduced slice representation of a maximal torus and is therefore not coreduced for $\ell\geq 4$. For $\ell=3$, hence $n\geq 3$, we have the weights $\pm \epsilon_i\pm\epsilon_j\pm\epsilon_k$ where $1\leq i< j<k\leq n$ as well as the weights $\pm\epsilon_i$, $1\leq i\leq n$ where the latter have multiplicity $\geq 2$. Now the relation 
\begin{equation}\label{eq:badweights}
(\epsilon_1+\epsilon_2+\epsilon_3)+(-\epsilon_1+\epsilon_2+\epsilon_3)+2(-\epsilon_2)+2(-\epsilon_3)=0
\end{equation}
is indecomposable and so the slice representation of the maximal torus is not coreduced. 

Now let $V$ be an irreducible representation of $G$ with highest weight $\lambda=\sum_{i}m_{i}\omega_{i}$. 
If $m_{i}>0$ for some $i\geq 3$ then, by Criterion \ref{critA}, $V$ has a non-coreduced slice representation of a maximal torus, and thus is not coreduced. 

Hence we are left with $\lambda=r\omega_{1}+s\omega_{2}$ where $s$ is even in case $n=2$. Let us first assume that $n>2$. Since $\phi_{2}^{2}$ contains the roots with multiplicity $\geq \dim W_{2}^{T} = n\geq 3$ and since $\phi_{1}\phi_{2}$ contains the indecomposable weight relation $ (2\epsilon_1+\epsilon_2)+2(-\epsilon_1)+(-\epsilon_2)
= 0$ and the short roots occur with multiplicity $>1$, we are reduced to the highest weights $r\omega_{1}$. If $r\geq 3$
 we have the roots and the weights $3\epsilon_1$ and $-2\epsilon_1$ which lead one to see that the slice representation is not coreduced. 

The arguments in the case $n=2$ are the same (one has to replace $\phi_2$ by $\phi_2^2$ everywhere).

Finally, we have to look at direct sums of $\phi_{1}$, $\phi_{2}$ and $\phi_{1}^{2}$. We will see in Theorem~\ref{CIT.thm}(2) that  $k\phi_{1}$ is coreduced if and only if $k\leq n$. Since $\phi_{2}$ and $\phi_{1}^{2}$ contain all roots it remains to show that $\phi_{1}^{2}+\phi_{1}$ and $\phi_{2}+\phi_{1}$ are not coreduced. First consider $\phi_1^2$, $n\geq 4$. The subgroup $\SO_3\times\SO_{2n-2}$ is maximal in $\SO_{2n+1}$, it has rank $n$ and has slice representation  
$\phi_1^4(\AA_1) \oplus\phi_1^2(\DD_{n-1})+\theta_{1}$. If we add a copy of $\phi_1(\BB_n)$, then we have a  subrepresentation $(\phi_1^4+\phi_1^2,\AA_1)$ which is not coreduced. 
The details work out similarly for $n=2$ and $n=3$.  We are left with $\Ad G+\phi_{1}$.   The slice representation   of the group $\SO_3\times(\SO_2)^{n-1}$ contains two copies of the standard representation of $\SO_3$ on $\C^3$  which is not coreduced (Theorem~\ref{thm:SLtwo}). Hence $\Ad G+\phi_1$ is not coreduced.
\end{proof}

For  type $\DD_{n}$ we get the following result. Recall that only irreducible representations of $\PSO_{2n}$ can be coreduced (Proposition \ref{prop:rootsmult2}).

\begin{proposition}\lab{prop6}
Let $G=\PSO_{2n}$ be the adjoint group of type $\DD_{n}$, $n\geq 4$. Then the only nontrivial coreduced representations are  the adjoint representation $\phi_{2}$, $\phi_1^2$, $\phi_{3}^{2}(\DD_{4})$ and $\phi_{4}^{2}(\DD_{4})$, and these are maximal coreduced. All other irreducible representations  admit a bad toral slice.
\end{proposition}
  
\begin{proof}
The highest weights of representations of $\SO_{2n}$ are just sums of the highest weights $\omega_{1},\ldots,\omega_{n-2},\omega_{n-1}+\omega_{n}$ of the representations  $W_\ell:=\bigwedge^\ell(\C^{2n})$ for $1\leq \ell\leq n-1$ and twice the highest weights $\omega_{n-1}$ and $\omega_n$ of the two half-spin representations. Moreover,   $\phi_{n-1}^{2}\oplus\phi_{n}^2\simeq W_n:=\bigwedge^n(\C^{2n})$. 
  
The representations $W_{2m}$ for $m>1$ contain the roots of $G$ with multiplicity $\binom {n-2}{m-1}\geq 2$. The representations $W_k$ for $k$ odd, $k>1$, have no zero weights, but they contain the weights of $\phi_1$ with multiplicity greater than one. Hence the Cartan products $W_kW_\ell$ for $k,\ell\leq n-1$ odd, $k+\ell\geq4$,   contain  the adjoint representation more than once, so that the representations are not coreduced. We already know that $\phi_1^2$ is   coreduced and by Criterion \ref{critA} no power $\phi_1^{2k}$ is coreduced for $k\geq 2$.

It remains to consider those representations $\phi$ of $\PSO_{2n}$ which are Cartan products with $\phi_{n-1}^2$ or $\phi_n^2$. If $n\geq 6$ is even   then both contain the roots at least three times, hence $\phi$ is not coreduced.     If $n=4$, then $\phi_3^2$ and $\phi_4^2$ are outer isomorphic to $\phi_1^2$ which is  coreduced.  If $\phi$ is not exactly one of these representations, then it is not coreduced by Criterion \ref{critA}.
If $n$ is odd then  $\phi_{n-1}^{2}$ and $\phi_{n}^{2}$ both contain the weights of $W_1$ at least three times and so $\phi$ contains the roots with multiplicity at least 3 and is not coreduced.
\end{proof}
  
For type $\CC_n$ we will use the following lemma

\begin{lemma}\lab{lem:HiHj}
Let $H_1,\dots,H_4$ be copies of $\SL_2$ and let $V_i\simeq\C^2$ have the standard action of $H_i$. Let $H=\prod_i H_i$ and $V=\bigoplus_{i<j} V_{ij}$ where $V_{ij}=V_i\otimes V_j$. Then $(V,H)$ is not coreduced.
\end{lemma}

\begin{proof}
Consider the subrepresentation $V':=V_{12}\oplus V_{14}\oplus V_{23}\oplus V_{34}\oplus V_{24}$. We have the quotient mapping (by $H_{1}$) from $V_{12}\oplus V_{14}$ to   $V'_{24}\oplus\theta_2$ where $V_{24}'$ is another copy  of $V_{24}$.  The image is a hypersurface $F$ defined by an equation saying that the invariant of $(V'_{24},H_2\times H_4)$ is the product of the coordinate functions on $\theta_2$. By Lemmas \ref{lem:invariant.ideal} and \ref{lem:HnormalizingG} (see Examples \ref{ex:invariant.hypersurface} and \ref{ex:HnormalG})   the representation $V_{24}'\oplus V_{23}\oplus V_{34}\oplus V_{24}\oplus\theta_2$ of $H_2\times H_3\times H_4$ is coreduced if 
$V'$ is coreduced. Quotienting by the action of $H_3$ we similarly obtain   a representation $(V_{24}'\oplus V_{24}''\oplus V_{24}\oplus \theta_4,H_2\times H_4)\simeq (3\C^4\oplus\theta_4,\SO_4)$ which is  not coreduced (Example~\ref{ex:3D2}). Hence $(V',H)$ and $(V,H)$ are not coreduced.
\end{proof}

The fundamental representations $\phi_i$ of $\CC_n$ are given by $\phi_{1}=\C^{2n}$, $\phi_{2}=\bigwedge^{2}\C^{2n}/\C\beta$, and $\phi_{i}=\bigwedge^{i}\C^{2n}/\beta\wedge\bigwedge^{i-2}\C^{2n}$ for $i=3,\ldots,n$ where $\beta\in\bigwedge^{2}\C^{2n}$ is the invariant form. They can be realized  as the irreducible subspaces $\bigwedge_{0}^i(\C^{2n})\subset \bigwedge^i(\C^{2n})$ of highest weight $\omega_{i}:=\eps_{1}+\cdots+\eps_{i}$. The generators of the representations of the adjoint group $G = \PSp_{2n}$ are the $\phi_i$ for $i$ even and the $\phi_i\phi_j$ for $i$ and $j$ odd.

\begin{proposition}\lab{prop:Cn}
Let $G=\PSp_{2n}$ be the adjoint group of type $\CC_n$, $n\geq 3$. Then the nontrivial irreducible coreduced representations of $G$ are the adjoint representation $\phi_1^2$,  $\phi_2$ and $\phi_4(\CC_4)$, and these are all maximal. Moreover, all other irreducible representations admit a bad slice.
\end{proposition}

\begin{proof}
(a) First consider the case of $\phi_i\phi_j$ 
where $i$ and $j$ are odd. We may suppose that $j\geq 3$. Then $\phi_j$ contains the weight $\epsilon_1+\epsilon_2+\epsilon_3$ (it is a dominant weight which is the highest weight of $\phi_j$ minus a sum of positive roots). By the action of the Weyl group we have all the weights $\pm\epsilon_1\pm\epsilon_2 \pm\epsilon_3$. In $\phi_i$ (and $\phi_j$) we similarly have all the weights $\pm\epsilon_k$. Thus $\phi_i\phi_j$ contains the roots $2\epsilon_1$ and $\epsilon_1-\epsilon_2$, hence all the roots. Moreover, we have the following indecomposable relation of weights in 
$\phi_{i}\phi_{j}$ (none of which are roots):

\begin{multline}\label{eq:badrelation}
(2\epsilon_1+\epsilon_2+\epsilon_3)+(2\epsilon_1-\epsilon_2-\epsilon_3)+\\
2(-\epsilon_1+2\epsilon_2+\epsilon_3)+2(-\epsilon_1-2\epsilon_2-\epsilon_3) = 0
\end{multline}
Hence $\phi_i\phi_j$ has a bad toral slice and is therefore not coreduced. The same holds for every Cartan product of $\phi_i\phi_j$ with any other representation of $G$. 

Now $\phi_1^4$ is a representation of $G$, but since $\phi_1^2$ contains the trivial representation $n$ times, $\phi_1^4$ contains the adjoint representation at least $n$ times, hence has a bad toral slice and is not coreduced. Therefore, the adjoint representation $\phi_1^2$ is the only coreduced irreducible representation $\phi=\phi_{i_{1}}\phi_{i_{2}}\cdots\phi_{i_{m}}$ of $G$ where at least one $i_{k}$ is odd.
\par\smallskip
(b) Now we consider representations $\phi_{2i}$, $2i\leq n$. These representations, one can show as above, contain the short roots of $G$. But   the long roots do not occur. Hence the connected component of the isotropy group at a generic zero weight vector is covered by a product $H:=\prod_{j=1}^n H_j$ where each $H_j$ is the copy of $\SL_2$ in $G$ corresponding to the positive long root $2\epsilon_j$. If $n\geq 5$ and $2i\geq 4$, then the slice representation
contains the subrepresentation
$$
\bigoplus_{1\leq j<k\leq n} V_{jk}\quad\text{where}\quad V_{jk} := (\C^{2}\otimes \C^{2},H_{j}\times H_{k}). 
$$
which is not coreduced (Lemma \ref{lem:HiHj}).
Finally, one easily sees that any product $\phi_{2i}\phi_{2j}$ contains all the roots  as well as the zero sum of weights given above in 
equation~\eqref{eq:badrelation}. This includes the case where a factor is $\phi_2$ or $\phi_4$. Hence the irreducible coreduced representations of $G$ are as claimed.
\par\smallskip
(c) It remains to show that the coreduced representations of $G$ are all irreducible. As seen above, the connected component of the isotropy group at a generic zero weight vector of $\phi_{2}$ is covered by a product $H:=\prod_{j=1}^n H_j$ where each $H_j$ is the copy of $\SL_2$ in $G$ corresponding to the positive long root $2\epsilon_j$. If we add another copy of $\phi_{2}$ or the adjoint representation $\phi_{1}^{2}$, then the slice representation contains $\bigoplus_{1\leq j<k\leq n} V_{jk}$ where $V_{jk} := (\C^{2}\otimes \C^{2},H_{j}\times H_{k})$, which is not coreduced for $n\geq 4$. The same holds if $n=4$ and we add a copy of $\phi_{4}(\CC_{4})$. This proves the claim for $n\geq 4$, because $\phi_{1}^{2}$ and $\phi_{4}(\CC_{4})$ contain all roots. For $\phi_{2}(\CC_{3})+\phi_{1}^{2}(\CC_{3})$ we have the slice representation of $H=H_1\times H_2\times H_3$ on  $V_{12}\oplus V_{13}\oplus V_{23}\oplus\lieh_1\oplus\lieh_2\oplus\lieh_3\oplus\theta_2                                                                                                                                                                                                                                                                                                                                                                                       $ where the $H_i$ are copies of $\SL_2$ and the $V_{ij}$ are as above. Consider the subrepresentation $(V',H_1\times H_2):=(V_{12}\oplus\lieh_1\oplus \lieh_2,H_1\times H_2)$. The principal isotropy group of $\lieh_1$ is $\C^*\times H_2$ where $\C^*$ acts on $V_{12}$ with weights $\pm 1$. Let $\lieh_2'$ denote a second copy of $\lieh_2$. Then the quotient of $V_{12}$ by $\C^*$ is a quadratic hypersurface in $\lieh_2'+\theta_1$ which equates the quadratic invariant of $\lieh_2'$ and the square of the coordinate function on $\theta_1$. Thus, as in Lemma \ref{lem:HiHj},  the fact that the representation $(\lieh_2+\lieh_2'+\theta_1,H_2)$  is  not coreduced (Example~\ref{ex:2c3}) implies that $V'$ is not coreduced, hence  neither is $\phi_{2}(\CC_{3})+\phi_{1}^{2}(\CC_{3})$.
Finally, $2\phi_{2}(\CC_{3})$ is not coreduced as we have seen in Example~\ref{ex:2c3}.
\end{proof}
\begin{remark}\lab{rem:irred-cored-classical}
The proofs above show that if an irreducible representation $(V,G)$ of an  adjoint classical group $G$ is not coreduced, then $V$ has a bad slice. We have already seen in the previous section that the same holds for the exceptional groups (Remark~\ref{rem:irred-cored-exceptional}).
\end{remark}

\bigskip
\section{Irreducible Coreduced Representations of Semisimple Groups}\lab{sec:irreducible-semisimple}

In this section we classify the irreducible coreduced representations of adjoint semisimple groups. 

\begin{example} The representation $(\C^{2n+1}\otimes\C^{2m+1},\SO_{2n+1}\times\SO_{2m+1})$  is the isotropy representation of a symmetric space. (Consider the automorphism $\theta$ of $\SO_{2(n+m+1)}$ given by conjugation with 
$\left[\begin{smallmatrix}\id_{2n+1}&\\&-\id_{2m+1}\end{smallmatrix}\right]$.) It now follows from
\cite[Theorem 14, p. 758]{KoRa1971Orbits-and-represe} that this representation is coreduced for all $n,m\geq 1$.
\end{example}

\begin{example} The representation $(V,G\times H)=(\C^{3} \otimes\phi_{1}(\Gtwo), \SO_{3}\times\Gtwo)$ is coreduced.
In  fact, $(V,H)$ is cofree and the quotient $\quot VH$ is the $\SO_{3}$-module $\phi_{1}^{4}\oplus \theta_{2}$ which is  cofree and coreduced. Hence $(V,G\times H)$ is cofree, too.  Now the proper nontrivial slice representations of $(3\phi_1,\Gtwo)$ are $(2\C^3+2(\C^3)^*+\theta_3,\SL_3)$ (coreduced by Theorem \ref{CIT.thm}) and $(2\C^2+\theta_6,\SL_2)$ (coreduced by Theorem \ref{thm:SLtwo}). Thus every fiber of $\pi\colon V\to\quot VH$ is reduced, except for the zero fiber, which has codimension 7. Thus the null cone of $(V,G\times H)$, which has codimension 4, is reduced.
\end{example}

Surprisingly, these two examples are the only irreducible coreduced representations besides those where $G$ is simple. 

\begin{theorem}\lab{thm:adjoint-semisimple} 
The coreduced irreducible representations of a semisimple non-simple adjoint group are 
$$
(\phi_{1}(\BB_{n})\otimes\phi_{1}(\BB_{m}), \BB_{n}\times\BB_{m})\quad \text{and}\quad 
(\phi_{1}^{2}(\AA_{1})\otimes\phi_{1}(\Gtwo), \AA_{1}\times\Gtwo).
$$
\end{theorem}

The proof needs some preparation. We first construct a list of non-coreduced representations which will help to rule out most candidates.

\begin{example}\lab{ex:onson}
Let $(V,G)=(\C^n\otimes\C^m+\C^n,\SO_n\times  \SO_m)$ where $m, n\geq 2$.  We show that $V$ is not coreduced.  
There are three cases. Recall that $(\Sym^2(\C^n)\oplus\C^n,\SO_n)$ is not coreduced even for $n=2$. 
\begin{enumerate}
\item $n< m$. Quotienting by the action of $\SO_m$ we obtain $\Sym^2(\C^n)\oplus\C^n$ which is not coreduced, hence neither is $V$.
\item  $n= m$. By Example \ref{ex:quotientSOn} the representation $V$ is not coreduced since quotienting by $\Orth_m$ we obtain the non-coreduced representation $\Sym^2(\C^n)\oplus\C^n$.
\item  $n>m$. We have at most $n$ copies of $\C^n$, so by Example \ref{ex:quotientSOn} we may quotient by the action of $\Orth_n$ to arrive at the representation $\Sym^2(\C^m)\oplus\C^m$ which is not coreduced.   Hence $V$ is not coreduced.
\end{enumerate}
\end{example}

We have seen in Lemma~\ref{lem:slice.tensor} that for two representations $(V,G)$ and $(W,H)$ with a zero weight, if $(V, G)$ has a bad slice, then so does $(V\otimes W,G\times H)$. Together with  Remarks~\ref{rem:irred-cored-exceptional} and \ref{rem:irred-cored-classical} this implies that  we need only consider tensor products of the irreducible coreduced representations $(V,G)$ of the simple adjoint groups. They fall into five types.
\begin{enumerate}
\item $(V,G)=\phi_1^2(\AA_1)=(\C^3,\SO_3)$.
\item $(V,G)=\phi_1^4(\AA_1)$ or there is a slice representation $(W,H)$ where $H^0=T$ is a maximal torus of $G$ (of rank at least 2) and $W$ contains weight spaces of roots $\alpha$, $\beta$ and $-(\alpha+\beta)$ or $W$ contains $\theta_2$ and weight spaces $\pm\alpha$.
\item $(V,G)=\phi_1(\BB_n)$, $n\geq 2$.
\item $(V,G)= \phi_1(\Gtwo)$. 
\item $(V,G)=\phi_4(\FF_4)$, $\phi_4(\CC_4)$, or $\phi_2(\CC_n)$, $n\geq 3$.
\end{enumerate}
Note that the representations $\phi_1^2(\DD_n)$, $n\geq 3$ and $\phi_1^2(\BB_n)$, $n\geq 2$, are of type (2) as is the representation $\phi_1^3(\AA_2)$.
We consider tensor products of the various types of representations.

\begin{lemma}
Let $(V_1,G_1)$ be of type (2) and $(V_2,G_2)$ of arbitrary type. Then $(V_1\otimes V_2,G_1\times G_2)$ has a bad slice.
\end{lemma}
 
\begin{proof}   We leave the case that  $(V_1,G_1)$ or $(V_2,G_2)$ is $\phi_1^4(\AA_1)$ to the reader. It will be clear from our techniques what to do in this case. Let $T_1$ be a maximal torus of $G_1$ fixing $v_1$. First assume that  the weights of the slice representation at $v_1$ contain roots $\alpha$, $\beta$ and $-\alpha-\beta$.   Let $T_2$ be a maximal torus of $G_2$. Suppose first that $(V_2,G_2)=(\C^3,\SO_3)$,  
Let  $v_2\in V_2$ be a zero weight vector  and  let $\gamma$ be a nonzero weight of $(V_2,T_2)$.    Then by Corollary \ref{cor:slice.tensor} the slice representation of $T_1\times T_2$ at $v_1\otimes v_2$ contains the weights 
$$
- \gamma+\alpha,\ \gamma+\beta,  \gamma-\alpha-\beta,\ \text{ and } -\gamma-\alpha-\beta.
$$
We have the minimal zero sum
$$
2(- \gamma+\alpha)+2(\gamma+\beta)+( \gamma-\alpha-\beta)+(-\gamma-\alpha-\beta)=0,
$$
hence the slice representation of $T_1\times T_2$ is not coreduced.    The same argument works in case  $(V_2,G_2)$ is of type (2). Now suppose that $(V_2,G_2)=\phi_1(\BB_n)$, $n\geq 2$. Then we have a slice representation of $\SO_{2n}\times T_{1}$ containing the irreducible components $\C^{2n}\otimes\C_{\alpha}$, $\C^{2n}\otimes\C_{\beta}$ and $\C^{2n}\otimes\C_{-\alpha-\beta}$. Quotienting by $\SO_{2n}$ we obtain a representation of $T_1$ with weights 
$$
2\alpha,\ 2\beta,\ \alpha+\beta,\ -\alpha, \ -\beta \text { and } -2\alpha-2\beta.
$$
Hence the slice representation is not coreduced. The same argument works in case $(V_2,G_2)$ is of type (5). For type (4) we get a slice representation of $\SL_3\times T_{1}$ containing
$$
\C^3\otimes\C_{\alpha},\  \C^3 \otimes\C_{\beta},  (\C^3)^*\otimes\C_{-\alpha-\beta}, \text { and }(\C^3)^*\otimes\C_{\alpha},
$$ 
and quotienting by $\SL_3$ we obtain a $T_1$-representation with weights $-\beta$, $-\alpha$, $2\alpha$ and $\alpha+\beta$. Hence we have a non-coreduced slice representation. 

Finally assume that the slice representation at $v_1$ contains $\theta_2$ and weights $\pm \alpha$ and that $(V_2,G_2)$ is  of arbitrary type. Let $\pm\gamma$ be  nonzero weights of $V_2$.  Because of the $\theta_2$, the slice representation at $v_1\otimes v_2$ contains the weights of $V_2$ (Corollary \ref{cor:slice.tensor}). Hence we have weights $\pm\alpha\pm\gamma$ and $\pm\gamma$.
and the minimal bad relation
$$
(\alpha+\gamma)+(-\alpha+\gamma)-2(\gamma)=0.
$$
Thus $(V_1\otimes V_2,G_1\times G_2)$ is not coreduced.
\end{proof}

We are left with type (1) and types (3--5).
\begin{lemma}
Suppose that $(V_1,G_1)$ is of type (1) or (3) or (5) and that $(V_2,G_2)$ is of type (5). Then $(V_1\otimes V_2,G_1\times G_2)$ has a bad slice.
\end{lemma}

\begin{proof} 
First assume that $(V_1,G_1)$ is $\phi_1(\BB_n)$, $n\geq 1$ (type (1) or type (3)). 
If $(V_2,G_2)$ is $\phi_4(\FF_4)$, then there is a (principal) slice representation of $\DD_4$ on $\theta_2$ where $(\phi_4(\FF_4),\DD_4)=(\phi_1+\phi_3+\phi_4+\theta_2)$ while $(V_1,G_1)$ has a slice representation of $\SO_{2n}$ on $\theta_1$ where $(V_1,\SO_{2n})=(\C^{2n}+\theta_1,\SO_{2n})$. By Corollary \ref{cor:slice.tensor} there is a subrepresentation of a slice representation of $(V_1\otimes V_2,G_1\times G_2)$ of the form $(\C^{2n}\otimes \C^8\oplus\C^{2n},\SO_{2n}\times\SO_8)$. It follows from Example \ref{ex:onson} that the slice representation is not coreduced.
  
If $(V_2,G_2)$ is $\phi_2(\CC_m)$, $m\geq 4$, then there is a  slice representation  
$
(W,H)=(\theta_2+\phi_2,\SL_2\times\SL_2\times\CC_{m-2})
$ 
where $(\phi_2(\CC_m),H)$ contains $(\C^2\otimes\C^2,\SL_2\times\SL_2)\simeq(\C^4,\SO_4)$.  There is a non-coreduced subrepresentation of the slice representation of $(V_1\otimes V_2,G_1\times G_2)$ of the form $(\C^{2n}\otimes\C^4\oplus\C^{2n},\SO_{2n}\times\SO_4)$. The case of $\phi_2(\CC_3)$ is only notationally different and the case of $\phi_4(\CC_4)$ is similar. Finally, if $(V_1,G_1)$ is of type (5), then   the same techniques   produce a non-coreduced slice representation at a zero weight vector.
\end{proof}

We leave the proof of the following to the reader.
\begin{lemma}
A tensor product $(V_1\otimes V_2,G_1\times G_2)$ has a bad slice if $(V_1,G_1)$ is $\phi_1(\Gtwo)$ (type (4)) and $(V_2,G_2)$ is of type (5).
\end{lemma}

We are now left with the problem of tensor products of representations of types (1), (3) and (4). First we handle types (1) and (3).

\begin{proposition}
Let $3\leq 2k+1\leq 2m+1\leq 2n+1$ and
$$
(V,G)=(\C^{2n+1}\otimes\C^{2m+1}\otimes\C^{2k+1},\SO_{2n+1}\times\SO_{2m+1}\times\SO_{2k+1}).
$$ 
Then the  slice representation at the zero weight vector is not coreduced. 
\end{proposition} 

\begin{proof}
The slice representation at the zero weight vector is  
\begin{multline*}
(W,H)=(\C^{2n}\otimes\C^{2m}\otimes\C^{2k}+\C^{2n}\otimes\C^{2m}+\\
+\C^{2m}\otimes\C^{2k}+\C^{2n}\otimes\C^{2k},\SO_{2n}\times\SO_{2m}\times\SO_{2k}).
\end{multline*}
If $k>1$, consider the subrepresentation $\C^{2m}\otimes\C^{2k}\oplus \C^{2n}\otimes\C^{2k}$. Quotienting by $\SO_{2m}\times\SO_{2n}$ 
we get  $(2\Sym^2(\C^{2k}),\SO_{2k})$ which is not coreduced.  

Now assume that $k=1$ but $m>1$. We have a subrepresentation
$$
\C^{2n}\otimes\C^{2m}\oplus\C^{2m}\otimes \C_{\nu} \oplus\C^{2m}\otimes\C_{- \nu}
$$
where  the $\C_{\pm\nu}$ are irreducible representations of $\SO_{2k}\simeq\C^*$ of weight $\pm 1$. Quotienting by $\Orth_{2n}$ 
we obtain the representation
$$
(\Sym^2(\C^{2m})\oplus\C^{2m}\otimes \C_{\nu} \oplus\C^{2m}\otimes\C_{-\nu},\SO_{2m}\times\SO_2).
$$
Let $\pm\eps_1,\dots,\pm\eps_m$ be the  weights of $\C^{2m}$ for the action of the maximal torus $T$ of $\SO_{2m}$. 
Then the slice representation of $\Sym^2(\C^{2m})$ at a generic zero weight vector is, up to trivial factors, the sum of the $\C_{\pm 2\eps_i}$. Hence we have a slice representation of $T\times\SO_2$ containing
$$
\C_{-2\epsilon_1}\oplus \C_{-2\epsilon_2}\oplus(\C_{\epsilon_1}\otimes\C_{\nu})\oplus(\C_{\epsilon_2}\otimes\C_{-\nu}).
$$
This last representation is not coreduced.

Now assume that $n\geq m=k=1$.  We rename the weight $\epsilon_1$ of $\SO_{2m}=\SO_{2}$ to be  just $\epsilon$. Then we have the subrepresentation
$$
(\C^{2n}\otimes\C_{\epsilon})\oplus(\C^{2n}\otimes\C_{\nu})\oplus(\C_{-\epsilon}\otimes\C_{-\nu}).
$$
Quotienting by $\Orth_{2n}$ we get a representation 
$$
\C_{2\epsilon}\oplus \C_{2\nu}\oplus (\C_{\eps}\otimes\C_{\nu})\oplus(\C_{-\epsilon}\otimes\C_{-\nu})
$$
of $\Cst\times\Cst$ which is not coreduced.
\end{proof}

\begin{proposition}
Let $(V,G)=(\C^{2n+1}\otimes\C^7,\SO_{2n+1}\times\Gtwo)$, $n\geq 2$, or $(\C^3\otimes\C^3\otimes\C^7,\SO_3\times\SO_3\times\Gtwo)$. Then $(V,G)$ has a bad slice.
\end{proposition}

\begin{proof}
We leave the latter case to the reader. In the former case we have the slice representation (minus the trivial factor)
$$
(W,H)=(\C^{2n}\otimes(\C^3\oplus(\C^3)^*),\SO_{2n}\times\SL_3).
$$
If $n\geq 3$, then quotienting by $\Orth_{2n}$ we obtain the representation 
$$
(\Sym^2(\C^3)\oplus\Sym^2({\C^3}^*)\oplus\C^3\otimes{\C^3}^*,\SL_3)
$$ 
which is not coreduced. 

We are left with the case $(W,H)=(\C^4\otimes(\C^3\oplus{\C^3}^*),\SO_4\times\SL_3)$.
Consider a 1-parameter subgroup  $\rho$ of $\SO_4\times\SL_3$ whose action on $\C^4$ has weights $\pm 1$ and on $\C^3$ has weights $2, 0, -2$. Then  $Z_\rho$, 
the span of the positive weight vectors,
has dimension 12 (which is not surprising since $(W,H)$ is self-dual of dimension 24). Note that $Z_\rho$ is in the null cone and is stable under a Borel subgroup of $H$. Now one can show that the dimension of $U^-Z_\rho$ is $17= 12+\dim U^{-}$, the maximal possible, where $U^-$ is the maximal unipotent subgroup of $H$ opposite $B$. Hence $HZ_\rho$ is a component of the null cone (see section~\ref{sec:comp-nullcone} for more details). 

The positive weights of $\rho$ on $W$ are 1 and 3  and the negative weights are $-1$ and $-3$. 
This implies that the differential of an invariant of degree $>4$ vanishes on $Z_{\rho}$, hence on $HZ_{\rho}$. But we have only 4 generating invariants in degree at most 4, and so the null cone is not reduced along $HZ_{\rho}$, because $\codim HZ_\rho=7$.
\end{proof}

We are left with the case $\Gtwo\times\Gtwo$ acting on $\C^7\otimes\C^7$.

\begin{proposition}\lab{prop:g2g2}
The representation $(\C^7\otimes\C^7,\Gtwo\times\Gtwo)$ is not coreduced. 
\end{proposition}
We have two proofs of this, and both need some computations. They are given in 
\ref{AppB}. 

\bigskip
\section{Classical Invariants}\lab{sec:CIT}
Classical Invariant Theory describes the invariants of copies of the standard representations of the classical groups, e.g., the $\GL(V)$- or $\SL(V)$-invariants of $mV \oplus nV^{*}$ or the $\Sp(V)$-invariants of $mV$ where $mV := V^{m\oplus}$ denotes the direct sum of $m$ copies of $V$. In this context we will prove the following theorem.

\begin{theorem}\lab{CIT.thm}
\be
\item The representation $(pV \oplus qV^{*},\GL(V))$ is
coreduced for all $p,q \geq 0$. The null cone is irreducible if and only if $p+q\leq n$.
\item The representation $(pV \oplus qV^{*},\SL(V))$ 
is coreduced for all $p,q \geq 0$. The null cone is irreducible in the following cases: $p+q\leq n$ or $(p,q)=(n,1)$ or $(p,q)=(1,n)$. 
\item The representations $(mV,\Sp(V))$ are coreduced for  all $m\geq 0$, and the null cone is irreducible and normal.
\item The representations $(mV,\Orth(V))$, $(mV,\SO(V))$ are core\-duced if and only if $2m\leq \dim V$. The null cone is irreducible and normal for $2m<\dim V$.
\ee
\end{theorem}

The basic reference for this section is \cite{Sc1987On-classical-invar}. Denote by $T_{m}$, $B_{m}$ and $U_{m}$ the subgroups of $\GL_{m}$ consisting of diagonal, upper triangular, and upper triangular unipotent matrices. If $\lambda$ is a dominant weight, i.e., $\lambda =\sum_{i=1}^{m}\lambda_{i}\eps_{i} \in X(T_{n})=\bigoplus_{i=1}^{m} \Z\eps_{i}$ and $\lambda_{1}\geq \lambda_{2}\geq\cdots\geq \lambda_{m}$, we denote by $\psi_{\lambda}$ or $\psi_{\lambda}(m)$ the corresponding irreducible representation of $\GL_{m}$. In the following, we will only deal with {\it polynomial} representations of $\GL_{m}$, so that $\lambda_{i}\geq 0$ for all $i$.
Set $|\lambda|:=\sum\lambda_{i}$ and define the {\it height} of a dominant weight by $\htt(\lambda) :=  \max\{i\mid\lambda_{i}>0\}$. 

The famous \name{Cauchy} formula describes the decomposition of the symmetric powers of a tensor product where we consider $\psi_{\lambda}(m) \otimes \psi_{\mu}(k)$ as a representation of $\GL_{m}\times\GL_{k}$ (see \cite[(1.9) Theorem]{Sc1987On-classical-invar}).

\begin{proposition}\lab{Cauchy.prop}
\[
S^{d}(\C^{m}\otimes\C^{k}) = \bigoplus_{|\lambda|=d,\,\htt(\lambda)\leq\min\{m,k\}}
\psi_{\lambda}(m)\otimes\psi_{\lambda}(k)
\]
\end{proposition}
If $\lambda$ is a dominant weight of height $r$, then $\psi_{\lambda}(m)$ makes sense for any $m\geq r$. In fact, $\psi_{\lambda}$ is a functor and $\psi_{\lambda}(V)$ is a well-defined $\GL(V)$-module for every vector space $V$ of dimension $\geq r$. In particular, if $\rho\colon G\to\GL(V)$ is a representation of a reductive group $G$, then all $\psi_{\lambda}(V)$ for $\htt(\lambda)\leq\dim V$ are representations of $G$ as well. From the \name{Cauchy} formula we thus get
\[
\OOO(mV)_{d}=S^{d}(\C^{m}\otimes V^{*}) = \bigoplus_{|\lambda|=d,\,\htt(\lambda)\leq\min(m,\dim V)}
\psi_{\lambda}(m)\otimes\psi_{\lambda}(V^{*})
\]
as a representation of $\GL_{m}\times G$. Taking $U_{m}$-invariants we find
\[
\OOO(mV)^{U_{m}}_{d}=S^{d}(\C^{m}\otimes V^{*})^{U_{m}} = \bigoplus_{|\lambda|=d,\,\htt(\lambda)\leq\min(m,\dim V)}
\psi_{\lambda}(V^{*})\tag{$*$}
\]
where the torus $T_{m}\subset \GL_{m}$ acts on $\psi_{\lambda}(V^{*})$ with weight $\lambda$. Thus the algebra $\OOO(mV)^{U_{m}}$ is $\Z^{m}$-graded, and the homogeneous component of weight $\lambda$ is the $G$-module $\psi_{\lambda}(V^{*})$. In particular, $\OOO(mV)^{U_{m}}$ is multiplicity-free as a $\GL(V)$-module. It follows that the product $\psi_{\lambda}(V^{*})\cdot \psi_{\mu}(V^{*})$ in $\OOO(mV)$ is equal to $\psi_{\lambda+\mu}(V^{*})$. This leads to the following definition.

\begin{definition}
Let $G$ be a connected reductive group and let $A$ be a $G$-algebra, i.e., a commutative $\C$-algebra with a locally finite and rational action of $G$ by algebra automorphisms. Two simple submodules  $U,V \subset A$ are called {\it orthogonal\/} if the product $U\cdot V \subset A$ is either zero or simple and isomorphic to the Cartan (highest weight) component of $U \otimes V$.
\end{definition}

The result above can therefore be expressed by saying that all irreducible $\GL(V)$-submodules of $\OOO(mV)^{U_{m}}$ are orthogonal to each other. The following crucial result is due to \name{Brion} \cite[Lemme 4.1]{Br1985Representations-ex}.

\begin{proposition}\lab{Brion.prop}
Let $A$ be a $G$-algebra and let $V_{1},V_{2},W \subset A$ be simple submodules. Assume that $V_{1},V_{2}$ are both orthogonal to $W$. Then any simple factor of $V_{1}\cdot V_{2}$ is orthogonal to $W$.
\end{proposition}

We will also need the following result about $U$-invariants (see \cite[III.3.3]{Kr1984Geometrische-Metho}).
\begin{proposition}\lab{Uinv.prop}
Let $G$ be a connected reductive group, $U\subset G$ a maximal unipotent subgroup, and let $A$ be a finitely generated $G$-algebra. Then $A$ is reduced, resp. irreducible, resp. normal if and only if $A^{U}$ is reduced, resp. irreducible, resp. normal.
\end{proposition}

Another consequence of formula $(*)$ is that $\OOO(mV)^{U_{m}} = \OOO(nV)^{U_{n}}$ for all $m\geq n=\dim V$. 
\par\smallskip
We start with the groups $\GL(V)$ and $\SL(V)$ acting on $W:=pV\oplus qV^{*}$. It is known that the $\GL(V)$-invariants are generated by the bilinear forms 
\[
f_{ij}\colon (v_{1},\ldots,v_{p},\xi_{1},\ldots,\xi_{q})\mapsto \xi_{j}(v_{i}).
\] 
If $V_i^{*}$ is the $i$th copy of $V^{*}$ in $W^{*}\subset \OOO(W)$ and $V_j$ the $j$th copy of $V$, then $V_i^*\cdot V_j = \sll(V)\oplus \C f_{ij}$ in $\OOO(W)$, and so $V_i^{*}$ and $V_j$ are orthogonal in $\OOO(W)/I$ where $I$ is the ideal generated by the invariants $f_{ij}$. It follows from Proposition~\ref{Brion.prop} above that all simple submodules of $\OOO(pV)$ are orthogonal to all simple submodules in $\OOO(qV^{*})$ modulo $I$. Thus the $\GL(V)$-homomorphism 
\[
\OOO(pV)^{U_{p}}\otimes \OOO(qV^{*})^{U_{q}} \to (\OOO(pV\oplus qV^{*})/I)^{U_{p}\times U_{q}}
\]
is surjective, and the same holds if we take invariants under $U := U_{p}\times U_{V} \times U_{q} \subset \GL_{p}\times \GL(V) \times \GL_{q}$ where $U_{V}\subset \GL(V)$ is a maximal unipotent subgroup. This also shows that 
the $(U_{p}\times U_{q})$-invariants do not change   once  $p\geq n$ or $q\geq n$, so that we can assume that $p,q\leq n$.

Now we have $\OOO(pV)^{U} = \C[x_{1},\ldots,x_{p}]$ where $x_{i}\in\bigwedge^{i}V^{*}$ is a highest weight vector. Similarly, $\OOO(qV^{*})^{U} = \C[y_{1},\ldots,y_{q}]$, and thus we get a surjective homomorphism
\[
\phi\colon \C[x_{1},\ldots,x_{p},y_{1},\ldots,y_{q}] \to (\OOO(pV\oplus qV^{*})/I)^{U}. \tag{$**$}
\]
\begin{proof}[Proof of Theorem~\ref{CIT.thm}(1) and (2)]
We claim that the kernel of $\phi$ is generated by the products $x_{r}y_{s}$ where $r+s>n$. This implies that
we have an isomorphism
\[
(\OOO(pV\oplus qV^{*})/I)^{U}\simeq \C[x_{1},\ldots,x_{p},y_{1},\ldots,y_{q}]/(x_{i}y_{j}\mid i+j>n),
\]
and so $\NN_{p,q}$ is reduced, by Proposition~\ref{Uinv.prop}. We also see that the ideal $(x_{i}y_{j}\mid i+j>n)$ is prime if and only if it is $(0)$, i.e., when $p+q\leq n$. This proves the theorem for $\GL(V)$.

To prove the claim we first remark that the kernel of $\phi$ is spanned by monomials, because $\phi$ is equivariant under the action of the  maximal torus $T_{p}\times T_{q}$.  Moreover, it is not difficult to see that $\phi(x_{r}y_{s})=0$ if $r+s>n$, see \cite[ Remark~1.23(2)]{Sc1987On-classical-invar}.

Now let $f:=x_{i_{1}}\cdots x_{i_{p}}y_{j_{1}}\cdots y_{j_{q}}$ be a monomial which is not in the ideal $(x_{i}y_{j}\mid i+j>n)$. Then $r+s \leq n$ where $r := \max(p_{i})$ and $s:=\max(q_{j})$. If $(e_{1},\ldots,e_{n})$ is a basis of $V$ and $(e_{1}^{*},\ldots,e_{n}^{*})$ the dual basis of $V^{*}$, then we can assume that $x_{i}= e_{n-i+1}^{*}\wedge e_{n-i+2}^{*}\wedge\cdots\wedge e_{n}^{*}$ and $y_{j}=e_{1}\wedge_{2}\wedge\cdots\wedge e_{j}$. Now it is clear that the monomial $f$ does not vanish at the point $w:=(0,\ldots,0,e_{n-r+1},\ldots,e_{n},e_{1}^{*},\ldots,e_{s}^{*},0,\ldots,0)$ which is in the null cone $\NN_{p,q}:= \NN(pV\oplus qV^{*})$.
\par\smallskip
For the group $\SL(V)$ there are more invariants, namely the determinants 
$d_{i_{1}\cdots i_{n}}:=\det \begin{bmatrix} v_{i_{1}}\\ \vdots \\ v_{i_{n}}\end{bmatrix}$ where $i_{1}<i_{2}<\cdots<i_{n}$,
and $d^{*}_{j_{1}\cdots j_{n}}:=\det \begin{bmatrix} \xi_{j_{1}}\\ \vdots \\ \xi_{j_{n}}\end{bmatrix}$ where $j_{1}<j_{2}<\cdots<j_{n}$. These invariants only appear if $p\geq n$, resp. $q\geq n$. In particular, we have the same invariants and the same null cone in case $p,q < n$. From the surjectivity of the map $\phi$ in $(**)$ above we see that there remain only the cases where either $p=n$ and $q\leq n$, or $q=n$ and $p\leq n$. Let $J$ denote  the ideal generated by the $\SL(V)$-invariants. Then $J^{U}=I^{U}+(x_{n})$ if $p=n>q$, $J^{U}=I^{U}+(y_{n})$ if $p<n=q$, and $J^{U}=I^{U}+ (x_{n},y_{n})$ if $p=n=q$. Hence
\[
(\OOO(pV\oplus qV^{*})/J)^{U}\simeq \C[x_{1},\ldots,x_{p'},y_{1},\ldots,y_{q'}]/(x_{i}y_{j}\mid i+j>n)
\]
where $p':=\min(p,n-1)$ and $q':=\min(q,n-1)$. The rest of the proof is as above.
\end{proof}

Next we study the case where $V$ is a symplectic space, i.e., $V$ is equipped with a non-degenerate skew form $\beta$, $\dim V = 2n$. Then the invariants of $mV$ are generated by the bilinear maps 
\[
\beta_{ij}\colon (v_{1},\ldots,v_{n})\mapsto \beta(v_{i},v_{j}), \ 1\leq i < j \leq m.
\]
We denote by $\psi_{k}:=\bigwedge^{k}_{0}V^{*}\subset \bigwedge^{k}V^{*}$ ($k=1,\ldots,n$)  the fundamental representations of $\Sp(V)$ where 
$\bigwedge^{k}V^{*} = \bigwedge^{k}_{0}V^{*} \oplus \beta\wedge\bigwedge^{k-2}V^{*}$. We know from equation $(*)$ that $\OOO(mV)_{k}^{U_{m}}$ contains a unique copy of $\bigwedge^{k}V^{*}$ for $k\leq \min(m,n)$.
\begin{lemma}
Let $I \subset \OOO(mV)$ be the ideal generated by the invariants $\beta_{ij}$. Then in $\OOO(mV)^{U_{m}}$ we have
\be
\item $\bigwedge^{k}V^{*} = \psi_{k}\pmod{I}$ for $k=1,\ldots,\min(m,n)$;
\item $\bigwedge^{k}V^{*} \cdot \bigwedge^{\ell}V^{*} = \psi_{k}\psi_{\ell} \pmod{I}$ for $1\leq k\leq \ell\leq \min(m,n)$.
\ee
\end{lemma}
\begin{proof}
Part (1) is clear since $\psi_{k+2} = \bigwedge^{k+2} V^{*} / \beta \bigwedge^k V^{*}$. For part (2) let $x_1,\dots ,x_n\in V^*$ correspond to the positive weights $\epsilon_1,\dots,\epsilon_n$ and let $y_1,\dots,y_n$ correspond to the $-\epsilon_j$. A simple submodule  occurring in $\psi_k\cdot \psi_\ell$ has a highest weight vector containing a unique term $\gamma:=x_1\wedge\cdots\wedge x_k\cdot \alpha$ where $\alpha$ is an $\ell$-fold wedge product of a certain number of $x_i$ and $y_j$. But the only possibility for obtaining a highest weight of $\Sp(V)$  is  
$\alpha=x_1\wedge\cdots\wedge x_q \wedge y_{k-r+1}\wedge\cdots\wedge y_k$ where $q\leq \ell$ and $r=\ell-q$.  This gives the highest weight  of a (unique) copy of $\psi_p\psi_q$  where $p=k-r$.  

Suppose that $r>0$. We have an element $\beta_r$  in $(\bigwedge^r(V^*)\otimes\bigwedge^r(V^*))^{\Sp(V)}$ where $\beta_r(v_1\wedge\cdots\wedge v_r,w_1\wedge\cdots\wedge w_r)=\det(\beta(v_i,w_j))$. Here the $v_i$ and $w_j$ are elements of $V$. It is easy to see that $\beta_r$  projects to a nontrivial invariant element $\beta_r'$ of $\psi_r\psi_r$, and that $\beta'_r\in I^r$. Then the product of $\beta_r'$ with $\psi_p\psi_q\subset \psi_p\cdot\psi_q$ is a copy of $\psi_p\psi_q$ in  $\psi_k\psi_\ell $, and we have (2).
\end{proof}

\begin{proof}[Proof of Theorem~\ref{CIT.thm}(3)]
It follows from the lemma above and Proposition~\ref{Brion.prop} that all simple submodules in $\OOO(mV)^{U_{m}}$ are orthogonal and the covariants are generated by $\psi_{1},\cdots,\psi_{m'}$ where $m':=\min(m,n)$.
Let $U_{V}\subset \Sp(V)$ be a maximal unipotent subgroup and let $x_{k}\in\bigwedge^{k}_{0}V^{*}\subset \OOO(mV)_{k}^{U_{m}}$ be a highest weight vector. Then we have a surjective homomorphism
\[
\phi\colon \C[x_{1},\ldots,x_{m'}]\to (\OOO(mV)/I)^{U_{m}\times U_{V}}.
\]
If $W\subset V$ is a maximal isotropic subspace, then $W^{\oplus m}$ belongs to the null cone of $mV$, and for a suitable choice of $W$ the function $x_{k}$ does not vanish on $W^{\oplus m}$ for $k\leq m'$. This implies that $\phi$ is an isomorphism,  because the grading of the action of $T_{m}$ has one-dimensional weight spaces and so the kernel of $\phi$ is linearly spanned by monomials. Now the theorem for $\Sp_{n}$ follows from Proposition~\ref{Uinv.prop}. 
\end{proof}

Finally, let $V$ be a quadratic space, i.e., an $n$-dimen\-sional vector space with a non-degenerate quadratic form $q$. The $\Orth(V)$-invariants of $mV$ are generated by the bilinear maps
\[
q_{ij}\colon (v_{1},\ldots,v_{m})\mapsto q(v_{i},v_{j}), \ 1\leq i \leq j \leq m.
\]
The $\SO(V)$-modules $\psi_{k}:=\bigwedge^{k}V^{*}$  are simple if $2k<n$. For $n=2m$ $\psi_{m}:=\bigwedge^{m}V^{*}$ is simple as an $\Orth(V)$-module, but decomposes as $\psi_{m} = \psi_{m}^{+}\oplus\psi_{m}^{-}$ as an $\SO(V)$-module.  

\begin{lemma} Let $2m\leq n$ and 
let $I \subset \OOO(mV)$ be the ideal generated by the invariants $q_{ij}$. Then in $\OOO(mV)^{U_{m}}$ we have
\be
\item
$\psi_{k} \cdot \psi_{\ell} = \psi_{k}\psi_{\ell} \pmod{I}$ for $1\leq k\leq \ell\leq \min(m,\frac{n-1}{2})$;
\item
If $n=2m$, then $\psi_{m}^{+}\cdot\psi_{m}^{-} = 0 \pmod{I}$.
\ee
\end{lemma}

\begin{proof}
Let $n=2s$ or $2s+1$ so that $m\leq s$. We consider a weight basis $x_1,\dots,x_s$ and $y_1,\dots,y_s$ (and a zero weight element $z$ if $n$ is odd). First suppose that $n$ is even.
For (1) we can then proceed as in the symplectic case.  The only difference is that we use the invariant bilinear form $q$ to generate an element   $q_r'$ lying in  $(\psi_r\psi_r)^{\SO(V)}$ and in $I^r$. As for (2), the highest weight vectors are $x_1\wedge\cdots\wedge x_m$ and $x_1\wedge\cdots \wedge x_{m-1}\wedge y_m$.  Their product   is the image of $q_1'  \otimes \psi_{m-1}\psi_{m-1}$ in $\psi_m^+\psi_m^-$. The argument of (1) shows that any other irreducible occurring in $\psi^+_m\cdot\psi^-_m$ also lies in $I$.

Now suppose that $n$ is odd. Then the argument for (1) above goes through except when the zero weight vector  appears in the expression for $\alpha$. 
So suppose that $\alpha=x_1\wedge\cdots\wedge x_{\ell-1}\wedge z$.
Then
$$
x_1\wedge\cdots\wedge x_k\cdot \alpha+(x_1\wedge\cdots\wedge x_{\ell-1}\wedge x_{\ell+1}\cdots\wedge x_k\wedge z)\cdot (x_1\wedge\cdots\wedge x_\ell)
$$ 
is a vector in $\psi_k\cdot\psi_\ell$. It is obtained from $(x_1\wedge\cdots\wedge x_k)\cdot (x_1\wedge\cdots \wedge x_\ell)$ by  applying elements of $U^-$. Hence we don't have a new irreducible component of $\psi_k\cdot \psi_\ell$.
\end{proof}

\begin{proof}[Proof of Theorem~\ref{CIT.thm}(4)]
Choose highest weight vectors $x_{k}\in \bigwedge^{k}V^{*}$ for $2k<n$ and $x_{m}^{+}\in \psi_{m}^{+}$, $x_{m}^{-}\in \psi_{m}^{-}$ for $n=2m$.
The lemma implies that the induced maps
\begin{gather*}
\C[x_{1},\ldots,x_{m}]\to (\OOO(mV)/I)^{U_{m}\times U_{V}} \text{ for } 2m<n, \text{ and}\\
\C[x_{1},\ldots,x_{m}^{+},x_{m}^{-}]/(x_{m}^{+}x_{m}^{-}) \to (\OOO(mV)/I)^{U_{m}\times U_{V}} \text{ for } 2m=n
\end{gather*}
are surjective. The weights $\nu(x_{k})$ of the highest weight vectors (with respect to $T_{m}\times T_{V}$, $T_{V}$ a maximal torus of $\SO(V)$) are linearly independent, except that in case $n=2m$ we have $\nu(x_{m}^{+}) + \nu(x_{m}^{-}) = 2\nu(x_{m-1})$. It follows that the algebras on the left hand side are multiplicity free, and so the kernels of the two maps are spanned by monomials. But it is easy to see that none of the $x_{k}$, $x_{m}^{\pm}$ vanish on the null cone, and so the two maps are isomorphisms. 
Again using Proposition~\ref{Uinv.prop} we obtain the theorem for  the groups $\Orth(V)$ and $\SO(V)$ in the case where $2m\leq n$.

It remains to 
show that the null cone  is not reduced for $2m >n$. Let $n=2k$ where $k>m$. The case $n=2k-1$ is similar and will be left to the reader. Then in degree $k+1$ we find the submodule $M:=\bigwedge^{k+1} \C^{k+1} \otimes \bigwedge^{k+1}V^*$, by \name{Cauchy}'s formula (Proposition~\ref{Cauchy.prop}). The $\SO(V)$-module $\bigwedge^{k+1}V^*$ is simple and isomorphic to $\psi_{k-1} = \bigwedge^{k-1}V$. We claim that $M$ vanishes on the null cone $\NN$, but is not contained in the ideal $I$ generated by the invariants.

The first part is clear, because $\NN = \Orth(V) \cdot (k+1)W$ where $W \subset V$ is a maximal isotropic subspace, and every function $f_1\wedge \cdots \wedge f_{k+1}$ vanishes on $(k+1)W$, because $\dim W = k$.

For the second, we remark that the module $\psi_{k-1}$ appears the first time in degree $k-1$, in the form $\bigwedge^{k-1} \C^{k+1} \otimes \bigwedge^{k-1} V^*$. If $M \subset I$, then $M$ must belong to the product
\[
\OOO((m+1)V)^{\SO(V)}_2 \cdot (\textstyle\bigwedge^{m-1} \C^{m+1} \otimes \bigwedge^{m-1} V^*)
\]
which is a quotient of $(S^2(\C^{m+1})\otimes \bigwedge^{m-1} \C^{m+1}) \otimes \bigwedge^{m-1} V^*$. But the tensor product $S^2(\C^{m+1})\otimes \bigwedge^{m-1} \C^{m+1}$ does not contain the "determinant" $\bigwedge^{m+1} \C^{m+1}$ as a $\GL_{m+1}$ module.
\end{proof}

\bigskip
\section{Non-Reduced Components of the Null Cone}\lab{sec:comp-nullcone}

We need some information about null cones (see  \cite{KrWa2006On-the-nullcone-of} for more details).
Let $G$ be a connected reductive complex group, $T \subset G$ a maximal torus and $V$ a $G$-module. Let  $X(T) =\Hom(T,\C^*)$ denote the character group of $T$ and let $Y(T)=\Hom(\C^*,T)$ denote the group of 1-parameter subgroups of $T$. Then $Y(T)$ and $X(T)$ are dually paired: $\langle \rho,\mu\rangle = n$ if $\mu(\rho(t)) = t^{n}$. 
For any $\rho\in Y(T)$ we set 
$$
Z_{\rho}:=\{v \in V \mid \lim_{t\to 0} \rho(t) v = 0\} = \bigoplus_{\mu\in X(T), \langle\rho,\mu\rangle>0}V_{\mu}
$$
where $V_{\mu}\subset V$ denotes the weight space of weight $\mu$. These $Z_{\rho}$ are called {\it positive weight spaces}.
Then the Hilbert-Mumford theorem says that $\NN$ is the union of the sets $GZ_\rho$, $\rho\in Y(T)$. In fact, one needs only a finite number of elements of $Y(T)$. Pick a system of simple roots for $G$. Then using the action of the Weyl group, we can assume that any given $\rho$ is positive when paired with the simple roots $\alpha_1,\dots,\alpha_\ell\in X(T)$, $\ell=\dim T$. In fact, we can always assume that the pairings are strictly positive and that $\rho$ only takes the value 0 on the zero weight.  We call such elements of $Y(T)$ \emph{generic}.  Now $Z_\rho$ is stable under the action of the Borel $B$, thus $GZ_\rho$ is  closed in the Zariski topology, and $GZ_\rho$ is   irreducible. Thus there are finitely many generic $\rho_i$ such that the sets $GZ_{\rho_i}$ are the irreducible components of $\NN$.

\begin{remark}\lab{rem:vanishingcov}
We will use this description of the null cone to show that a given homogeneous covariant $\tau\colon V \to W$ of degree $d$ vanishes on the null cone, generalizing Lemma~\ref{lem:vanishingcov}. It suffices to show that $\tau$ vanishes on $Z_{\rho}$ for the relevant generic $\rho$'s. Denote by $\mu_{1},\ldots,\mu_{m}$ the weights of $Z_{\rho}$. 
 {\it If $\tau\neq 0$, then the highest weight $\mu$ of $W$ is of the form $\sum_{i}d_{i}\mu_{i}$ where $\sum_{i}d_{i}=d$.} 
(This follows from the $B$-equivariance of $\tau$.) \ %
Hence $\tau$ vanishes if $\mu$ cannot be expressed as such a sum.
\end{remark}

Let $\Lambda(V)$ denote the set of weights of $V$. For $\rho\in Y(T)$, let $\Lambda_{\rho}$ denote the subset of $\Lambda(V)$ of elements which pair  strictly positively with $\rho$. A subset $\Lambda \subset\Lambda(V)$ is called {\it admissible} if $\Lambda=\Lambda_{\rho}$ for a generic $\rho$. In this case  set $Z_\Lambda:=Z_\rho$.  We will often switch between looking at generic elements of $Y(T)$ (or $Y(T)\otimes\Q$) and corresponding subsets $\Lambda\subset \Lambda(V)$.
We say that an admissible $\Lambda$ is \emph{dominant\/} if   $GZ_\Lambda$ is a component of the null cone. 

Here is a way to show that the null cone $\NN$ is not reduced. 
 
\begin{proposition}\lab{prop:method1} 
Let $\Lambda\subset\Lambda(V)$ be dominant and let $W\subset V$ be a $T$-stable complement of $Z_{\Lambda}$. 
Assume that for any $z\in Z_{\Lambda}$ the differential $d\pi_{z}
$ restricted to $W$ has rank $< \codim_{V} GZ_{\Lambda}$, or,  equivalently; there is a subspace $W'\subset W$ of dimension $>\codim_{GZ_{\Lambda}}Z_{\Lambda}$ such that the differential  of any invariant vanishes on $W'$. 
Then no point of $GZ_\Lambda\subset\NN$ is reduced.
\end{proposition}

\begin{proof}
Either condition implies that the rank of $d\pi_z$ is less than the codimension of $GZ_\lambda$ for any $z\in Z_\lambda$.
\end{proof}

\begin{remark}\lab{rem:method1}
Let $(v_{1},\ldots,v_{n})$ be a basis of $V$ consisting of weight vectors of weight $\mu_{1},\ldots,\mu_{n}$, and let $(x_{1},\ldots,x_{n})$ be the dual basis. If $f=x_{i_{1}}\cdots x_{i_{d}}$ is a monomial of weight zero, then an $x_{i}$ such that $\mu_{i}\notin \Lambda$ has to appear. If two such $x_i$ appear in $f$, then clearly $(df)|_{Z_{\lambda}}=0$.  This gives our first method to show that $\NN$ is not reduced.
\be 
\item  Let  $\Lambda'$ be the complement of  $\Lambda$ in $\Lambda(V)$. Let   $d\in\N$ be minimal such that every zero weight monomial $f$ containing exactly one factor $x_{i}$ corresponding to a weight from $\Lambda'$ has degree $\leq d$.
\item Show that there are not enough invariants of degree $\leq d$, i.e., show that the number of invariants of degree $\leq d$ is strictly less than the codimension of $GZ_\lambda$.
\ee
\end{remark}

If $W$ is irreducible of highest weight $\lambda$ we denote by $\lambda^{*}$ the highest weight of the dual representation $W^{*}$. The next result will give us another way to see if the null cone is not reduced. It uses the method of covariants introduced in section~\ref{sec:covariants} (see Proposition~\ref{prop:gencovariants}).

\begin{proposition}\lab{prop:method2}
Let $\phi\colon V \to W$ be a covariant where $W$ is irreducible of highest weight $\lambda$. Let $\Lambda\subset \Lambda(V)$ be admissible and assume that $\phi$ does not vanish on $GZ_{\Lambda}$. Then $\lambda^{*}\in \N\Lambda$.
\end{proposition}

\begin{proof}
Let $W^*$ be the subspace of $\OOO(V)$ corresponding to $\phi$. Let $f$ be a highest weight vector of $W^*$. Then $f$ has weight $-\lambda^*$ and $f$ does not vanish on $GZ_{\Lambda}$ by assumption.  It follows that $f$ contains a monomial 
$m=x_{i_{1}}x_{i_{2}}\cdots x_{i_{d}}$ where the corresponding $v_{i_{k}}$ all belong to $Z_{\Lambda}$, i.e., $\lambda^{*}= \mu_{i_{1}}+\mu_{i_{2}}+ \cdots+\mu_{i_{d}}\in\N\Lambda$.
\end{proof}

\begin{remark}\lab{rem:method2}
This proposition will be used in the following way.
\be
\item Find a suitable highest weight $\lambda$ and an integer $d$ such that $\lambda^{*}$ cannot be written as a sum of more than $d$ weights from $\Lambda$.
\item Show that there are generating covariants of type $W_{\lambda}$ in degree $>d$.
\ee
By the proposition above this implies that the generating covariants from (2) vanish  on $GZ_{\Lambda}$. In order to apply Proposition~~\ref{prop:gencovariants} one has to  fix $d$ and check (1) for any admissible $\Lambda$.  
\end{remark}

We finish this section by giving some criteria to find the dominant $\Lambda$ among the admissible ones.
Let $\Lambda_1$ and $\Lambda_2$ be admissible subsets of $\Lambda(V)$. Set $Z_i:=Z_{\Lambda_i}$, $i=1,2$. We say that $\Lambda_2$ \emph{dominates} $\Lambda_1$, and we write $\Lambda_1\leq\Lambda_2$, if $GZ_1\subset GZ_2$. Given $\sigma\in W$, let $\Lambda_1^{(\sigma)}:=\{\lambda\in\Lambda_1\mid\sigma(\lambda)\in\Lambda_2\}$ and let $Z_1^{(\sigma)}$ denote the sum of the weight spaces with weights in $\Lambda_1^{(\sigma)}$.

\begin{lemma}\lab{lem:BZsigma}
Let  $\Lambda_1$ and  $\Lambda_2$ be admissible. Then $\Lambda_2$ dominates $\Lambda_1$ if and only if there is a $\sigma\in W$    
such that $BZ_1^{(\sigma)}$ is dense in $Z_1$.
\end{lemma}

\begin{proof}
Suppose that $\Lambda_1\leq\Lambda_2$. Let $z\in Z_1$. Then there is a $g\in G$ such that $gz\in Z_2$. Write $g=u\sigma b$ where $b\in B$, $u\in U$ and $\sigma\in W$ (Bruhat decomposition).  
Since $b$ and $u$ preserve the $Z_i$, we see that $bz\in Z_1^{(\sigma)}$.  Thus $Z_1$ is the union of the constructible subsets $BZ_1^{(\sigma)}$, $\sigma\in W$, and one of them must be dense.

Conversely, suppose that some $BZ_1^{(\sigma)}$ is dense in $Z_1$. Since $\sigma(BZ_1^{(\sigma)})$ lies in $GZ_2$ and $GZ_2$ is closed, we see that $GZ_1\subset GZ_2$.
\end{proof}

The condition that $BZ_1^{(\sigma)}$ is dense in $Z_1$ has some consequences for the weights of $Z_{1}^{(\sigma)}$. Denote by 
$\Phi^{+}$ the set of positive roots, i.e., the weights of $\bb:=\Lie B$.

\begin{lemma}\lab{lem:Bdense}
Let $Z$ be a $B$-module and $Z'\subset Z$ a $T$-stable subspace. If $BZ'$ is dense in $Z$, then
$$
\Lambda(Z) = \Lambda(Z')+(\Phi^{+}\cup\{0\}).
$$
In particular, $\Lambda(Z')$ contains the set $\Omega:=\{\lambda\in\Lambda(Z) \mid \lambda\notin \Lambda(Z)+\Phi^{*}\}$.
\end{lemma}

\begin{proof} The tangent map of $B \times Z' \to Z$ at a point $(e,z_{0})$ has the form $(X,v)\mapsto Xz_{0}+ v$, and so
$\bb Z' + Z' = Z$. If $z\in Z'$ is a weight vector of weight $\lambda$ then $\bb z \subset \bigoplus_{\omega\in\Phi_{+}\cup\{0\}} Z_{\lambda+\omega}$, hence  
$\Lambda(Z)$ is as claimed.
\end{proof}

\begin{proposition}\lab{prop:component}
Let $\Lambda_{1},\Lambda_{2}\subset \Lambda(V)$ be admissible subsets. Define $\Omega_1 := 
\{\lambda\in\Lambda_1\mid \lambda\not\in \Lambda_1+\Phi^{+}\}$ and suppose that $\Q_{\geq0}\Omega_{1}$ contains the simple roots.
Then $\Lambda_1\leq\Lambda_2$ implies that $\Lambda_1\subset\Lambda_2$.
\end{proposition}

\begin{proof}
Let $\sigma$ be as in the lemma. Then $\Lambda_1^{(\sigma)}$ contains $\Omega_1$ by Lemma~\ref{lem:Bdense}. This in turn implies that $\Lambda_2$ is positive on $\sigma(\alpha_j)$, $j=1,\dots,\ell$. Thus each $\sigma(\alpha_j)$ is a positive root and so 
$\sigma$ is the identity. Hence $\Omega_1\subset Z_2$ and thus $\Lambda_1\subset\Lambda_2$.
\end{proof}
 
\begin{corollary} \lab{cor:irredsl3} 
Suppose that $G=\SL_3$ with simple roots $\alpha$ and $\beta$.
Let $\Lambda=\Lambda_{\rho}\subset\Lambda(V)$ be admissible and maximal with respect to set inclusion. Suppose that $\Lambda$ contains nonzero weights of the form $\lambda_1:=-a\alpha+b\beta$ and $\lambda_2:=c\alpha-d\beta$ where the coefficients $a$, $b$, $c$ and $d$ are non-negative rational numbers. Then $\Lambda$ is dominant.
\end{corollary}

\begin{proof}
Let $\Omega\subset\Lambda$ be the minimal elements.   We may assume that $\lambda_1$ and $\lambda_2$ are in $\Omega$. Clearly $b,c\neq 0$. If $a=0$ or $d=0$, then the hypotheses of Proposition~\ref{prop:component} are satisfied. If $a,d\neq 0$, then 
$\langle\rho,\lambda_1\rangle>0$ and $\langle\rho,\lambda_2\rangle>0$ forces $bc-ad>0$. Thus the inverse of the matrix $\twomatrix c {-a}{-d}b$ has positive entries, so that the hypotheses of Proposition~\ref{prop:component} are satisfied and $\Lambda$ is dominant.
\end{proof}

See Example~\ref{ex:v31} below for a calculation of components of a null cone.

\bigskip
\section{Coreduced Representations of \texorpdfstring{$\SLthree$}{SL(3)}} \lab{sec:sl3}

In this section we classify the coreduced representations of $G=\SL_{3}$ (Theorems~\ref{thm:mainsl3.1} and \ref{thm:mainsl3.2}).

We denote the representation $V:=\phi_1^r\phi_2^s$ by $V[r,s]$, $r,s\in\N$, and its highest weight by $[r,s]$. We denote a weight $p\alpha+q\beta$ of a representation of $G$ by $(p,q)$ where $p$, $q\in(1/3)\Z$ and $p+q\in\Z$. Hence $\alpha=(1,0)=[2,-1]$, $\beta=(0,1) = [-1,2]$, and so 
$$
[r,s]=(\frac{2r+s}{3},\frac{r+2s}{3}) \text{ \ and \ } (p,q) = [2p-q,2q-p].
$$ 
Moreover, $[r,s]$ is in the root lattice if and only if $r\equiv s \mod 3$.

We leave the following lemma to the reader (see Lemma \ref{lem:domweights}).

\begin{lemma}\lab{lem:basics}
Let $V:=V[r,s]$ be an irreducible representation of $G$ and set $(p,q)=[r,s]$.  
\begin{enumerate}
\item The dominant weights of $V[r,s]$ are the weights $[r',s']$ obtained starting with $[r,s]$ and using  the following inductive process:  $[r',s']$ gives rise to  $[r'-2,s'+1]$ if $r'\geq 2$ and to  $[r'+1,s'-2]$ if $s'\geq 2$. Finally,  $[1,1]$ gives rise to  $[0,0]$. 
\newline
Equivalently, the dominant weights of $V[r,s]$ are those of the form $(k,l):=(p-a,q-b)$ where $a$, $b\in\N$, $0\leq k\leq 2l$ and $0\leq l\leq  2k$.
 \item The Weyl group orbit of the dominant weight $(k,l)$ is
\begin{enumerate}
\item $(k,l)$, $(l-k,l)$, $(k,k-l)$, $(l-k,-k)$, $(-l,k-l)$, $(-l,-k)$ if $k\neq 2l$ and $l\neq 2k$,
\item $(2l,l)$, $(-l,l)$, $(-l,-2l)$ if $k=2l$ and 
\item $(k,2k)$, $(k,-k)$ and $(-2k,-k)$ if $l=2k$.
\end{enumerate}
\item Let $(p,q)$ be dominant, $p\neq q$, and let $W\cdot(p,q)$ be the Weyl group orbit of $(p,q)$. Then
$$
\max\left\{\frac{-k}{\ell} \mid (k,\ell)\in W\cdot (p,q) \right\} = \frac{\min(p,q)}{|p-q|},
$$
$$
\min\left\{\frac{-k}{\ell} \mid (k,\ell)\in W\cdot (p,q),  \, \frac{-k}{\ell}>0\right\} = \frac{|p-q|}{\min(p,q)}.
$$
\end{enumerate}
\end{lemma}

Suppose that $\Lambda(V)$ is not contained in the root lattice. Then let $\Lambda_\alpha$ denote the weights $(p,q)$ of $V$ where $p>0$.   We define $\Lambda_\beta$ similarly. Note that $\Lambda_{\alpha}$ is stable under $\sigma_{\beta}$ and that $\Lambda_{\beta}$ is stable under $\sigma_{\alpha}$.

\begin{example}\lab{ex:v31}
Consider the module $V=V[3,1]$. Then the dominant weights are $[3,1]$, $[1,2]$, $[2,0]$ and $[0,1]$. Thus the weights of $V$ are
\begin{enumerate}
\item $(7/3,5/3)$, $(-2/3,5/3)$, $(7/3,2/3)$, $(-2/3,-7/3)$, $(-5/3,2/3)$, $(-5/3,-7/3)$ (the $W$-orbit of $[3,1]$);
\item $(4/3,5/3)$, $(1/3,5/3)$, $(4/3,-1/3)$, $(1/3,-4/3)$, $(-5/3,-1/3)$, $(-5/3,-4/3)$ (the $W$-orbit of $[1,2]$);
\item $(4/3,2/3)$, $(-2/3,2/3)$, $(-2/3,-4/3)$ (the $W$-orbit of $[2,0]$);
\item $(1/3,2/3)$, $(1/3,-1/3)$, $(-2/3,-1/3)$ (the $W$-orbit of $[0,1]$).
\end{enumerate}
Let $\rho\in Y(T)\otimes\Q$ be generic. We may assume that $\rho(\alpha)=1$ and,  
of course, we have $\rho(\beta)>0$. Then $\rho(\beta)$ has to avoid the values $2/5$, $5/2$, $4$, $1/4$ and $1$, so there are six cases to consider.
\smallskip

\begin{description}
\item[Case 1] 
Let $\Lambda_{1}$ correspond to $2/5<\rho(\beta)<1$. 
It is easy to see that $\Lambda_1$ is maximal. Then $\Lambda_1$ is dominant by Corollary~\ref{cor:irredsl3} since $(-2/3,5/3)$ and $(1/3,-1/3)$ are $\rho$-positive.													
\item[Case 2] 
Let $\Lambda$ correspond to $0<\rho(\beta)<1/4$.  Then $\Lambda=\Lambda_\alpha$ is $\sigma:=\sigma_\beta$-stable so that $\sigma(\Lambda^{(\sigma)})=\sigma(\Lambda\cap\Lambda_1)$. Now $\Lambda\cap\Lambda_1$ is $\Lambda\setminus\{(1/3,-4/3)\}$, hence  $\Lambda^{(\sigma)}$ is $\Lambda\setminus\{(1/3,5/3)\}$ where $(1/3,5/3)$ has multiplicity one. Thus $UZ_{\Lambda}^{(\sigma)}$ is dense in $Z_{\Lambda}$ so that $\Lambda<\Lambda_1$. (One can also see directly that $U^-Z_1$ has $Z_{\Lambda}$ in its closure.) Now it is easy to calculate that $\dim GZ_\Lambda<\dim GZ_1$, so that $\Lambda=\Lambda_{\alpha}$ is not dominant.
\item[Case 3]
Let $\Lambda$ correspond to $1/4<\rho(\beta)<2/5$. Then $\Lambda\subset\Lambda_1$. 
\item[Case 4] 
Let $\Lambda_2$ correspond to $5/2<\rho(\beta)<4$. Then 
$\Lambda_2$ is maximal and $(-5/3,2/3)$ and $(4/3,-1/3)$ are $\rho$-positive, so that $\Lambda_2$ is dominant by Corollary~\ref{cor:irredsl3}.
\item[Case 5] 
Let $\Lambda$ correspond to  $1<\rho(\beta)<5/2$. Then $\Lambda\subset \Lambda_1$.  
\item[Case 6] 
Let $\Lambda$ correspond to $\rho(\beta)>4$. Then $\Lambda=\Lambda_\beta$ and  as in Case 2 we see that $\Lambda<\Lambda_1$ and that $\Lambda$ is not dominant.
\end{description}
Thus there are only two components of the null cone, $GZ_{\Lambda_{1}}$ and $GZ_{\Lambda_{2}}$ corresponding to cases 1 and 4. Note that neither $\Lambda_\alpha$ nor $\Lambda_\beta$ is dominant.
\end{example}

Lemma \ref{lem:basics}   does not tell us anything about  multiplicities of  weights, but the following result gives us some lower bounds, which suffice for our uses. If $[r,s]$ is a weight of $V$, then we denote by $V_{[r,s]}\subset V$ the corresponding weight space.

\begin{lemma}\lab{lem:multiplicities}
Let $r=r_{0}+r'$ and $s=s_{0}+s'$ where $r'\equiv s'\mod 3$. Then every weight of $V[r_{0},s_{0}]$ occurs in $V[r,s]$ with multiplicity at least the dimension of the zero weight space $V[r',s']_{[0,0]}$.
\end{lemma}

\begin{proof}
This follows from the fact that $\OO(G/U)$ is a domain and that the product of the copies of $V[r_{0},s_{0}]$ and $V[r',s']$ in $\OO(G/U)$ is just the copy of $V[r,s]$
\end{proof}

\begin{example}
Consider $V[3,2]$. Then the multiplicity   of $[1,0]$   is at least the multiplicity of the zero weight in $V[2,2]$, which is $3$. The multiplicity of   $[2,1]$ is similarly seen to be at least 2. Thus the multiplicities of the dominant weights of $V[3,2]$ are at least as follows:  $[3,2]$, $[4,0]$, $[1,3]$ and $[0,2]$ with multiplicity one, $[2,1]$ with multiplicity two and $[1,0]$ with multiplicity three. In fact, these multiplicities are correct, except that $[0,2]$ actually has multiplicity two.
\end{example}

In Example~\ref{ex:v31} we have seen that neither $\Lambda_\alpha$ nor $\Lambda_\beta$ is dominant. But this is an exception as shown by the following result.

\begin{lemma} \lab{lem:lambdabeta} Let $V=V[r,s]$ where $r\geq s$.
\begin{enumerate}
\item If $r-s\equiv 1 \mod 3$, then $\Lambda_\beta$ is dominant.
\item If $r-s\equiv 2\mod 3$ and $[r,s]\neq [3,1]$ or $[5,0]$, then $\Lambda_\alpha$ is dominant.
\end{enumerate}
\end{lemma}

\begin{proof} For $t>0$ define $\rho_{t}\in Y(T)$ by $\rho_{t}(\alpha)=1$ and $\rho_{t}(\beta)=t$, and set $\Lambda_{t}:=\Lambda_{\rho_{t}}$. Define 
$$
\TT:=\{t>0\mid \rho_{t}(\lambda)=0 \text{ for some }\lambda\in\Lambda(V[r,s]), \lambda\neq 0\}
$$
We have $\TT = \{t_{1}, t_{2},\ldots, t_{m}\}$ where $0<t_{1}<t_{2}<\cdots<t_{m}$, so there are $m+1$ admissible subsets $\Lambda^{(i)}$, $i=0,\ldots,m$, defined by $\Lambda^{(i)}:=\Lambda_{t}$ for $t_{i}<t<t_{i+1}$,  where $t_{0}=0, t_{m+1}=\infty$. Clearly, $\Lambda^{(0)} =\Lambda_{\alpha}$, $\Lambda^{(m)}=\Lambda_{\beta}$,  and $\Lambda_{\alpha}$ (resp. $\Lambda_{\beta}$) is not maximal if and only if $\Lambda_{\alpha}\subset \Lambda^{(1)}$ (resp. $\Lambda_{\beta}\subset\Lambda^{(m-1)}$). Note that if $\rho_{t}((k,l))=0$ then $t=-k/l$.

\par\smallskip
(1) First suppose that $r-s\equiv 1\mod 3$ and let $(p,q)=[r,s]$. Then $[1,0]=(2/3,1/3)$ is a weight of $V$, 
and the $\alpha$-string through $[1,0]$ has the form
$$
\Sigma = \left((-q,1/3), (-q+1,1/3),\ldots,(2/3,1/3),\ldots,(q+1/3,1/3)\right)
$$
where $(-q,1/3)$ is in the $W$-orbit of $(q+1/3,q)$.
Note that $\#\Sigma = 2q+4/3$. Since the case $V=V[1,0]$ is obvious we can assume that $q\geq 4/3$, hence $\#\Sigma\geq 4$.
\par\smallskip
{\bf Claim 1:} {\it We have $t_{m}=3q$ and $t_{m-1}=3q-3$, and these values are attained at the first two weights $(-q,1/3)$ and $(-q+1,1/3)$ of the $\alpha$-string $\Sigma$. In particular, $\Lambda_{\beta}\supset\Lambda^{(m-1)}$ and $\#(\Sigma\cap\Lambda^{(m')})\leq \#\Sigma - 2$ for $m'\leq m-2$.}

This implies that $\Lambda_{\beta}$ is dominant. In fact, suppose that $\Lambda_\beta<\Lambda$ for some admissible $\Lambda$. Set $Z_\beta:=Z_{\Lambda_\beta}$. 
There is a $\sigma\in W$ such that $BZ_\beta^{(\sigma)}$ is dense in $Z_\beta$ and $\sigma( \Lambda_\beta^{(\sigma)})\subset\Lambda$ (Lemma~\ref{lem:BZsigma}). Clearly   $\Lambda_\beta^{(\sigma)}$ has to contain a subset $\Sigma'$ of the $\alpha$-string $\Sigma$ which omits at most one element and contains $(-q,1/3)$ (see Lemma~\ref{lem:Bdense}). Since $\Sigma'$ contains at least 3 elements  it is easy to see that $\sigma=e$ and $\sigma=\sigma_{\alpha}$ are the only elements from $W$ which send $\Sigma'$ to elements which have at least one  positive $\alpha$ or $\beta$  coefficient. Thus $\sigma(\Sigma')\subset \Lambda\cap\Sigma$. By the claim above this implies  that $\Lambda=\Lambda_{\beta}$ or $\Lambda=\Lambda^{(m-1)}$ and so $\Lambda\subset\Lambda_{\beta}$.

\par\smallskip
(2) Now suppose that $r-s\equiv 2\mod 3$. Then $[0,1] = (1/3,2/3)$ is a weight of $V$, and the $\beta$-string through $[0,1]$ has the form
$$
\Sigma = \left((1/3,-q+1/3), (1/3,-q+4/3),\ldots,(1/3,2/3),\ldots,(1/3,q)\right)
$$
where $(1/3,-q+1/3)$ is in the $W$-orbit of $(q-1/3,q)$.
Note that $\#\Sigma = 2q+4/3$.
\par\smallskip
{\bf Claim 2:} {\it If  $\#\Sigma\geq 6$ (i.e., $q\geq 8/3$), then $t_{1}=1/(3q-1)$ and $t_{2}=1/(3q-4)$, and these values are attained at the first two weights $(1/3,-q+1/3)$ and $(1/3,-q+4/3)$ of the $\beta$-string $\Sigma$. Moreover, $\Lambda_{\alpha}\supset\Lambda^{(1)}$ and $\#(\Sigma\cap\Lambda^{(m')})\leq \#\Sigma - 2$ for $m'\geq 2$.}

Now the same argument as above implies that $\Lambda_{\alpha}$ is dominant. Note that the condition $q\geq 8/3$ is satisfied for $[r,s]\neq [2,0]$, $[3,1]$ or $[5,0]$. For $V[2,0]$ there are only two admissible sets, $\Lambda_{\alpha}$ and $\Lambda_{\beta}$, both are dominant and $\NN=G Z_{\alpha}=G Z_{\beta}$.

\par\smallskip
(3) It remains to prove the two claims. Let $r-s\equiv 1\mod 3$.  We use the first formula given in Lemma~\ref{lem:basics}(3) for a dominant $(p',q')$:
$$
\mu_{(p',q')}:= \max\left\{\frac{-k}{\ell} \mid (k,\ell)\in W\cdot (p',q') \right\} = \frac{\min(p',q')}{|p'-q'|}
$$
By assumption we have $q\geq 4/3$. If $(p',q')\leq (p,q)$ is dominant, then 
$|p'-q'|\geq 1/3$. Thus 
$$
t_{m}=\max(\mu_{(p',q')}\mid (p',q') \text{ dominant}, (p',q')\leq (p,q) ) = \mu_{(q+1/3,q)}=3q,
$$ 
and this value is attained at a single weight of $V$, namely at $(-q,1/3)\in W\cdot (q+1/3,q)$. It follows that $t_{m-1}$ is either equal to $\mu_{(q-2/3,q-1)}=3(q-1)$ or equal to $\mu_{(p',q)}$ for a suitable $p'\leq p$, $p'\neq q+1/3$. But then $p'=q-2/3$ or $p'=q+4/3$ and in both cases we get $\mu_{(p',q)}\leq 3(q-1)$, because $q\geq 4/3$. Hence $t_{m-1}=3(q-1)$ and this value is attained at the weight $(-q+1,1/3)\in W\cdot (q-2/3,q-1)$. As a consequence, $\Lambda_{\beta}\supset \Lambda^{m-1}=\Lambda_{\beta}\setminus\{(-q,1/3)\}$, and $(-q,1/3), (-q+1,1/3)\notin \Lambda^{m'}$ for $m'\leq m-2$.
This proves Claim 1.
\par\smallskip
For $r-s\equiv 2\mod 3$ we use the second formula in Lemma~\ref{lem:basics}(3) for a dominant $(p',q')$:
$$
\nu_{(p'q')}:=\min\left\{\frac{-k}{\ell} \mid (k,\ell)\in W\cdot (p',q'),  \, \frac{-k}{\ell}>0\right\} = \frac{|p'-q'|}{\min(p',q')}.
$$
The minimal values of $|p'-q'|$ are $1/3$ and $2/3$ and they are attained at $(q'-1/3,q')$ and $(q'+2/3,q')$. Thus, for a fixed $q'$ the minimal values of $\nu_{(p',q')}$ are $1/(3q'-1)$ and $2/(3q')$. Since $q\geq 8/3>4/3$ we get 
$$
t_{1}=\min\left(\nu_{(p',q')}\mid (p',q')\leq (p,q) \text{ dominant}\right) = \nu_{(q-1/3,q)}=1/(3q-1),
$$
and this value is attained at a single weight, namely at $(1/3,-q+1/3)\in W\cdot (q-1/3,q)$. If follows that $t_{2}$ is either equal to $\nu(q+2/3,q) = 2/(3q)$ or equal to $\nu(q-4/3,q-1)=1/(3q-4)$. Since  $q\geq 8/3$ we get $3q-4=(3/2)q + ((3/2)q-4) \geq (3/2)q$. Hence $t_{2}=1/(3q-4)$ and this value is attained at $(1/3,-q+4/3)\in W\cdot (q-4/3,q-1)$. Now Claim~2 follows as above.
 \end{proof}

\begin{remark}\lab{rem:dimgzalpha}
Let $\Lambda=\Lambda_\alpha$ or $\Lambda_\beta$. Then $Z_\lambda$ is stabilized by a parabolic subgroup of codimension 2, hence $\codim_{GZ_\Lambda}Z_\Lambda\leq 2$.
\end{remark}

We need the following estimate on the dimension of $S^3(V)^G$:
  
\begin{proposition}\lab{prop:degree3invars} Let $r\geq s\geq 0$. Then 
\begin{enumerate}
\item The multiplicity of $[r-s,0]$ in $V[r,0]\otimes V[0,s]$ is $\binom {s+2}2$.
\item The multiplicity of $[r-s,0]$ in $V[r,s]$ is $s+1$.
\item The multiplicity of $V[s,r]$ in $V[r,s]\otimes V[r,s]$ is at most $s+1$.
\item The dimension of $S^3(V[r,s])^G$ is at most $s+1$, hence there are at most $s+1$ linearly independent cubic invariants of $V[r,s]$.
\end{enumerate}  
\end{proposition}
  
\begin{proof} Let $e_1$, $e_2$ and $e_3$ be the usual basis of $\C^3$ and let $f_1$, $f_2$, $f_3$ be the dual basis. Then
the weight vectors of weight  $[r-s,0]$ in $V[r,0]\otimes V[0,s]$ have basis the vectors  $e_1^{r-t}m\otimes f_1^{s-t}m^*$ where $0\leq t\leq s$ and $m$ is a monomial of degree $t$ in $e_2$ and $e_3$ and $m^*$ is the same monomial in $f_2$ and $f_3$. Thus the dimension of this weight space is $1+\dots+(s+1)$, giving (1). 
  
Part (2)  follows from the fact that $V[r,0]\otimes V[0,s]=V[r,s]\oplus V[r-1,0]\otimes V[0,s-1]$. This is an immediate consequence of Pieri's formula (see \cite[formula (10.2.2) in 9.10.2]{Pr2007Lie-groups}). 
  
The multiplicity of $V[s,r]$ in $V[r,s]\otimes V[r,s]$ is bounded by the multiplicity of the weight  $[r,s]-(r,s)$ in $V[r,s]$ since $[r,s]+([r,s]-(r,s))=[s,r]$. Now $[r,s]-(r,s)=1/3(-r+s,r-s)$ which is in the $W$-orbit of $1/3(2r-2s,r-s)=[r-s,0]$. Thus (2) implies (3). Clearly (3) implies (4).
\end{proof}

\begin{example}\lab{exa:multiplicities}
Assume that $r\geq s\geq 1$ and that $r-s\equiv 2 \mod 3$.
Then the multiplicities of the weights of $V[0,1]$ and $V[3,1]$  in $V[r,s]$ are $\geq s$, and the multiplicities of the weights of $V[2,0]$ are $\geq s+1$ in case $r\geq 5$.
\par\smallskip
(In fact, for $V[3,1]$ the  multiplicities are $\geq \dim V[r-3,s-1]_{[0,0]}$ by Lemma~\ref{lem:multiplicities} and $\dim V[r-3,s-1]_{[0,0]}\geq \dim V[r-3,s-1]_{[r-s-2,0]} = s$ by Proposition~\ref{prop:degree3invars}(2). The other cases follow by similar arguments.)
\end{example}

\begin{proposition}\lab{prop:mainsl3}
Let $V=V[r,s]$ where $r+s\geq 4$ or $(r,s)=(2,1)$ or $(r,s)=(1,2)$. Then there is an irreducible component $\NN_1$ of $\NN$ such that the rank of $d\pi$ is less than the codimension of $\NN_1$ on $\NN_1$. In particular, $\NN$ is not reduced.
\end{proposition}

An immediate consequence is
\begin{theorem}\lab{thm:mainsl3.1}
Let $V$ be an irreducible representation of $G=\SL_{3}$. Then $V$ is coreduced if and only if $V$ is on the following list:
\begin{enumerate}
\item $V[1,0]$, $V[2,0]$, $V[3,0]$;
\item $V[0,1]$, $V[0,2]$, $V[0,3]$;
\item $V[1,1]$.
\end{enumerate}
Equivalently, $V$ is coreduced if and only if it is cofree.
\end{theorem}

\begin{proof}[Proof of Proposition \ref{prop:mainsl3}] 
We may assume that $V=V[r,s]$ where $r\geq s$ and $V[r,s]$ does not appear in (1), (2) or (3) of the theorem. Let $(p,q)= [r,s]$. We are constantly applying Remarks \ref{rem:method1} and \ref{rem:method2}.
\par\smallskip\noindent
{\bf Case 1:}  Assume that $r-s\equiv 1\mod 3$ and consider $\Lambda=\Lambda_\beta$ which is dominant by Lemma~\ref{lem:lambdabeta}. Recall that $\codim_{GZ_{\Lambda}}Z_{\Lambda} \leq 2$. First suppose that $s\geq 1$ and $r>2$.  
Then $[1,3]$ and $[0,2]$ are weights of $V$. Let $\lambda\in Y(T)$ correspond to $\Lambda_\beta$. Then $\lambda$ is negative on the weights $(2/3,-2/3)$ and $(-4/3,-2/3)$ in the $W$-orbit of $[0,2]$, on the weights $(-7/3,-2/3)$ and $(5/3,-2/3)$ in the $W$-orbit of $[1,3]$ and on the weight $(-1/3,-2/3)$ in the $W$-orbit of $[1,0]$ which occurs with multiplicity at least $s+1$ since $[r-s,0]$ has multiplicity $s+1$ by Proposition~\ref{prop:degree3invars}(2). These negative weights can be paired with at most quadratic expressions in the positive weights (just look at the coefficients of $\beta$). Now there are at most $s+1$ cubic invariants (and no quadratic invariants), hence $\NN$ is not reduced if $s\geq 1$, $r>2$.  

If $(r,s)=(2,1)$, then we have the negative weights $(2/3,-2/3)$, $(-4/3,-2/3)$ and $(-1/3,-2/3)$ (with multiplicity 2). There is only a one-dimensional space of degree 3 invariants, and so $\NN$ is not reduced. 

If $s=0$, then the cases to  consider are $V[4,0]$, $V[7,0]$, etc. If $r\geq 7$, then we have a dominant weight $[1,3]$ whose $W$-orbit contains $(-7/3,-2/3)$ and $(5/3,-2/3)$. We still have $(2/3,-2/3)$,   $(-4/3,-2/3)$ and $(-1/3,-2/3)$. Since there is at most one degree three invariant, $\NN$ is not reduced. 

We are left with the case  of $V[4,0]$. Here we have negative weights $(2/3,-2/3)$,   $(-4/3,-2/3)$ and $(-1/3,-2/3)$ as well as $(-1/3,-5/3)$ and $(-4/3,-5/3)$ in the $W$ orbit of $[2,1]$. Thus $\NN$ is not reduced since there are only two irreducible invariants of degree $\leq 6$ (the Poincar\'e series of $\OO(V)^G$ is $1+t^3+2t^6+\dots$).
\par\smallskip\noindent
{\bf Case 2:} Assume that $r-s\equiv 2\mod 3$. For the cases  $[r,s]=[3,1]$ or $[5,0]$ see Example~\ref{ex:v31revisited} below. So we may assume that  $\Lambda=\Lambda_\alpha$ is dominant. If $s\geq 1$ (and so $r\geq5$), then among the dominant weights we have $[3,1]$  with multiplicity at least $s$,  $[2,0]$ with multiplicity at least $s+1$ and $[0,1]$ with multiplicity at least $s$ (see Example~\ref{exa:multiplicities}). The   $W$-orbit  of $[3,1]$ contains the weights $(-2/3,5/3)$ and $(-2/3,-7/3)$ with negative $\alpha$-coefficient, the $W$-orbit of $[2,0]$ contains $(-2/3,2/3)$ and $(-2/3,-7/3)$ and the $W$-orbit of $[0,1]$ contains $(-2/3,-1/3)$. Since there is at most an $(s+1)$-dimensional space of degree three invariants, $\NN$ is not reduced. If $s=0$ (and so $r\geq 5$), then we have the weights $[3,1]$, $[2,0]$ and $[0,1]$ with multiplicity one, and $\NN$ is not reduced because $\dim S^3(V)^G\leq 1$.

\par\smallskip\noindent
{\bf Case 3:} If $r-s\equiv 0\mod 3$, then we are in the adjoint case and the claim follows from Proposition~\ref{prop:PSLn}.
\end{proof}

\begin{example}\lab{ex:v31revisited}
Let $V=V[3,1]$. Then from Example~\ref{ex:v31} we see that there are two dominant $\Lambda$, one corresponding to $\lambda(\alpha)=1$ and $2/5<\lambda(\beta)<1$ (choose $\lambda(\beta)=1/2$) and the other to $\lambda(\alpha)=1$ and $5/2<\lambda(\beta)<4$ (choose $\lambda(\beta)=3$). Neither $\Lambda_\alpha$ nor $\Lambda_\beta$ is dominant. Consider the case where $\lambda(\beta)=1/2$. Then the minimal positive weights (in terms of their $\lambda$-value) are $(1/3,-1/3)$ and $(-2/3,5/3)$, both having $\lambda$-value $1/6$. Now consider the covariants of type $V[1,0]$. The highest weight is $(2/3,1/3)$ where $\lambda(2/3,1/3)=5/6$. Thus the highest degree in which the covariant could occur in $S^*(V)$ and not vanish on $GZ_\lambda$ is 5. One gets the same bound in case $\lambda(\beta)=3$. The Poincar\'e series of the invariants is $1+t^3+\dots$ and for the $V[1,0]$ covariants it is $4t^5+44t^8+\dots$. Thus there are generating covariants in degree 8, which vanish on $\NN$, so that $\NN$ is not reduced.

If $V=V[5,0]$, then  the calculations of  Example~\ref{ex:v31} show that the dominant $\Lambda$ again correspond to $\lambda(\beta)=1/2$ or $3$. (The only new weights are (10/3,5/3), (-5/3,5/3) and (-5/3,-10/3) and they give rise to no new ratios.)\ Hence the highest degree in which the covariant $V[1,0]$ could occur in $S^*(V)$ and not vanish on $GZ_\lambda$ is again 5. The covariant $V[1,0]$ first occurs in degree 5, with multiplicity one. But since the principal isotropy group of $V$ is trivial, the $V[1,0]$ covariants have to have generators in higher degree, and these necessarily vanish on $\NN$. Thus $\NN$ is not reduced.
\end{example}

We now have the following result, which uses Theorem \ref{prop:mainsl3}.

\begin{theorem}\lab{thm:mainsl3.2}
Let $G=\SL_3$ and $V$ a nontrivial reducible $G$-module with $V^G=0$. Then, up to isomorphism and taking duals, we have the following list:
\begin{enumerate}
\item $kV[1,0]+\ell V[0,1]$, $k+\ell\geq 2$.
\item $V[2,0]+V[0,1]$.
\end{enumerate}
\end{theorem}
 
 \begin{proof}
 We already know that the representations in (1) and (2)  are coreduced by Theorem \ref{CIT.thm} and Example \ref{ex:quadraticform}. 
 We have to show that combinations not on the list are not coreduced.
 
 Consider $V[1,1]$ together with another irreducible. For $V[1,1]$  there is a slice representation of a group (with finite cover) $\SL_2\times \C^*$ and slice representation $\theta_1+R_2$. If we add $V[1,0]$ we get an additional copy of $R_1\otimes\nu_1+\nu_{-2}$ in the slice representation. Quotienting by $\C^*$ we get the hypersurface in $\theta_1+R_2+R_2$ where the quadratic invariant of the second copy of $R_2$ vanishes. The hypersurface is not coreduced, hence  $V[1,1]+V[1,0]$ is not coreduced. The other cases involving $V[1,1]$ are similarly not coreduced (for $V[1,1]+V[2,0]$ use the slice representation of the maximal torus). For $V[3,0]$ one similarly uses the slice representation at the zero weight vector to rule out a coreduced sum involving $V[3,0]$.
 
Next consider $2V[2,0]$ (cofree) and the 1-parameter subgroup $\lambda$ with weights $(1,1,-2)$. Then one can easily see that the codimensions of $GZ_\lambda$ and $\NN$ are both 4  and that the rank of the differentials of the  invariants is 
2 on $GZ_\lambda$. Hence the representation is not coreduced. For the (cofree) representation $V[2,0]+V[0,2]$ 
 and the same $\lambda$ one computes that the rank is 
 3 while the codimensions of $\NN$ and $GZ_\lambda$ are 4, so this possibility is ruled out. 
 We cannot add $V[1,0]$ to   $V[2,0]$ 
 (the rank of the two generating invariants is only one on the null cone).
 
 Finally, we have to show that $V:=V[2,0]+2V[0,1]$ (not cofree but coregular) is not coreduced. Consider the 1-parameter subgroup with weights $(2,-1,-1)$. This clearly gives a maximal dimensional component of the null cone and it has codimension $3=\dim\quot VG-1$. The 1-parameter subgroup with weights $(1,1,-2)$ gives something of codimension 5, which is too small to be an irreducible component of $\NN$ since it is cut out by 4 functions. Hence $\NN$ is irreducible. But $V$ has the  slice representation $2R_2+\theta_1$ of $\SL_2$ whose null cone is not reduced but also has codimension three. Thus the associated cone to a fiber $F:=G\times^{\SL_2}\NN(2R_2)$ is $\NN(V)$. But $F$ is not reduced, hence neither is $\NN(V)$ by the argument of Proposition \ref{prop:slice.zero.weight}.
\end{proof}

\setcounter{section}{0}
\renewcommand{\thesection}{Appendix\ \Alph{section}}
\bigskip
\section{Computations for \texorpdfstring{$\Ffour$}{F4}}\lab{AppA}
\renewcommand{\thesection}{\Alph{section}}
Let $G$ be a simple group of type $\Ffour$ and let $V=\phi_{4}$ be the 26-dimensional representation of $G$.
The main result of this Appendix is the following proposition.
\begin{proposition}\lab{Aprop:F4}
The representation $(V^{\oplus n},G)$ is coreduced if and only if $n\leq 2$. Moreover, $V$ and $V\oplus V$ are both cofree and contain a dense orbit in the null cone. 
\end{proposition}
We will use the notation introduced in section~\ref{sec:comp-nullcone}.
The nonzero weights of $V$ are the short roots of $\FF_4$. Hence $Z_\rho$ is the span of the positive short root spaces for any generic $\rho\in Y(T)$ which implies that the null cone  $\NN(V^{\oplus n})$ is irreducible for any $n$. We also know that  $V\oplus V$ is cofree with $\dim \quot {(V\oplus V)}{G} = 8$ \cite{Sc1978Representations-free}, hence $\dim \NN(V\oplus V) = 44$. Let us look at the following statements which imply the proposition.
{\it 
\be
\item[(a)] $V\oplus V \oplus V$ is not coreduced.
\item[(b)]  $V\oplus V$ is coreduced.
\item[(c)] There is a dense orbit in the null cone of $V \oplus V$.
\ee
}
Although we know that (c) implies (b) (Corollary~\ref{cor:cofree}) we will present direct proofs of all three claims. They are based on some explicit computations.

\begin{proof}[Proof of statement (a)] There is a maximal subgroup of type $\BB_{4}$ of $\Ffour$
where $(\phi_4(\FF_4),\BB_4)=\phi_1+\phi_4+\theta_1$. The slice representation of $\BB_4$ on $\phi_4(\FF_4)$ is $\phi_1(\BB_4)+\theta_1$. To prove that $3\phi_4(\FF_4)$ is not coreduced, it suffices to prove that $W:=3\phi_1(\BB_{4})+2\phi_4(\BB_{4})$ is not coreduced. Now $\DD_{4}$ is a maximal subgroup of $\BB_{4}$ and $V:=2\phi_{1}(\DD_{4})+2\phi_{3}(\DD_{4})+2\phi_{4}(\DD_{4})$ is a slice representation of $W$ at a zero weight vector. So we have to show that $V$ is not coreduced.
Since our representations are self-dual, we will deal with the symmetric algebra $\Sym(V)$ in place of $\OOO(V)$.

We have $\bigwedge^{2}\phi_{1}=\bigwedge^{2}\phi_{3}=\bigwedge^{2}\phi_{4}= \phi_{2}$, the adjoint representation. In the tensor product of three copies of $\phi_{2}$ we have 7 copies of $\phi_{2}$, but only five of them are in the ideal generated by the invariants. (This can be checked using LiE). We will show now that every covariant of type $\phi_{2}$ in $\bigwedge^{2}\phi_{1}\otimes \bigwedge^{2}\phi_{3}\otimes\bigwedge^{2}\phi_{4}\subset S(V)_{(2,2,2)}$ vanishes on the null cone, i.e., vanishes on $Z_{\rho}$ for every generic $\rho\in Y(T)$. 

Recall that the weights of $\phi_{1}$ are $\pm\eps_{i}$, those of $\phi_{3}$ are $1/2(\pm\eps_{1}\pm\eps_{2}\pm\eps_{3}\pm\eps_{4})$ where the number of minus signs is even. The weights of $\phi_{4}$ look similar, but have an odd number of minus signs. We use the 
notation $(\pm\pm\pm\pm)$ for these weights. 

There is an outer automorphism $\tau$ of $\DD_{4}$ of order 2 (coming from the Weyl group of $\BB_{4}$) which normalizes the maximal torus, fixes $\eps_{1}, \eps_{2}, \eps_{3}$ and sends $\eps_{4}$ to $-\eps_{4}$. If $G$ is of type $\DD_{4}$ and if $\rho_{i}\colon G\to \GL(V_{i})$ denotes the $i$th fundamental representation $\phi_{i}$, then $\rho_{1}\circ\tau \simeq\rho_{1}$, $\rho_{2}\circ\tau \simeq\rho_{2}$, and $\rho_{3}\circ\tau \simeq\rho_{4}$. Thus there is a linear automorphism $\mu\colon V\simto V$ which is $\tau$-equivariant, i.e., $\mu(gv) = \tau(g)\mu(v)$. It follows that $\mu$ has the following properties:
\be
\item $\mu$ sends $G$-orbits to $G$-orbits. In particular, $\mu(\NN) = \NN$.
\item $\mu(V_{\lambda}) = V_{\tau(\lambda)}$.
\item If $\psi\colon V\to U$ is a covariant of type $\phi_{2}$ in $\bigwedge^{2}\phi_{1}\otimes \bigwedge^{2}\phi_{3}\otimes\bigwedge^{2}\phi_{4}$, then so is $\psi\circ\mu\colon V\to U$.
\ee
This implies  that for every 1-PSG $\rho$ we have 
 $\mu(Z_{\rho})=Z_{\tau(\rho)}$ and 
 that if all $\psi\colon V\to U$ as in (3) vanish on $Z_{\rho}$, then they vanish on $Z_{\tau(\rho)}$, too. 

As a consequence, we can assume that $\eps_{1}>\eps_{2}>\eps_{3}>\eps_{4}>0$ and that $\eps_{1}-\eps_{2},\eps_{2}-\eps_{3},\eps_{3}-\eps_{4}>0$. This implies that the following weights are positive:
\begin{align*}
\{\eps_{1},\eps_{2},\eps_{3}, \eps_{4}\} &\subset \Lambda(\phi_{1}),\\
\{(++++), (+-+-), (++--)\} &\subset\Lambda(\phi_{3}), \\
\{(+++-), (++-+), (+-++)\} &\subset\Lambda(\phi_{4}).
\end{align*}
Since $(-++-) < (-+++)$ we see  that there are only three cases of maximal positive weight spaces to be considered.
\be
\item $\{\eps_{1},\eps_{2},\eps_{3}, \eps_{4}\}$, $\{(++++), (+-+-), (++--), (-++-)\}$ and 
\newline
$\{(+++-), (+-++), (++-+),(-+++)\}$;
\item $\{\eps_{1},\eps_{2},\eps_{3}, \eps_{4}\}$, $\{(++++), (+-+-), (++--), (+--+)\}$ and 
\newline
$\{(+++-), (+-++), (++-+),(+---)\}$;
\item $\{\eps_{1},\eps_{2},\eps_{3}, \eps_{4}\}$, $\{(++++), (+-+-), (++--), (+--+)\}$ and 
\newline
$\{(+++-), (+-++), (++-+),(-+++)\}$.
\ee
Now we have to calculate the positive weights in $\bigwedge^{2}\phi_{1}$, $\bigwedge^{2}\phi_{3}$ and $\bigwedge^{2}\phi_{4}$. For 
$\bigwedge^{2}\phi_{1}$ we get  $\{\eps_{i}+\eps_{j} \mid i<j\}$ in all three cases. For the two others we find the following sets.
\be
\item $\bigwedge^{2}\phi_{3}$: $\{\eps_{i}+\eps_{j}\mid i<j<4\} \cup \{\eps_{i}-\eps_{4}\mid i<4\}$; 
$\bigwedge^{2}\phi_{4}$: $\{\eps_{i}+\eps_{j} \mid i<j\}$.
\item $\bigwedge^{2}\phi_{3}$: $\{\eps_{1}\pm\eps_{j}\mid j>1\}$; 
$\bigwedge^{2}\phi_{4}$: $\{\eps_{1}\pm\eps_{j}\mid j>1\}$.
\item $\bigwedge^{2}\phi_{3}$: $\{\eps_{1}\pm\eps_{j}\mid j>1\}$; $\bigwedge^{2}\phi_{4}$: $\{\eps_{i}+\eps_{j} \mid i<j\}$.
\ee
Now it is easy to see that in all three cases there is no way to write the highest weight $\eps_{1}+\eps_{2}$ of $\phi_{2}$ as a sum of three positive weights, one from each $\bigwedge^{2}$.
\end{proof}
\begin{proof}[Proof of statement (b)]
Since $V \oplus V$ is cofree the null cone is (schematically) a complete intersection. Therefore it suffices to find a element $v\in V\oplus V$ such that the differential $d\pi_{v}$ of the quotient morphism $d\pi\colon V\oplus V \to Y$ at $v$ has maximal rank $8=\dim Y$.
 
The nonzero weights of $\phi_4(\FF_4)$ are $\pm\eps_i$, $i=1,\dots,4$ (the  nonzero weights of $\phi_1(\BB_4)$) and $(1/2)(\pm\eps_1\pm\eps_2\pm\eps_3\pm\eps_4)$ (the weights of $\phi_4(\BB_4)$). We will abbreviate the latter weights as $(\pm\pm\pm\pm)$ from now on. The positive weights are the $\eps_i$ and the weights of $\phi_4(\BB_4)$ where the coefficient of $\eps_1$ is positive. Let $v_{\pm i}$ denote a nonzero vector in the weight space of $\pm\eps_i$, let $v_0$ denote a zero weight vector and let $v_{++++}$ denote a nonzero vector in the weight space $(++++)$ and similarly for $v_{+++-}$, etc. 
We claim that $d\pi$ has rank 8 at the point $v=(v_{2}+v_{3}+v_{+--+}+v_{+---})\in 2\phi_1(\BB_4)+2\phi_4(\BB_4)$.

The invariants of $2\phi_4(\FF_4)$ are the polarizations of the degree 2 invariant and the degree 3 invariant of one copy of $\phi_4(\FF_4)$ together with an invariant of degree $(2,2)$. The restriction of the degree 2 invariant to $\phi_1(\BB_4)+\phi_4(\BB_4)$ is the sum of the degree two invariants there (see \cite{Sc1978Representations-reg} for descriptions of the invariants of $(2\phi_1+2\phi_4,\BB_4)$.). Clearly the differentials of the degree 2 invariants of $2\phi_4(\FF_4)$ at $v$ have rank 3 when applied to the subspace spanned by the vectors $v_{-2}$, $v_{-3}$ in the two copies of $\phi_1(\BB_4)$. There is only one degree three generator in $(\phi_1+\phi_4,\BB_4)$ 
and it is the contraction of $\phi_1$ with the copy of $\phi_1$ in  $\Sym^2(\phi_4)$. Another way to think of the invariant is as the contraction of $\phi_4$ with the copy of $\phi_4$ in $\phi_1\otimes\phi_4$.  Now the highest weight vector of the copy of $\phi_4$ in $\phi_1\otimes\phi_4$ is (up to some nonzero coefficients)
$$
v_1\otimes v_{-+++}+v_2\otimes v_{+-++}+v_3\otimes v_{++-+}+v_4\otimes v_{+++-}+v_0\otimes v_{++++}.
$$
From this one derives the form of the other weight vectors of $\phi_4\subset\phi_1\otimes\phi_4$ and restricting to $v$ one gets contributions to the weights  $(++-+)$, $(++--)$, $(+-++)$ and $(+-+-)$. Thus the differential of the degree 3 invariant of $\phi_1(\FF_4)$ at $v$ vanishes  on $\phi_4$ except on $v_{--+-}$, $v_{--++}$, $v_{-+--}$ and $v_{-+-+}$. Now polarizing it is easy to see that the four generators of degree 3 have differential of rank 4 at $v$ when applied to vectors in $2\phi_4(\BB_4)$. 

There remains the generator of degree 4. Restricted to $\BB_4$ one easily sees that the invariant is a sum of two generators (modulo products of the generators of degree 2), one of which is the invariant   which contracts the copy of $\bigwedge^2(\phi_1)\subset \Sym^2(2\phi_1)$ with the copy in $\bigwedge^2(\phi_4)\subset\Sym^2(2\phi_4)$ and the other which is of degree 4 in $2\phi_4(\BB_4)$ (and doesn't involve $2\phi_1$). Now the highest weight vector of $\bigwedge^2(\phi_1)\subset\bigwedge^2(\phi_4)$ is (up to nonzero scalars)
$$
v_{++++}\wedge v_{++--}+v_{+++-}\wedge v_{++-+}
$$
from which it follows that the weight vector of weight $-\eps_2-\eps_3$ does not vanish on $v_{+---}+v_{---+}$. Now in $\bigwedge^2(\phi_1)\subset\Sym^2(2\phi_1)$ we have $v_2\wedge v_3$ of weight $\eps_2+\eps_3$. Hence the differential of the degree 4 invariant evaluated at $v$ does not vanish on $v_{---+}$ and   the rank of the differentials of the 8 invariants is indeed 8. Thus  $2\phi_4(\FF_4)$ is coreduced.
\end{proof}

\begin{proof}[Proof of statement (c)]
Recall that the root system $R$ of $G$ has the following 3 parts $A$, $B$ and $C$:
$$
A = \{\pm\eps_{i}\}, \quad B=\{\pm\eps_{i}\pm\eps_{j}\mid i<j\}, \quad C=\{\frac{1}{2}(\pm\eps_{1}\pm\eps_{2}\pm\eps_{3}\pm\eps_{4})\}
$$
with cardinality $\#A = 8$, $\#B = 24$ and $\#C = 16$. Thus 
$$
\lieg=\Lie G = \hh \oplus \bigoplus_{\delta\in A\cup B\cup C}\lieg_{\delta}
$$ 
where $\hh = \Lie T$ is the Cartan subalgebra and $T \subset G$ a maximal torus.
The  weights $\Lambda=\Lambda_{V}$ of the representation $V$ are the short roots $A\cup C$ together with the zero weight $0$. 
The non-zero weight spaces $V_{\rho}$ are 1-dimensional, and the zero weight space $V_{0}=V^{T}$ has dimension 2. This implies the following.
\begin{lemma}
Let $\delta\in R$ be a root and $\rho\in\Lambda$ a weight of $V$. If $\delta+\rho$ is a weight of $V$, then $\lieg_{\delta}V_{\rho}$ is a non-trivial subspace of $V_{\delta+\rho}$.
\end{lemma}
Note that $\lieg_{\delta}V_{0}=V_{\delta}$ and $\lieg_{\delta}V_{-\delta} \subset V_{0}$ is  1-dimensional for every short root $\delta\in A\cup C$.

The subspace $\lieg':=\hh\oplus\bigoplus_{\delta\in A\cup B}\lieg_\delta\subset \lieg$ is the Lie algebra of a maximal subgroup $G' \subset G$ of type $\Bfour$, and the representation $V$ decomposes under $G'$ into $V = \theta_{1}\oplus \phi_{1}(\Bfour)\oplus\phi_{4}(\Bfour)$ where $\phi_{4}(\Bfour)=\bigoplus_{\gamma\in C} V_{\gamma}$, $\phi_{1}(\Bfour) = V_{A}\oplus\bigoplus_{\alpha\in A}V_{\alpha}$, and $V_{A}:= \phi_{1}(\Bfour)^{T}\subset V_{0}$. It follows that   $\lieg_{\alpha}V_{-\alpha}=V_{A}$ for  $\alpha\in A$, but  $\lieg_{\gamma}V_{-\gamma}\nsubseteq V_{A}$ for $\gamma\in C$, so that  $\lieg_{\alpha}V_{-\alpha} + \lieg_{\gamma}V_{-\gamma} = V_{0}$.

The basic idea for the calculations  is the following. To every vector $v\in V$ we define its {\it weight support} $\omega(v) \subset A \cup C \cup \{0_{A},0_{C}\}$ in the following way. Write $v$ as a sum of weight vectors, $v = \sum_{A} v_{\alpha} + \sum_C v_{\gamma} + v_{0}$. Then 
$$
\omega(v) := \{\alpha\in A\mid v_{\alpha} \neq 0\}\cup \{\gamma \in C \mid v_{\gamma}\neq 0\} \cup 
\begin{cases} \emptyset & \text{if } v_{0}=0,\\
\{0_{A}\}& \text{if } v_{0}\in V_{A}\setminus\{0\},\\
\{0_{C}\}& \text{if } v_{0}\in V_0\setminus V_{A}.\\
\end{cases}
$$
This extends in an obvious way to the weight support of elements from $V \oplus V$. Now we look at a pair $v=(v',v'')=(v_{\alpha'}+v_{\gamma'}, v_{\alpha''}+v_{\gamma''})\in V\oplus V$ where $\alpha',\alpha''\in A$ and $\gamma',\gamma''\in C$ are distinct weights. Define 
$$
\Omega(v) := \{\omega(x_{\delta} v)\mid \delta \in A\cup B \cup C\} \cup \{\alpha',\alpha'',\gamma',\gamma''\}
$$
where $x_{\delta}\in\lieg_{\delta}$ is a (non-zero) root vector. This is the set of weight supports of  generators of $\lieg v$ where we use that $\lieh v = \C v_{\alpha'} \oplus \C v_{\gamma'} \oplus \C v_{\alpha''} \oplus \C v_{\gamma''}$. 

Our problem can now be understood in the following way. We are given a matrix $M$ of column vectors from which we want to calculate the rank. We replace $M$ by the ``support matrix'' $\Omega(M)$ which is obtained from $M$ by replacing each non-zero entry by 1. How can one find a lower bound for the rank of $M$ from $\Omega(M)$? 

There is an obvious procedure. We first look for a column of $\Omega(M)$ which contains a single 1, let us say in row $i$. Then we remove all other $1$'s in row $i$ and repeat this procedure as often as possible to obtain a matrix $\Omega(M)'$. It is clear that this ``reduced'' matrix $\Omega(M)'$ is again the support matrix of a matrix $M'$ which is obtained by column reduction from $M$. This first step is called ``column reduction.''

Now we apply row reduction to $M'$ which amounts to looking at rows of $\Omega(M)'$ which contain a single 1. Then we delete all other $1$'s in the corresponding column. Again we repeat this procedure as often as possible and obtain a matrix $\Omega(M)''$. We call this procedure ``row reduction.'' It is clear now that a lower bound for the rank of $M$ is given by the number of columns of $\Omega(M)''$ containing a single 1.

Now we choose $v\in V\oplus V$ as above with weights 
$$
(\alpha',\gamma',\alpha'',\gamma'') = (\eps_{3},1/2(\eps_{1}-\eps_{2}-\eps_{3}+\eps_{4}),\eps_{2}, 1/2(\eps_{1}-\eps_{2}-\eps_{3}-\eps_{4})).
$$
We obtain a set $\Omega(v)$ with 45 elements where each element's  weight support has cardinality at most two. (We use Mathematica to perform these and the following calculations.) After applying the ``column reduction'' we obtain a new set $\Omega(v)'$ which contains 44 elements where 34 of them contain a single weight. For the remaining 10 elements, the ``row'' reduction produces 10 sets with a single weight. Thus we get $\dim Gv = \dim \lieg v = 44 = 2\dim V - \dim \quot VG$,  and we are done.
\end{proof}

\begin{remark} 
We are grateful to Jan Draisma who did some independent calculations (using GAP) to show that there is a dense orbit in the null cone of  $V \oplus V$.
\end{remark}

\bigskip
\renewcommand{\thesection}{Appendix\ \Alph{section}}
\section{Computations for \texorpdfstring{$\Gtwo\times\Gtwo$}{G2xG2}}\lab{AppB}
\renewcommand{\thesection}{\Alph{section}}
The main result of this Appendix is the following proposition. We will give two proofs.

\begin{proposition}\lab{propA:g2g2}
The representation  $(\C^7\otimes\C^7, \Gtwo\times\Gtwo)$ is not coreduced. 
\end{proposition}

\begin{proof}[First Proof of Proposition~\ref{propA:g2g2}]
The nontrivial part of the slice representation at the zero weight vector is $G:=\SL_3\times\overline{\SL}_3$ on the  four possible versions of $(W:=\C^3$ or $W^*)$ tensored with $(\overline W:=\overline{\C^3}$ or $\overline{W}^*)$. 
\begin{lemma}\lab{lem:g2g2}
The $G$-module
$$
V:=W\otimes\overline W+ W\otimes\overline W^*+ W^*\otimes\overline W+W^*\otimes\overline W^*
$$
is not coreduced.
\end{lemma}

We have a group $N$ of order 8 which acts on $V$ by interchanging $W$ and $W^*$, $\overline W$ and $\overline W^*$ as well as interchanging $W$, $W^*$ with $\overline W$, $\overline W^*$. Then $N$ normalizes the action of $G$.
Here are the steps in the proof of the proposition above.
\begin{enumerate}
\item We show that there is a minimal generator $f$ of the invariants of $(V,G)$ which is multihomogeneous of degree (3,3,3,3) in the four irreducible subspaces  of $V$.
\item We show that, up to the action of the Weyl group and $N$, there are  eight $1$-parameter subgroups $\rho$ of $G$
such that the union of the $GZ_\rho$ is $\NN(V)$.
\item We show that for each such $\rho$, the differential of $f$ vanishes on $Z_\rho$.
\end{enumerate}
It then follows from Remark \ref{rem:gendifferential} that $V$ is not coreduced.

Let $R=\C[a_1,\dots,a_n]$ be a finitely generated $\N^d$-graded   ring where the $a_i$ are homogeneous. Recall that $a_1,\dots,a_m$ are a \emph{regular sequence\/} in $R$ if $a_1$ is not a zero divisor and  $a_{j+1}$ is not a zero divisor in $R/(Ra_1+\dots+Ra_j)$, $1\leq j<m$. We may write $R$ as a quotient $R=\C[x_1,\dots,x_n]/I$ where the image of $x_i$ in $R$ is $a_i$, $i=1,\dots,n$.  Let $I_s$ denote the elements of $I$ homogeneous of degree $s:=(s_1,\dots,s_d)$, $s_j\in\N$ and let $\bar I$ denote $I/(x_1,\dots,x_m)$.
We leave the proof of the following lemma to the reader.

\begin{lemma}
Let $R$, etc.\ be as above. Then $I_s\to \bar I_s$ is an isomorphism for all $s\in\N^d$.
\end{lemma}
The lemma says that we can determine the dimension of the space of relations of the $a_i$ in degree $s$ by first setting $a_1,\dots,a_m$ to zero.

\begin{lemma}\lab{lem:3333generator}
There is a generator $f$ of $R:=\OO(V)^G$ of multidegree $(3,3,3,3)$.
\end{lemma}

We used LiE to compute a partial Poincar\'e series of $R$:
\begin{multline*}
1+(ps+qr)+2(p^2s^2+q^2r^2)+(p^3+q^3+r^3+s^3)\\
+2(p^3q^3+p^3r^3+q^3s^3+r^3s^3)+3(p^3s^3+q^3r^3)\\
+2(q^3ps+r^3ps+p^3qr+s^3qr)\\
+6(q^3p^2s^2+r^3p^2s^2+p^3q^2r^2+s^3q^2r^2)\\
+13(p^3q^3r^3+p^3q^3s^3+p^3r^3s^3+q^3r^3s^3)\\
+4pqrs+10(pq^2r^2s+p^2qrs^2)+18(pq^3r^3s+p^3qrs^3)\\
+37p^2q^2r^2s^2
+86(p^2q^3r^3s^2+p^3q^2r^2s^3)+265p^3q^3r^3s^3.
\end{multline*}
If there were no relations among the generators of $R$ of degree at most $(3,3,3,3)$, then the Poincar\'e series would indicate that we have generators in degree $(a,b,c,d)$ of a certain multiplicity which we denote by $\gen(a,b,c,d)$. We list the relevant $\gen(a,b,c,d)$, modulo symmetries. 
\begin{enumerate}
\item $\gen(0110)=1$
\item $\gen(0220)=1$
\item $\gen(3000)=1$
\item $\gen(3300)=1$
\item $\gen(0330)=0$
\item $\gen(1111)=3$
\item $\gen(3110)=1$
\item $\gen(3220)=3$
\item $\gen(3330)=3$
\item $\gen(1221)=5$
\item $\gen(1331)=2$
\item $\gen(2222)=14$
\item $\gen(3223)=13$
\item $\gen(3333)=11$
\end{enumerate}
It is easy to see that the representations $V_i+V_j$, $1\leq i<j\leq 4$ are cofree.  Now $V_2+V_3$ has generators in degrees (3,0), (0,3), (1,1) and (2,2) while $V_1+V_2$ has generators in degrees (3,0), (0,3) and (3,3). Thus it is easy to see that we have a regular sequence in  $R$ consisting of the (determinant)  invariants of degree 3 and those of  degree 
$(0110)$, $(0220)$, $(1001)$ and $(2002)$.

\begin{lemma}
Suppose that we are  in one of the cases above, except for  $(2222)$, $(3223)$ and $(3333)$. Then $R$ has $\gen(abcd)$ generators in degree $(abcd)$.
\end{lemma}

\begin{proof}
We   set the  invariants of our regular sequence equal to zero and see if we have any relations. But then there are no  nonlinear polynomials in the remaining generators  in the degrees we are worried about.
\end{proof}

 \begin{proof}[Proof of Lemma~\ref{lem:3333generator}]
As usual, we set the elements of our regular sequence equal to zero. This does not change the number of minimal generators of degree $(3333)$. Now how can we have fewer generators than $\gen(3333)$ in degree $(3333)$? This can only occur if there is a degree $(abcd)$ with an ``unexpected''  relation such that $(3333)-(abcd)$ is the degree of a generator not in our regular sequence. Thus the only problem could occur because of relations in degree $(2222)$ multiplied by the 3 generators $f_1$, $f_2$ and $f_3$ in degree $(1111)$. Moreover, modulo our regular sequence, the unexpected relations in degree $(2222)$ have the form $(r=\sum_{ij} c_{ij}f_if_j)=0$. Thus there are  unexpected relations $r_1,\dots r_d$, $d\leq 6$. For each relation $r_k$ we add an additional generator $y_k$ in degree $(2222)$ and to get the correct count of non-generators in degree $(3333)$ we have to adjust our formal count by adding $3d$ (from the product of  the $y_k$ by the $f_i$)   and subtracting  the dimension of the span of the  $f_ir_k$ in the polynomial ring $\C[f_1,f_2,f_3]$. But the correction is by less than 11:

\smallskip
Case 1. $d\leq 5$. Then  we have a correction of at most $3d-d\leq 10$.

\smallskip
Case 2. $d=6$. Then the correction is $18- \dim \Sym^3(\C^3)=8$.
\end{proof}

We now have our generator $f$ of degree (3,3,3,3). Next we need to calculate the irreducible components of the null cone, up to the action of $N$. 

Let $\rho$ be a 1-parameter subgroup of $G:=\SL_3\times\overline{\SL}_3$ whose weights for $\C^3$ are $a$, $b$ and $c$ and whose weights for $\overline \C^3$ are $\overline a$, $\overline b$ and $\overline c$. We have that $a\geq b\geq c$ and similarly for $\overline a$, etc. We also can assume that no weight of $V$ is zero.   Of course, many choices of $a$, etc. will give the same subset $Z_\rho$ in $V$. We say that a particular choice of $a$, etc. is a \emph{model\/}  if it gives the correct $Z_\rho$. 
 
The action of our group $N$ does not change the weights that occur,  just in which of the four components they occur. Thus to show that  $df$   vanishes on $\NN(V)$, we can always reduce to the case that $a>\overline a$   and that the other numbers are negative (or zero). For every possibility we will give a model such that $df$ vanishes  on $Z_\rho$. 

\begin{lemma}
We have that $c-\overline b\leq c-\overline c\leq c+\overline a\leq b+\overline a \geq b-\overline c\geq b-\overline b$. Moreover, $c-\overline b<0$ and not both $c-\overline c$ and $b-\overline b$ are positive.
\end{lemma}

\begin{proof}
The string of inequalities is obvious. If $c-\overline b>0$, then $b-\overline c>0$ and adding we get that $-a+\overline a>0$ which is a contradiction. Similarly, not both $c-\overline c$ and $b-\overline b$ can be positive. 
\end{proof}

Given the lemma, there  are eight possibilities for the signs of $c-\overline b$, $c-\overline c,\dots,b-\overline b$ which we present in matrix form. In another matrix, we present the values $a$, $b$, $c$, $\overline a$, $\overline b$ and $\overline c$ of a 1-parameter subgroup $\rho$ which is a model for the signs. Note that the signs  tell you exactly which vectors in $V$ are in the positive weight space of $\rho$.
$$
\left(\begin{matrix} -1 &-1&-1&-1&-1&-1\\
                                 -1 &-1&-1&\phm1&-1&-1\\
                                 -1 &-1&-1&\phm1&\phm1&-1\\
                                 -1 &-1&-1&\phm1&\phm1&\phm1\\
                                 -1 &-1&\phm1&\phm1&-1&-1\\
                                 -1 &-1&\phm1&\phm1&\phm1&-1\\
                                 -1 &-1&\phm1&\phm1&\phm1&\phm1\\
                                 -1 &\phm1&\phm1&\phm1&\phm1&-1 
 \end{matrix}\right)\quad
 \left(\begin{matrix}  4 & -2 & -2 &1 & \phm0 &-1\\
                                   8 & -3 & -5 & 4 &-2& -2\\
                                   4 & -1 & -3 & 2 & \phm0 & -2\\
                                   3 & \phm0 & -3 & 2 & -1 & -1\\
                                   6 & -3 & -3 & 4 & -2 & -2\\
                                   8 & -3 & -5 & 6 & -2 & -4\\
                                   7 & -2 & -5 & 6 & -3 & -3\\
                                   4 & -2 & -2 & 3 & \phm0 & -3
 \end{matrix}\right)
$$
\begin{proof}[Proof of Lemma~\ref{lem:g2g2}] Consider  signs which have the 1-parameter subgroup $\rho$ with weights $\begin{pmatrix}8 & -3 & -5 & 6 & -2 & -4\end{pmatrix}$ as model. Then the largest negative weight occurring in $V$ is $-14$ while the positive weights occurring in $V_1,\dots,V_4$ are
$$
(1,3,4,6,14),\ (1,2,10,12),\ (1,1,3,9,11),\ (7,7,5,9).
$$
Now consider a monomial $m$ in the weight vectors which occurs in $f$. If $df$ does not vanish on $Z_\rho$
then there is a monomial with only one negative weight vector. But the sum of the positive weights occurring in $m$ is at least $3+3+3+2*5=19$ which is greater than $14$. Hence this is impossible and $df$ vanishes on $GZ_\rho$. One similarly (and more easily) sees that $df$ vanishes on $GZ_\rho$ in the other 7 cases.
\end{proof}
This finishes the first proof of Proposition~\ref{propA:g2g2}.
\end{proof}

\begin{proof}[Second Proof of Proposition~\ref{propA:g2g2}]
The weights of $V$ are the short roots of $G$ together with $0$, and all weight spaces are 1-dimensional. We use the notation $\Lambda:=\{\pm\alpha,\pm\beta,\pm(\alpha+\beta),0\}$ where $\alpha+\beta$ is the highest weight. Thus the weight spaces of $V\otimes V$ are given by the tensor products $V_{\mu}\otimes V_{\nu}$, $(\mu,\nu)\in \Lambda\times \Lambda$.

We first determine the maximal positive subspaces of $W:=V\otimes V$, up to the action of the Weyl group. If $\rho$ is a  one-parameter subgroup of $G\times G$ we denote by $W_{\rho}$ the sum of the $\rho$-positive weight spaces, i.e.
$$
V_{\rho} := \bigoplus_{\rho(\mu,\nu)>0} W_{(\mu,\nu)}.
$$ 
The 1-PSG $\rho$   is defined by the  values $a:=(\rho,(\alpha,0))$, $b:=(\rho,(\beta,0))$, $a':=(\rho,(0,\alpha))$, $b':=(\rho,(0,\beta))$. Using the action of the Weyl group, we can assume that that 
$$
a,b,a',b' >0,\quad a\geq b,  \quad a'\geq b', \quad \{a,b,a+b\}\cap\{a',b',a'+b'\}=\emptyset.
$$
We can also assume that $a>a'$; we will then get the other maximal positive subspaces by the symmetry $(\mu,\nu)\mapsto (\nu,\mu)$.
Now  $V_{\rho}$ depends only on the relative position of the values $a+b > a \geq b$ and the values $a'+b' > a' \geq b'$. It is not difficult to see that there are eight cases.
\be
\item $a+b>a>b>a'+b'>a'\geq b'$ represented by $\rho=(5,4,2,1)$;
\item $a+b>a>a'+b'>b>a'\geq b'$ represented by $\rho=(6,4,3,2)$;
\item $a+b>a'+b'>a\geq b>a'\geq b'$ represented by $\rho=(6,5,4,3)$;
\item $a+b>a>a'+b'>a'>b>b'$ represented by $\rho=(5,2,3,1)$;
\item $a+b>a'+b'>a>a'>b>b'$ represented by $\rho=(6,4,5,3)$; 
\item $a+b> a >a'+b' \geq a'>b'\geq b$ represented by $\rho=(7,1,4,2)$;
\item $a+b>a'+b'>a>a' > b > b'$ represented by $\rho=(6,2,4,3)$;
\item $a'+b'>a+b>a>a' \geq b' >b$ represented by $\rho=(6,2,5,4)$.
\ee
To get the full set of maximal positive subspaces we have to add the 8 $\rho$'s obtained from the list above by replacing $(a,b,a',b')$ with $(a',b',a,b)$.

Now we used LiE to look at the covariants of type $\theta_{1}\otimes V$. The multiplicities of this covariant in degrees 1 to 9 are  $(0,0,1,1,3,5,12,18,41)$, and the dimensions of the invariants in these degrees are $(0,1,1,3,2,8,\newline 7,17,19)$. It follows that at most $37 = 1\cdot12 + 1\cdot5 + 3\cdot3 + 2\cdot1 + 8\cdot1$ covariants of degree 9 are in the ideal generated by the invariants, hence there are generating covariants of this type in degree $9$. Now we have to show that for every positive weight  space $V_{\rho}$  the highest weight $(0,\alpha+\beta)$ of $\theta\otimes V$ cannot be expressed as a sum of 9  weights from $V_{\rho}$. Because of duality, each $V_{\rho}$ has dimension $24 = (7*7 - 1)/2$. If we denote by $\Lambda_{\rho}$ the set of weights of $V_{\rho}$, this amounts to prove that the system
$$
\sum_{\lambda\in\Lambda_{\rho}} x_{\lambda}\lambda = (0,\alpha+\beta),\quad \sum_{\lambda\in\Lambda_{\rho}} x_{\lambda}
=9
$$ 
has no solution in non-negative integers $x_{\lambda}$. Note that the first condition consists in 4 linear equations in 24 variables. Now we used Mathematica to show that there are no solutions for each one of the sixteen  maximal positive weight spaces $Z_{\rho}$ given above.
\end{proof}

\bigskip

\bigskip

\end{document}